\documentclass[12pt]{amsart}
\usepackage{enumerate,amssymb,amsfonts,latexsym,psfrag}
\usepackage{amscd,amsmath}
\usepackage[dvips]{graphicx,epsfig}

\usepackage{graphics}
\usepackage[all]{xy}
\usepackage{url}
\usepackage{fullpage}
\usepackage{subfig}
\usepackage{multirow,array}
\usepackage{color}
\usepackage{cancel}
\usepackage{color}

\usepackage{titlesec}

\titleformat{\section}
  {\Large\bfseries\center}{\thesection.}{0.5em}{}
  \titlespacing*{\section} {0pt}{1em}{2.3ex plus .2ex}
\titlelabel{\thetitle.\quad}
\titleformat{\subsection}
  [runin]{\normalfont\bfseries}{\thesubsection.}{0.5em}{}[.]
\titlespacing*{\subsection}{0pt}{1em}{2.3ex plus .2ex}

\usepackage{titletoc}
\titlecontents{section}[1.5em]
  {\normalfont\bfseries}
  {\contentslabel{1.5em}}
  {\hspace*{-2.3em}}
 {\titlerule*[1pc]{}\contentspage}

\newtheorem{thm}{Theorem}[section]
\newtheorem{cor}[thm]{Corollary}
\newtheorem{lem}[thm]{Lemma}
\newtheorem{prop}[thm]{Proposition}

\theoremstyle{remark}
\newtheorem{rem}{Remark}[section]

\theoremstyle{definition}
\newtheorem{defn}{Definition}[section]
\newtheorem{example}[thm]{Example}

\numberwithin{equation}{section}
\numberwithin{figure}{section}

\font\nt=cmr7

\def\note#1
{\marginpar
{\nt $\leftarrow$
\par
\hfuzz=20pt \hbadness=9000 \hyphenpenalty=-100 \exhyphenpenalty=-100
\pretolerance=-1 \tolerance=9999 \doublehyphendemerits=-100000
\finalhyphendemerits=-100000 \baselineskip=6pt
#1}\hfuzz=1pt}

\newcommand{\di}{\partial}

\newcommand{\ra}{\rightarrow}

\def\lra{\longrightarrow}

\def\ssk{\smallskip}
\def\msk{\medskip}
\def\bsk{\bigskip}

\def\nin{\noindent}

\newcommand{\diam}{\operatorname{diam}}
\newcommand{\dist}{\operatorname{dist}}

\newcommand{\id}{\operatorname{id}}

\newcommand{\const}{\mathrm{const}}

\newcommand{\eps}{{\varepsilon}}

\newcommand{\de}{{\delta}}
\newcommand{\la}{{\lambda}}
\newcommand{\La}{{\Lambda}}
\newcommand{\si}{{\sigma}}

\newcommand{\II}{{\mathcal I}}

\newcommand{\NN}{{\mathcal N}}
\newcommand{\OO}{{\mathcal O}}

\newcommand{\TT}{{\mathcal T}}

\newcommand{\N}{{\mathbb N}}

\newcommand{\R}{{\mathbb R}}

\newcommand{\vv}{{\mathbf v}}
\newcommand{\cc}{{\mathbf c}}

\def\BPhi{{\boldsymbol{\BPhi}}}
\def\B0{{\mathbf{0}}}

\newcommand{\Jac}{\operatorname{Jac}}

\newcommand{\Dom}{\operatorname{Dom}}

\catcode`\@=12

\def\Empty{}
\newcommand\oplabel[1]{
  \def\OpArg{#1} \ifx \OpArg\Empty {} \else
  	\label{#1}
  \fi}
		
\newcommand{\comm}[1]{}
\newcommand{\comment}[1]{}

\usepackage[titletoc,title,header]{appendix}

\allowdisplaybreaks
\newcommand\numberthis{\addtocounter{equation}{1}\tag{\theequation}}

\begin{document}

\bigskip\bigskip

\title[H\'enon renormalization]{Invariant space under H\'enon renormalization : Intrinsic geometry of Cantor attractor}

\address {Hong-ik University }
\date{August 13, 2014}
\author{Young Woo Nam}
\thanks{
Email $ \colon $ namyoungwoo\,@\,hongik.ac.kr}


\begin{abstract}
Three dimensional H\'non-like map 
$$ F(x,y,z) = (f(x) - \eps(x,y,z),\ x,\ \de(x,y,z)) $$
is defined on the cubic box $ B $. An invariant space under renormalization would appear only in higher dimension. Consider renormalizable maps each of which satisfies the condition
$$ \di_y \de \circ F(x,y,z) + \di_z \de \circ F(x,y,z) \cdot \di_x \de(x,y,z) \equiv 0 $$
for $ (x,y,z) \in B $. Denote the set of maps satisfying the above condition be $ \NN $. Then the set $ \NN \cap \II(\bar \eps) $ is invariant under the renormalization operator where $ \II(\bar \eps) $ is the set of infinitely renormalizable maps. H\'enon like diffeomorphism in $ \NN \cap \II(\bar \eps) $ has universal numbers, $ b_2 \asymp | \di_z \de | $ and $ b_1 = b_F /b_2 $ where $ b_F $ is the average Jacobian of $ F $. The Cantor attractor of $ F \in \NN \cap \II(\bar \eps) $, $ \OO_F $ has {\em unbounded geometry} almost everywhere in the parameter space of $ b_1 $. If two maps in $ \NN $ has different universal numbers $ b_1 $ and $ \widetilde b_1 $, then the homeomorphism between two Cantor attractor is at most H\"older continuous, which is called {\em non rigidity}.
\end{abstract}

\maketitle

\thispagestyle{empty}

\setcounter{tocdepth}{1}
\tableofcontents

\renewcommand{\labelenumi} {\rm {(}\arabic{enumi}{)}}

\renewcommand{\CancelColor}{ \color{blue} }

\section{Introduction}
H\'enon renormalziation with universal limit was introduced in \cite{CLM}. H\'enon-like map is 
$$ F(x,y) = (f(x) - \eps(x,y),\ x) $$
two dimensional. The interesting maps in \cite{CLM} are H\'enon-like with strong hyperbolic fixed points and its period doubling renormalization. In higher dimension, we are interested in the case that the invariant set has the only one repulsive or neutral direction and other directions is strongly contracting. H\'enon renormalization was extended to three dimensional H\'enon-like maps from a cubic box, $ B $ to itself  
$$ F(x,y,z) = (f(x) - \eps(x,y,z),\ x,\ \de(x,y,z)) $$ 
in \cite{Nam1}. 
%
%
In three dimension, the geometric properties of Cantor attractor were studied for the map in invariant spaces under renormalization operator rather than in the set of all infinitely renormalizable maps. The first invariant space is the set of toy model maps. In \cite{Nam1}, H\'enon-like maps in this set has embedded invariant surfaces with the assumption of strong contraction along $ z- $axis. 
%
\ssk \\
However, in this paper we discover another invariant space in which any H\'enon-like map does not  require any invariant surfaces. Instead of this, we see the formula between derivatives of the third coordinate map of $ F $, $ D\de $ and $ D\de_1 $ of $ RF $ in Lemma \ref{appendix - Dde recursive form}. Since renormalized map, $ RF $ is determined by $ F $, the third coordinate map $ \de_1 $ of $ RF $ is so. In particular, $ f $, $ \eps $ and $ \de $ of $ F $ affects $ \de_1 $ in general. However, if $ F $ satisfies the following equation, 
\ssk
\begin{align*} 
\di_y \de \circ (F(x,y,z)) + \di_z \de \circ (F(x,y,z)) \cdot \di_x \de(x,y,z) \equiv 0  \numberthis  \label{eq-equation for NN}
\\[-1.2em]
\end{align*}
then only $ D\de $ affects the next third coordinate map, $ D\de_1 $. Let $ \NN $ be the set of renormalizable H\'enon-like maps each of which satisfies the above equation \eqref{eq-equation for NN}. 
Then the set $ \II(\bar \eps) \cap \NN $ is invariant under renormalization operator (Theorem \ref{Invariance of the space-NN}). Moreover, there exist two universal numbers, say $ b_1 $ and $ b_2 $ and $ b_1b_2 = b_F $ by Proposition \ref{b-2 as the asymptotic of di-z de} and Lemma \ref{Another universal number b_1}. Each of three dimensional H\'enon-like maps in this space has its Cantor attractor which has {\em unbounded geometry} and {\em non-rigidity} (Theorem \ref{unbounded geometry with b-1 a.e.} and Theorem \ref{Non rigidity with b-1}). It is worth to emphasize that these two geometric properties of Cantor attractor only depend on $ b_1 $ which is from the two dimensional H\'enon-like map in three dimension.
 
\msk

\bsk
\section{Preliminaries}
Let three dimensional {\em H\'enon-like map} $ F $ be the map as follows
\begin{equation*}
F(x,\,y,\,z) = (f(x) - \eps(x,y,z),\ x,\ \de(x,y,z))
\end{equation*}
where $ f $ is the unimodal map and $ \eps $ and $ \de $ from $ \Dom(F) $ to $ \R $ are maps with small norms, that is, $ \| \:\!\eps\|, \| \:\!\de \| \leq \bar \eps $ for small enough $ \bar \eps > 0 $. Then the image of the plane $ \{ \,x =\const .\} $ under H\'enon-like map is the plane $ \{\, y= \const .\} $. \ssk \\
Let $ f $ be the unimodal map defined on the closed interval $ I $ such that $ f(I) \subset I $ and the critical point and the critical value are in $ I $. $ f $ is {\em renormalizable} with period doubling type if there exists a closed subinterval $ J \neq I $ containing the critical point, $ c_f $ of $ f $ and $ f^2(J) $ is invariant and $ f^2(c_f) \in \di J $. Thus if $ f $ is renormalizable, then we can consider the smallest interval $ J_f $ satisfying the above properties. The conjugation \ssk of the appropriate affine rescaling of $ f^2|_{J_f} $ defines the renormalization of $ f $, $ Rf \colon I \ra I $. If $ f $ is infinitely renormalizable, there exists the renormalization fixed point $ f_* $. The {\em scaling factor} of $ f_* $ is
\begin{equation*}
\si = \frac{| \;\!J_{f_*}|}{| \:\! I |}
\end{equation*}
and $ \la = 1/ {\si} = 2.6 \ldots $\;. For the properties of renormalizable unimodal maps on the bounded interval, for example, see \cite{BB}.
However, the image of $ \{ \,x= \const .\} $ under $ F^2 $ is the surface
\begin{equation*}
f(x) - \eps(x,y,z) = \const .
\end{equation*}
which is not a plane except that $ \eps \equiv 0 $. Thus analytic definition of renormalization of $ F $ requires non-linear coordinate change map. The {\em horizontal-like} diffeomorphism $ H $ of $ F $ is defined as follows
\begin{equation*}
H(x,y,z) = (f(x) - \eps(x,y,z),\ y,\ z - \de(y,f^{-1}(y),0))
\end{equation*}
and it preserves the plane $ \{\,y=\const .\} $. Then {\em renormalization} of $ F $ is defined
\begin{equation*}
RF = \La \circ H \circ F^2 \circ H^{-1} \circ \La^{-1}
\end{equation*}
where $ \La(x,y,z) = (sx,\,sy,\,sz) $ for the appropriate constant $ s < -1 $. If $ F $ is $ n- $times renormalizable, then $ R^kF $ is defined as the renormalization of $ R^{k-1}F $ for $ 2 \leq k \leq n $. Denote $ \Dom(F) $ to be the cubic box region, $ B $. If the set $ B $ is emphasized with the relation of a certain map $ R^kF $, for example, then denote this region to be $ B(R^kF) $. 
\ssk \\
Let $ \II(\bar \eps) $ be the space of the infinitely renormalizable H\'enon-like maps with small norm, $ \| \:\!\eps \|, \| \:\!\de \| = O(\bar \eps) $. If $ \bar \eps >0 $ is sufficiently small, then renormalization operator $ R $ has the unique fixed point in $ \II(\bar \eps) $. The fixed point $ F_* $ is the degenerate map which is of following form
\begin{equation*}
F(x,y,z) = (f_*(x),\ x,\ 0) .
\end{equation*}
$ F_* $ is the hyperbolic fixed point of the renormalization operator and it has codimension one stable manifold at $ F_* $. 
\ssk \\
$ F_k $ denotes $ R^kF $ for each $ k $. Let the coordinate change map which conjugates $ F_k^2|_{\La_k^{-1}(B)} $ and $ RF_k $ is denoted by 
\begin{equation*}
\begin{aligned}
\psi^{k+1}_v \equiv H_k^{-1} \circ \La_k^{-1} \colon \Dom(RF_k) \ra \La_k^{-1}(B)
\end{aligned} \msk
\end{equation*}
\nin where $ H_k $ is the horizontal-like diffeomorphism and $ \La_k $ is dilation with the constant $ s_k < -1 $. Let $ \psi^{k+1}_c = F_k \circ \psi^{k+1}_v $. For $ k < n $ and express the compositions of $ \psi^{k+1}_v $ and $ \psi^{k+1}_c $ as follows
\msk
\begin{equation*}
\begin{aligned}
\Psi^n_{k,\,{\bf v}} &= \psi^{k+1}_v \circ \psi^{k+2}_v \circ \cdots \circ \psi^n_v \\[0.2em]
\Psi^n_{k,\,{\bf c}} &= \psi^{k+1}_c \circ \psi^{k+2}_c \circ \cdots \circ \psi^n_c .
\end{aligned} \msk
\end{equation*}
Moreover, the word of length $ n $ in the Cartesian product, $ \{ v, c \}^n $ be $ {\bf w}_n $ or simply $ {\bf w} $. 
The map $ \Psi^n_{k,\,{\bf w}} $ is from $ B(R^nF) $ to $ B(R^kF) $ where the word $ {\bf w} $ of the length $ n-k $. Denote the region $ \Psi^n_{k,\,{\bf w}}(B(R^nF)) $ by $ B^n_{{\bf w}} $. 
We see that $ \diam (B^n_{{\bf v}}) \leq C\si^n $ where $ {\bf v} = v^n $ for some $ C>0 $ in \cite{CLM} or \cite{Nam1}. If $ F $ is a infinitely renormalizable H\'enon-like map, then it has invariant Cantor set
\begin{equation*}
\OO_F = \bigcap_{n=1}^{\infty} \bigcup_{{\bf w} \in\,W^n} B^n_{\bf w} 
\end{equation*}
and $ F $ acts on $ \OO_F $ as a dyadic adding machine. The counterpart of the critical value of the unimodal renormalizable map is called the {\em tip}
\begin{equation*}
\{\tau_F \} \equiv \bigcap_{n\geq \;\! 0} B^n_{v^n} .
\end{equation*}
The unique invariant probability measure on $ \OO_F $ is denoted by $ \mu $. The average Jacobian of $ F $, $ b_F $ is defined as
\begin{equation*}
b_F = \exp \int_{\OO_F} \log \,\Jac F \;d\mu .
\end{equation*}
Then there exists the asymptotic expression of $ \Jac R^nF $ for the map $ F \in \II(\bar \eps) $ with $ b_F $ and the universal function.
\ssk
\begin{thm}[\cite{Nam2}]
For the map $ F \in \II(\bar \eps) $ with small enough positive number $ \bar \eps $, the Jacobian determinant of renormalization of $ F $ as follows
\begin{equation}
\Jac R^nF = b_F^{2^n}a(x)\, (1+O(\rho^n))
\end{equation}
where $ b_F $ is the average Jacobian of $ F $ and $ a(x) $ is the universal positive function for $ n \in \N $ and for some $ \rho \in (0,1) $.
\end{thm}
\ssk 
\nin Denote the tip, $ \tau_{F_n} $ for $ n \in \N $ to be $ \tau_n $. The definitions of tip and $ \Psi^n_{k,\,{\bf v}} $ imply that $ \Psi^n_{k,\,{\bf v}}(\tau_n) = \tau_k $ for $ k < n $. Then after composing appropriate translations, tips on each level moves to the origin as the fixed point
\begin{equation*}
\Psi^n_k(w) = \Psi^n_{k,\,{\bf v}}(w + \tau_n) - \tau_k
\end{equation*}
for $ k < n $. The map $ \Psi^n_k $ is separated non linear part and dilation part after reshuffling
\msk
\begin{equation*}
\begin{aligned}
\Psi^n_k(w) = 
\begin{pmatrix}
1 & t_{n,\,k} & u_{n,\,k} \\[0.2em]
& 1 & \\
& d_{n,\,k} & 1
\end{pmatrix}
\begin{pmatrix}
\alpha_{n,\,k} & & \\
& \si_{n,\,k} & \\
& & \si_{n,\,k}
\end{pmatrix}
\begin{pmatrix}
x + S^n_k(w) \\
y \\[0.2em]
z + R^n_k(y)     
\end{pmatrix} 
\end{aligned} \msk
\end{equation*} 
where $ \alpha_{n,\,k} = \si^{2(n-k)}(1 + O(\rho^k)) $ and $ \si_{n,\,k} = (-\si)^{n-k}(1 + O(\rho^k)) $. In this paper, we confuse the map $ \Psi^n_{k,\,{\bf v}} $ with $ \Psi^n_k $ to obtain the simpler expression of each coordinate map of $ \Psi^n_{k,\,{\bf v}} $. For example, the third coordinate expression of $ \Psi^n_k $
\begin{equation*}
\si_{n,\,k}\,d_{n,\,k}\,y + \si_{n,\,k}\,\big[\,z + R^n_k(y)\,\big] 
\end{equation*}
means that $ \si_{n,\,k}\,d_{n,\,k}\,(y + \tau_n^y) + \si_{n,\,k}\,\big[\,z + \tau_n^z + R^n_k(y + \tau_n^y)\,\big] - \tau_k $ where $ \tau_n = (\tau_n^x,\; \tau_n^y,\; \tau_n^z) $. By the same way, the first coordinate map
\begin{equation*}
\alpha_{n,\,k}\,\big[\,x + S^n_k(w)\,\big] + \si_{n,\,k}\,t_{n,\,k}\,y + \si_{n,\,k}\,u_{n,\,k}\,\big[\,z + R^n_k(y) \,\big]
\end{equation*}
means that
\msk
\begin{equation*}
\alpha_{n,\,k}\,\big[\,(x + \tau_n^x) + S^n_k(w + \tau_n)\,\big] + \si_{n,\,k}\,t_{n,\,k}\,(y + \tau_n^y) + \si_{n,\,k}\,u_{n,\,k}\,\big[\,(z + \tau_n^z) + R^n_k(y+ \tau_n^y) \,\big] - \tau_k
\end{equation*} \ssk
for $ k < n $.
 Recall the definitions 
for later use
\ssk
\begin{equation*}
\begin{aligned}
\La_n^{-1}(w) =  \si_n \cdot w , \quad  \psi^{n+1}_v(w) =  H^{-1}_n (\si_n w) , \quad   
\psi^{n+1}_c(w) =  F_n \circ H^{-1}_n (\si_n w)  \\[0.3em]
\psi^{n+1}_v (B(R^{n+1}F)) = B^{n+1}_v ,  \qquad \psi^{n+1}_c (B(R^{n+1}F)) = B^{n+1}_c
\end{aligned} \ssk
\end{equation*}
for each $ n \in \N $.

\bsk 


\section{An invariant space under renormalization}

\subsection{A space of renormalizable maps from recursive formulas of $ D\de $}

Let $ F $ be a renormalizable three dimensional H\'enon-like map. 
\nin Denote partial derivatives of the composition as follows
$$ \di_x \{ P \circ Q(w) \} \equiv \di_x P(Q(w)) \qquad
\di_x P \ \text{at} \  Q(w) \ \;  \text{is} \ \; \di_x P \circ Q(w) .
 $$
\nin The similar notation is defined for partial derivatives over any other variables also. 
Recall that $ RF $ is the renormalized map of $ F $ and its third coordinate map of $ RF $ is $ \de_1 = \pi_z \circ RF $. Thus by the definition of renormalization, we obtain the relation between $ \de $ and $ \de_1 $ as follows 
\begin{equation*}
\de_1(w) = \si_0 \cdot \big[\,\de \circ F \circ H^{-1}(\;\!\si_0 w) - \de (\si_0 x,\, f^{-1}(\si_0 x),\,0 \:\!) \,\big] .
\end{equation*}
Then by Lemma \ref{appendix - Dde recursive form}, we obtain
\msk
\begin{equation*}
\begin{aligned}
\di_x \de_1(w) &= \ \boxed{\big[\,\di_y \de \circ \psi^1_c(w) + \di_z \de \circ \psi^1_c(w) \cdot \di_x \de \circ \psi^1_v(w) \,\big]} \cdot \di_x \phi^{-1}(\si_0 w) \\
& \qquad + \di_x \de \circ \psi^1_c(w) - \frac{d}{dx}\;\de(\si_0 x,\, f^{-1}(\si_0 x),\, 0) \\[0.6em]
\di_y \de_1(w) &= \ \boxed{\big[\,\di_y \de \circ \psi^1_c(w) + \di_z \de \circ \psi^1_c(w) \cdot \di_x \de \circ \psi^1_v(w) \,\big]} \cdot \di_y \phi^{-1}(\si_0 w) \\
& \qquad + \di_z \de \circ \psi^1_c(w) \cdot \Big[\, \di_y \de \circ \psi^1_v(w) + \di_z \de \circ \psi^1_v(w) \cdot \frac{d}{dy}\;\de(\si_0 y,\, f^{-1}(\si_0 y),\, 0)\,\Big] \\[0.6em]
\di_z \de_1(w) &= \ \boxed{\big[\,\di_y \de \circ \psi^1_c(w) + \di_z \de \circ \psi^1_c(w) \cdot \di_x \de \circ \psi^1_v(w) \,\big]} \cdot \di_z \phi^{-1}(\si_0 w) \\[0.2em]
& \qquad + \di_z \de \circ \psi^1_c(w) \cdot \di_z \de \circ \psi^1_v(w) 
\end{aligned} \msk
\end{equation*}
Then we can consider the maps such that above boxed expression is identically zero.
\ssk

\begin{defn}
\nin 
Let $ \NN $ be the set of renormalizable three dimensional H\'enon-like maps each of which satisfies the following equation
\msk
\begin{equation*} \label{the Delta class}
\begin{aligned}
\di_y \de \circ (F(w)) + \di_z \de \circ (F(w)) \cdot \di_x \de (w) = 0
\end{aligned} \ssk
\end{equation*}
for all $ w \in \psi^1_v(B) \cup \psi^1_c(B) $. 
\end{defn}

\msk
\begin{example}
H\'enon-like maps in the above class $ \NN $ is non empty and non trivial. For instance, suppose that the third coordinate map of the given H\'enon-like map $ F $ is 
$$ \de(x,y,z) = \eta (Cy -z) + Cx $$
where $ \max \{ \| \eta \|_{C^3},\, | \:\!C| \,\} = O(\bar \eps) $ for $ C \in \R $. Then $ F $ is in $ \NN $. The map $ \eta $ can be arbitrary with small norm.
\end{example}
\ssk

\nin In the rest of this paper, we use the notation $ q(y) $ and $ q_k(y) $ as follows
\begin{equation} \label{definition of q(y)}
\begin{aligned}
q(y)  = \ \frac{d}{dy}\,\de(y, f^{-1}(y),0), \quad q_k(y)  = \ \frac{d}{dy}\,\de_k(y, f_k^{-1}(y),0) 
\end{aligned}
\end{equation}
for each $ k \in \N $. 
Moreover, the value of $ q_k $ at $ \si_k y $ is expressed as $ q_k \circ (\si_k y) $.

\msk

\subsection{Invariance of the space $ \NN $ under renormalization}

\msk

\begin{lem} \label{recursive formula of de in N}
Suppose that H\'enon-like map $ F_{k-1} $ is renormalizable. Denote $ RF_{k-1} $ by $ F_{k} $ and $ \pi_z \circ F_j(x,y,z) $ by $ \de_j(x,y,z) $ for $ j = k-1,\, k $. Then \ssk
\begin{equation*}
\begin{aligned}
\di_x \de_k(w) &= \di_x \de_{k-1} \circ \psi^k_c(w) - q_{k-1} \circ (\si_{k-1}\cdot x) \\[0.2em]
\di_y \de_k(w) &= \di_z \de_{k-1} \circ \psi^k_c(w) \cdot \big[\, \di_y \de_{k-1} \circ \psi^k_v(w) + \di_z \de_{k-1} \circ \psi^k_v(w) \cdot q_{k-1}(\si_{k-1} \cdot y)\,\big] \\[0.2em] 
\di_z \de_k(w) &= \di_z \de_{k-1} \circ \psi^k_c(w) \cdot \di_z \de_{k-1} \circ \psi^k_v(w) .
\end{aligned} \ssk
\end{equation*}
\end{lem}
\begin{proof}
See Lemma \ref{appendix - Dde recursive form} and use the induction. 
\end{proof}

\msk
\begin{lem} \label{lem-relation of F-k and F-k-1 with psi}
Let $ F $ be an infinitely renormalizable H\'enon-like map and let $ F_k $ be $ R^kF $ for each $ k \in \N $. Then
\begin{equation*}
\begin{aligned}
\psi^k_v \circ F_k &= F_{k-1} \circ \psi^k_c \\
\pi_y \circ \psi^k_v \circ F_k &= \pi_x \circ \psi^k_c
\end{aligned}
\end{equation*}
for each $ k \in \N $.
\end{lem}
\begin{proof}
Recall that $ \psi^k_v = H_{k-1} \circ \La_{k-1} $, $ \psi^k_c = F_{k-1} \circ \psi^k_v $ and $ F_k = (\psi^k_v)^{-1} \circ F_{k-1}^2 \circ \psi^k_v $. Then
\begin{equation} \label{eq-relation between psi-c,v}
\begin{aligned}
 \psi^k_v \circ F_k &= \psi^k_v \circ (\psi^k_v)^{-1} \circ F_{k-1}^2 \circ \psi^k_v \\[0.2em]
&= F_{k-1}^2 \circ \psi^k_v = F_{k-1} \circ \big[\,F_{k-1} \circ \psi^k_v \,\big] \\[0.2em] 
&= F_{k-1} \circ \psi^k_c
\end{aligned}
\end{equation}
for $ k \in \N $. The special form of H\'enon-like map implies that $ \pi_y (F_{k-1}(w)) = \pi_x(w) $. Hence, the equation \eqref{eq-relation between psi-c,v} implies that $ \pi_y \circ \psi^k_v \circ F_k = \pi_x \circ \psi^k_c $.
\end{proof}
\msk

\nin Let us express the notation of $ \psi^k_w \circ \cdots \circ \psi^n_w $ where $ w = v $ or $ c \in W $ as follows \footnote{By the above definition, $ \Psi^n_{v^n} $ and $ \Psi^n_{c^n} $ can be also expressed as follows $$ \Psi^n_{v^n} = \Psi^n_{0,\, {\bf v}}, \quad \Psi^n_{c^n} = \Psi^n_{0,\, {\bf c}} $$}
\begin{equation*}
\begin{aligned}
\psi^k_v \circ \psi^{k+1}_v \circ \cdots \circ \psi^n_v & = \Psi^n_{k,\, v^{n-k}} \equiv \Psi^n_{k,\, {\bf v}} \\
\psi^k_c \circ \psi^{k+1}_c \circ \cdots \circ \psi^n_c & = \Psi^n_{k,\, c^{n-k}} \equiv \Psi^n_{k,\, {\bf c}} .
\end{aligned} \msk
\end{equation*}
Moreover, let us take the following notations
\begin{equation*}
\begin{aligned}
\Psi^n_{k,\, {\bf v}} \circ \psi^{n+1}_c \equiv \Psi^{n+1}_{k,\, {\bf v}c}, \quad 
\Psi^n_{k,\, {\bf c}} \circ \psi^{n+1}_v \equiv \Psi^{n+1}_{k,\, {\bf c}v} 
\end{aligned}
\end{equation*}
for each $ n \in \N $. The notation $ \Psi^{n+2}_{k,\, {\bf v}cv} $ or $ \Psi^{n+2}_{k,\, {\bf v}c^2} $ and any similar ones are allowed.

\begin{cor} \label{cor-relation between Psi-c and Psi-v}
Let $ F $ be the infinitely renormalizable H\'enon-like map. Then
$$ \Psi^n_{k,\,\vv} \circ F_n = F_k \circ \Psi^n_{k,\,\cc} $$
for $ k < n $.
\end{cor}
\begin{proof}
Recall the equation $ \psi^{j+1}_c = F_j \circ \psi^{j+1}_v $ for $ k \leq j < n $. Thus 
\msk
\begin{equation*}
\begin{aligned}
F_k \circ \Psi^n_{k,\,\cc} 
& = F_k \circ \psi^{k+1}_c \circ \psi^{k+2}_c \circ \cdots \circ \psi^n_c \\[0.2em]
& = \big[\,F_k \circ (F_k \circ \psi^{k+1}_v)\,\big] \circ (F_{k+1} \circ \psi^{k+2}_v) \circ \cdots \circ (F_{n-1} \cdot \psi^n_v) \\[0.2em]
& = \psi^{k+1}_v \circ \psi^{k+2}_v \circ F_{k+2} \circ (F_{k+2} \circ \psi^{k+3}_v) \circ \cdots \circ (F_{n-1} \circ \psi^n_v) \\
&  \hspace{1in} \vdots \\[0.2em]
& = \psi^{k+1}_v \circ \psi^{k+2}_v \circ \cdots \circ \psi^n_v \circ F_n \\[0.2em]
& = \Psi^n_{k,\,\vv} \circ F_n
\end{aligned}
\end{equation*}
\end{proof}

\begin{thm} \label{Invariance of the space-NN}
Let $ \NN $ be the set of renormalizable H\'enon-like maps defined in Definition \ref{the Delta class}. 
The space $ \II(\bar \eps) \cap \NN $ is invariant under renormalization, that is, if $ F \in \II(\bar \eps) \cap \NN $, then $ RF \in \II(\bar \eps) \cap \NN $ .
\end{thm}
\begin{proof}
Recall that $ \pi_z \circ R^kF $ is $ \de_k $ and $ D\de_k(w) = (\di_x \de(w) \quad \di_y \de_k(w) \quad \di_z \de_k(w)) $ for $ k \in \N $. Suppose that 
\begin{equation*}
\begin{aligned}
\di_y \de_{k-1} \circ (F_{k-1}(w)) + \di_z \de_{k-1} \circ (F_{k-1}(w)) \cdot \di_x \de_{k-1} (w) = 0
\end{aligned} \msk
\end{equation*}
where $ w \in \psi^k_c(B) \cup \psi^k_v(B) $. By induction it is sufficient to show that
$$ \di_y \de_{k} \circ (F_{k}(w)) + \di_z \de_{k} \circ (F_{k}(w)) \cdot \di_x \de_{k} (w) = 0 $$
where $ w \in \psi^{k+1}_c(B) \cup \psi^{k+1}_v(B) $. Observe that $ \si_{k-1}y = \pi_y \circ \psi^k_v(w) $ and $ \si_{k-1} x = \pi_x \circ \psi^k_c(w) $. By Lemma \ref{lem-relation of F-k and F-k-1 with psi}, we have that $ \pi_x \circ \psi^k_c(w) = \pi_y \circ \psi^k_v \circ F_{k}(w) $ and $ F_k \circ \psi^k_c(w) = \psi^k_v \circ F_{k}(w) $. Then by Lemma \ref{recursive formula of de in N},
\msk
\begin{align*}
&\quad \ \ \di_y \de_{k} \circ (F_{k}(w)) + \di_z \de_{k} \circ (F_{k}(w)) \cdot \di_x \de_{k} (w) \\[0.4em]
&= \ \di_z \de_{k-1} \circ \psi^k_c \circ F_k(w) \\
&\qquad \cdot \big[\, \di_y \de_{k-1} \circ \psi^k_v \circ F_k(w) + \di_z \de_{k-1} \circ \psi^k_v \circ F_k(w) \cdot q_{k-1} \circ \pi_y \circ \psi^k_v \circ F_k(y)\,\big] \\[0.2em]
&\quad \ + \di_z \de_{k-1} \circ \psi^k_c \circ F_k(w) \cdot \di_z \de_{k-1} \circ \psi^k_v \circ F_k(w) 
\cdot \big[\,\di_x \de_{k-1} \circ \psi^k_c(w) - q_{k-1} \circ \pi_x \circ \psi^k_c(x) \,\big] \\[0.4em]
&= \ \di_z \de_{k-1} \circ \psi^k_c \circ F_k(w) \cdot \big[\, \di_y \de_{k-1} \circ \psi^k_v \circ F_k(w)\\
&\quad \ + \di_z \de_{k-1} \circ \psi^k_c \circ F_k(w) \cdot \di_z \de_{k-1} \circ \psi^k_v \circ F_k(w) \cdot \di_x \de_{k-1} \circ \psi^k_c(w) \\[0.4em]
&= \ \di_z \de_{k-1} \circ \psi^k_c \circ F_k(w) \\
&\qquad \cdot \big[\, \di_y \de_{k-1} \circ \psi^k_v \circ F_k(w) + \di_z \de_{k-1} \circ \psi^k_v \circ F_k(w) \cdot \di_x \de_{k-1} \circ \psi^k_c(w)\,\big]\\[0.4em]
&= \ 0
\end{align*} 
For any point $ w \in \Dom(R^kF) $, we obtain that $ \psi^k_c(w) \in \psi^k_c(B) \cup \psi^k_v(B) $. Then $ RF_{k-1} \in \NN \cap \II(\bar \eps) $. Hence, the space $ \NN \cap \II(\bar \eps) $ is invariant under renormalization.
\end{proof}

\bsk

\section{Universal numbers with $ \di_z \de $ and $ \di_y \eps $}

\subsection{Critical point and recursive formula of $ \di_x \de_n $}
Let us define the {\em critical point}, $ c_F $ of $ F \in \II(\bar \eps) $ as the inverse image of the tip, $ \tau_F $ under $ F $, that is, $ c_F = F^{-1}(\tau_F) $. Let the tip and the critical point of $ R^kF $ be $ \tau_k $ and $ c_{F_k} $ respectively. Recall the definition of tip
\begin{equation*}
\{\tau_k \} = \bigcap_{n \geq k+1} \Psi^n_{k,\,\vv}(B)
\end{equation*} 
where $ B $ is the domain, $ B(R^nF) $ for each $ n \in \N $. The above intersection is nested and each $ \Psi^n_{k,\,\vv}(B) $ is connected. Then the tip is just the limit of the sequence of $ \Psi^n_{v^n}(B) $ as follows
\begin{align} \label{definition of the tip}
\{\tau_k \} = \bigcap_{n \geq 1} \Psi^n_{k,\,\vv}(B) = \lim_{n\ra \infty} \Psi^n_{k,\,\vv}(B) .
\end{align}

\begin{lem}
Let $ F $ be the H\'enon-like map in $ \II(\bar \eps) $. Then the critical point of $ F $, $ c_F $ is the following limit
$$ \{ c_{F_k} \} = \lim_{n \ra \infty} \Psi^n_{k,\,\cc}(B) $$
where $ B $ is the domain, $ B(R^nF) $ for $ k < n $.
\end{lem}
\begin{proof}
Since $ \diam ( \Psi^n_{\bf w}) \leq C \si^{n-k} $ for some $ C>0 $, the limit of $ \Psi^n_{{\bf w}_n}(B) $ as $ n \ra \infty $ is a single point. 
By Corollary \ref{cor-relation between Psi-c and Psi-v}, the following equation holds 
\begin{align*}
\Psi^n_{k,\,\vv} \circ F_n = F_k \circ \Psi^n_{k,\,\cc}
\end{align*}
for each $ k< n $. Observe that $ \psi^{n+1}_v (B(R^{n+1}F)) \subset B(R^nF) $ and $ \psi^{n+1}_c (B(R^{n+1}F)) \subset B(R^nF) $. Then passing the limit, the following equation holds
\msk
\begin{equation} \label{pre-critical point as the limit}
\begin{aligned}
 F_k \circ \lim_{n \ra \infty} \Psi^n_{k,\,\cc}(B(R^nF)) & = \lim_{n \ra \infty} F_k \circ \Psi^n_{k,\,\cc}(B(R^nF)) \\
 & = \lim_{n \ra \infty} F_k \circ \Psi^{n}_{k,\,\cc} \circ \psi^{n+1}_v (B(R^{n+1}F)) \\
 & = \lim_{n \ra \infty} \Psi^n_{k,\,\vv} \circ \psi^{n+1}_c (B(R^{n+1}F)) \\
 & = \lim_{n\ra \infty} \Psi^n_{k,\,\vv}(B(R^nF)) = \{\tau_k \}
\end{aligned} \msk
\end{equation}
for all $ n \in \N $ because each limit is a single point set. Then the critical point of $ F $, $ \{ c_{F_k} \} $ is $ \lim_{n \ra \infty} \Psi^n_{k,\,\cc}(B) $.
\end{proof}
\msk

\begin{prop} \label{recursive formula of di-x de-n}
Let $ F $ be the H\'enon-like map in $ \NN \cap \II(\bar \eps) $. Then the following equation holds
\begin{equation*}
\begin{aligned}
\di_x \de_n(w) = \di_x \de_k \circ \Psi^n_{k,\,\cc}(w) - \sum_{i=k}^{n-1} \; q_i \circ \big( \pi_x \circ \Psi^n_{i,\,{\bf c}}(w) \big) 
\end{aligned}
\end{equation*} 
where $ \de_n(w) $ is the third coordinate map of $ R^nF $ for each $ n \in \N $. Moreover, passing the limit the following equation holds
\begin{equation*}
\di_x \de (c_{F_k}) = \lim_{n \ra \infty} \; \sum_{i=k}^{n-1} \, q_i \circ \big( \pi_x \circ \Psi^n_{i,\,{\bf c}}(w) \big) = \lim_{n \ra \infty} \; \sum_{i=k}^{n-1} \, q_i (\pi_x(c_{F_i}))
\end{equation*}
where $ c_{F_k} $ is the critical point of $ R^kF $ for $ k<n $.
\end{prop}
\begin{proof}
By Lemma \ref{recursive formula of de in N}, we see
\begin{equation*}
\di_x \de_{n}(w) = \ \di_x \de_{n-1} \circ \psi^n_c (w) - q_{n-1} \circ (\si_{n-1}\cdot x)
\end{equation*}
Recall the definition of $ q_k(x) $ in the equation \eqref{definition of q(y)}. Then inductively we obtain that \msk
\begin{equation*}
\begin{aligned}
\di_x \de_{n}(w) &= \ \di_x \de_{n-1} \circ \psi^n_c (w) - q_{n-1}(\pi_x \circ \psi^n_c (w)) \\[0.2em]
&= \ \di_x \de_{n-2} \circ (\psi^{n-1}_c \circ \psi^n_c (w)) - q_{n-2} \circ (\pi_x \circ \psi^{n-1}_c \circ \psi^n_c (w)) - q_{n-1} \circ (\pi_x \circ \psi^n_c (w)) \\
& \hspace{2in} \vdots \\
&= \ \di_x \de_k \circ \Psi^n_{k,\,\cc}(w) - \sum_{i=k}^{n-1} \; q_i \circ \big( \pi_x \circ \Psi^n_{i,\,{\bf c}}(w) \big) .
\end{aligned}
\end{equation*}
Recall that $ \| \,\di_x \de_{n} \| \leq C \bar \eps^{2^n} $ for some $ C>0 $ and $ \displaystyle \lim_{n \ra \infty} \Psi^n_{i,\,{\bf c}}(B) = \{ c_{F_i} \} $ for each $ i < n $. Thus passing the limit, we obtain 
\begin{equation*}
\di_x \de_k (c_F) = \lim_{n \ra \infty} \; \sum_{i=k}^{n-1} \, q_i \circ \big( \pi_x \circ \Psi^n_{i,\,{\bf c}}(w) \big) .
\end{equation*}
for $ w \in \Dom(R^nF) $. Since $ \di_x \de_k (c_F) $ is the single point for each $ k $ and the critical points of each level, $ c_{F_i} $ are contained in  $ \Psi^n_{i,\,{\bf c}}(B) $ for all $ n \in \N $. Then 
$$ \di_x \de_k (c_{F_k}) = \lim_{n \ra \infty} \; \sum_{i=k}^{n-1} \, q_i (\pi_x(c_{F_i})) . $$
\end{proof}

\msk

\subsection{Universal number $ b_2 $ and the asymptotic of $ \di_z \de_n $ and $ \di_y \de_n $} \label{Universal number b-1}

\begin{prop} \label{asymptotic of di-z de-n in Delta}
Let $ F $ be H\'enon-like diffeomorphism in $ \NN \cap \II(\bar \eps) $. Let $ \de_n $ be the third coordinate map of $ R^nF $ for $ n \in \N $. Then 
\begin{equation*}
\begin{aligned}
\di_z \de_n = b_2^{2^n} (1 + O(\rho^n))
\end{aligned}
\end{equation*}
for each $ n \in \N $ and $ 0 < \rho <1 $ where $ b_2 $ is a non zero number. 
\end{prop}
\begin{proof}
The recursive formula of $ \di_z \de_n $ in Lemma \ref{recursive formula of de in N} implies the following equation by inductive calculation \msk
\begin{equation} \label{de-z recursive formula}
\begin{aligned}
\di_z \de_{n}(w) &= \ \di_z \de_{n-1} \circ (F_{n-1} \circ H^{-1}_{n-1}(\si_{n-1} w)) \cdot \di_z \de_{n-1} \circ H^{-1}_{n-1}(\si_{n-1} w) \\[0.3em]
&= \ \di_z \de_{n-1} \circ \psi^n_c(w) \cdot \di_z \de_{n-1} \circ \psi^n_v(w) \\[0.3em]
&= \ \di_z \de_{n-2} \circ (\psi^{n-1}_c \circ \psi^n_c(w)) \cdot \di_z \de_{n-2} \circ (\psi^{n-1}_v \circ \psi^n_c(w)) \\
& \qquad \cdot \di_z \de_{n-2} \circ (\psi^{n-1}_c \circ \psi^n_v(w)) \cdot \di_z \de_{n-2} \circ  (\psi^{n-1}_v \circ \psi^n_v(w)) \\
& \hspace{1in} \vdots \\
&= \ \prod_{{\bf w} \in \; W^n}  \di_z \de \circ \Psi^n_{{\bf w}}(w) .
\end{aligned}
\end{equation}
where $ {\bf w} $ is the word of length $ n $ in $ W^n = \{v \ c\}^n $. The number of words in $ W^n $ is $ 2^n $. Let us take the logarithmic average of $ | \:\! \di_z \de_n | $ on the regions $ \Psi^n_{{\bf w}}(B) $ and let this map be $ l_n(w) $ for each $ n \in \N $ 
\footnote{If $ \di_z \de (w) = 0 $ for some $ w \in B $, then $ \di_y \de (w) = 0 $ at the same point because $ F \in \NN $. Thus $ \Jac F(w) = 0 $, that is, $ F $ cannot be a dffeomorphism. Moreover, $ \di_z \de $ is defined on some compact set which contains the set $ \bigcup_{{\bf w} \in \; W^n} \Psi^n_{{\bf w}}(B) $. Then we may assume that $ \di_z \de (w) $ has the positive lower bounds (or negative upper bounds) on $ B^1_v \cup B^1_c $.}
\ssk
\begin{equation} \label{average of de-z n}
\begin{aligned}
l_n(w) = \frac{1}{\;2^n} \sum_{{\bf w} \in \; W^n} \log |\, \di_z \de \circ \Psi^n_{{\bf w}}(w) | .
\end{aligned} \ssk
\end{equation}

\nin The limit of $ l_n (w) $ as $ n \ra \infty $ is a function defined on the critical Cantor set, $ \OO_F $ as $ n \ra \infty $. However, values of the limit function at all points of $ \OO_F $ are the same as each other, that is, the limit is a constant function. In particular, we have
\ssk
\begin{equation*}
\begin{aligned}
l_n(w) \lra \int_{\OO_F} \log |\:\! \di_z \de | \; d\mu .
\end{aligned} \ssk
\end{equation*} 
where $ \mu $ is the unique ergodic probability measure on $ \OO_F $. Let this limit be $ \log b_2 $ \,for some $ b_2 > 0 $. Since $ \diam (\Psi^n_{{\bf w}}(B)) \leq C \si^n $ for some $ C>0 $ and for all $ {\bf w} \in W^n $, the above equation \eqref{average of de-z n} converges exponentially fast as $ n \ra \infty $. In other words,
\begin{equation*}
\begin{aligned}
\frac{1}{\;2^n} \log |\;\! \di_z \de_n (w) | & = \ \log b_2 + O(\rho_0^n) 
\end{aligned}
\end{equation*}
for some $ 0< \rho_0 <1 $. Let us choose the constant $  \rho = \rho_0 / 2 $. Then we obtain the following equation
\begin{equation} \label{b-2 as the asymptotic of di-z de}
\begin{aligned}
\log |\, \di_z \de_n (w) | & = \ 2^n \log b_2 + O(\rho^n) \\
& = \ 2^n \log b_2 + \log (1 + O(\rho^n)) \\
& = \ \log \, b_2^{2^n} (1 + O(\rho^n)) .
\end{aligned}
\end{equation}
Hence,
\begin{equation}
| \;\!\di_z \de_n | = b_2^{2^n} (1 + O(\rho^n)) .
\end{equation}
We may assume that $ \di_z \de $ is non zero. The proof is complete.
\end{proof}
\msk


\begin{lem} \label{asymptotic of di-y de-n in Delta}
Let $ F $ be the H\'enon-like diffeomorphism in $ \NN \cap \II(\bar \eps) $. Let $ \de_n(w) $ be the third coordinate map of $ F_n \equiv R^nF $ for each $ n \in \N $. Then the following equation holds \msk
\begin{equation*}
\begin{aligned}
&  \di_y \de_n(w) \cdot \di_z \de_k \circ \Psi^n_{k,\,\vv}(w) \\[-0.5em] 
& \qquad \qquad = \ \di_z \de_n(w) \cdot \Big[\, \di_y \de_k \circ \Psi^n_{k,\,\vv}(w) + \sum_{i=k}^{n-1} \; q_i \circ \big( \pi_y \circ \Psi^n_{i,\,{\bf v}}(w) \big) \cdot \di_z \de_k \circ \Psi^n_{k,\,\vv}(w) \,\Big]
\end{aligned}
\end{equation*} 
for $ k < n $. Moreover, 
\begin{equation*}
\di_y \de_k \circ \Psi^n_{k,\,\vv}(w) \cdot \big[\,\di_z \de_k \circ \Psi^n_{k,\,\vv}(w)\,\big]^{-1} + \sum_{i=k}^{n-1} \; q_i \circ \big( \pi_y \circ \Psi^n_{i,\,{\bf v}}(w) \big) \ \leq \ C \si^{n-k}
\end{equation*}
for some $ C>0 $ and $ 0 < \rho <1 $. 
\end{lem}
\begin{proof}
Recall the recursive formula in Lemma \ref{recursive formula of de in N} \msk
\begin{align*}
\di_y \de_{n}(w) &= \di_z \de_{n-1} \circ \psi^n_c (w) \cdot \di_y \de_{n-1} \circ \psi^n_v (w) + \di_z \de_n(w) \cdot q_{n-1} \circ (\pi_y \circ \psi^n_v (w)) \\[0.2em]
\di_z \de_n(w) &= \di_z \de_{n-1} \circ \psi^n_c(w) \cdot \di_z \de_{n-1} \circ \psi^n_v(w) . \\[-1em]
\end{align*} %
Then by the inductive calculation, we have the following equation
\msk 
\begin{align*}
&\quad \ \ \di_y \de_{n}(w) \\[0.5em]
&= \ \di_z \de_{n-1} \circ \psi^n_c (w) \cdot \di_y \de_{n-1} \circ \psi^n_v (w) + \di_z \de_n(w) \cdot q_{n-1} \circ (\pi_y \circ \psi^n_v (w)) \\[0.3em]
&= \ \di_z \de_{n-1} \circ \psi^n_c (w) \cdot \Big[\,\di_z \de_{n-2} \circ (\psi^{n-1}_c \circ \psi^n_v(w)) \cdot \di_y \de_{n-2} \circ (\psi^{n-1}_v \circ \psi^n_v(w)) \\[-0.2em]
&\qquad + \di_z \de_{n-1} \circ \psi^n_v (w) \cdot q_{n-2} \circ (\pi_y \circ (\psi^{n-1}_v \circ \psi^n_v (w)) \,\Big] + \di_z \de_n(w) \cdot q_{n-1} \circ (\pi_y \circ \psi^n_v (w)) \\[0.3em]
&= \ \di_z \de_{n-1} \circ \psi^n_c (w) \cdot \di_z \de_{n-2} \circ (\psi^{n-1}_c \circ \psi^n_v(w)) \cdot \di_y \de_{n-2} \circ (\psi^{n-1}_v \circ \psi^n_v(w)) \\
&\qquad + \di_z \de_n(w) \cdot \Big[\,q_{n-2} \circ (\pi_y \circ (\psi^{n-1}_v \circ \psi^n_v (w)) + q_{n-1} \circ (\pi_y \circ \psi^n_v (w))\,\Big] \\
& \hspace{2in} \vdots \\
&= \ \di_z \de_{n-1} \circ \psi^n_c (w) \cdot \di_z \de_{n-2} \circ (\psi^{n-1}_c \circ \psi^n_v(w)) \cdots  \di_z \de_k \circ (\psi^{k+1}_c \circ \psi^{k+2}_v \circ \cdots \circ \psi^n_v(w)) \\
&\qquad \cdot \di_y \de_k \circ \Psi^n_{k,\,\vv}(w) 
+ \di_z \de_n(w) \; \sum_{i=k}^{n-1} \; q_i \circ \big( \pi_y \circ \Psi^n_{i,\,{\bf v}}(w) \big) . \\[-1em]
\end{align*} 
Thus let us multiply $ \di_z \de_k \circ \Psi^n_{k,\,\vv}(w) $ to both sides. Then we obtain that \bsk
\begin{align*}
&\quad \ \ \di_y \de_{n}(w) \cdot \di_z \de_k \circ \Psi^n_{k,\,\vv}(w)  \\[-0.2em]
&= \ \di_z \de_n(w) \cdot \di_y \de_k \circ \Psi^n_{k,\,\vv}(w) + \di_z \de_n(w) \cdot \sum_{i=k}^{n-1} \; q_i \circ \big( \pi_y \circ \Psi^n_{i,\,{\bf v}}(w) \big) \cdot \di_z \de_k \circ \Psi^n_{k,\,\vv}(w) \\[-0.8em]
&= \ \di_z \de_n(w) \cdot \Big[\, \di_y \de \circ \Psi^n_{k,\,\vv}(w) + \sum_{i=k}^{n-1} \; q_i \circ \big( \pi_y \circ \Psi^n_{i,\,{\bf v}}(w) \big) \cdot \di_z \de_k \circ \Psi^n_{k,\,\vv}(w) \,\Big] 
\end{align*}
for $ k<n $. Since $ F_k $ is a diffeomorphism, there exists $ \big[\,\di_z \de_k \circ \Psi^n_{k,\,\vv}(w)\,\big]^{-1} $ for all $ w \in B(R^nF) $. Then
\begin{equation*}  \label{recursive formula of di-y de-n}
\begin{aligned}
\di_y \de_n(w) 
= \ \di_z \de_n(w) \cdot \Big[\,\big(\,\di_z \de_k \circ \Psi^n_{k,\,\vv}(w)\,\big)^{-1} \cdot \di_y \de_k \circ \Psi^n_{k,\,\vv}(w) + \sum_{i=k}^{n-1} \; q_i \circ \big( \pi_y \circ \Psi^n_{i,\,{\bf v}}(w) \big) \,\Big]
\end{aligned}
\end{equation*}
Let us estimate the second factor the right side of the above equation 
\begin{equation} \label{estimation of second factor in di-y de-n}
\begin{aligned}
\big(\,\di_z \de_k \circ \Psi^n_{k,\,\vv}(w)\,\big)^{-1} \cdot \di_y \de_k \circ \Psi^n_{k,\,\vv}(w)  + \sum_{i=k}^{n-1} \; q_i \circ \big( \pi_y \circ \Psi^n_{i,\,{\bf v}}(w) \big) .
\end{aligned} \ssk
\end{equation} 

\nin Since $ F_i(c_{F_i}) = \tau_{i} $ and H\'enon-like map $ F_i $ is $ (f_i(x) -\eps_i(w),\ x,\ \de_i(w)) $, observe that 
\begin{equation*}
\pi_x(c_{F_i}) = \pi_y(\tau_{i})
\end{equation*}
for every $ i \in \N $. By Proposition \ref{recursive formula of di-x de-n}, we have the following equation
\begin{equation*}
\di_x \de_k (c_F) = \lim_{n \ra \infty} \; \sum_{i=k}^{n-1} \, q_i \circ \big( \pi_x \circ \Psi^n_{i,\,{\bf c}}(w) \big) 
\end{equation*}
and it converges exponentially fast. Recall the fact that \msk
\begin{align}
 \label{Psi n-i-c expression} \pi_x \circ \Psi^n_{i,\,{\bf c}}(w) &= \si_{n,\,i}\,x \\[0.2em]
\pi_y \circ \Psi^n_{i,\,{\bf c}}(w) &= \si_{n,\,i}\,y  \label{Psi n-i-v expression}
\end{align} 
for $ i < n $. Then the expression \eqref{Psi n-i-v expression} converges with the same rate of the expression \eqref{Psi n-i-c expression}. Take the limit of \eqref{estimation of second factor in di-y de-n}. 
\begin{align*}
&\quad \ \lim_{n \ra \infty} \big(\,\di_z \de_k \circ \Psi^n_{k,\,\vv}(w)\,\big)^{-1} \cdot \di_y \de_k \circ \Psi^n_{k,\,\vv}(w)  + \lim_{n \ra \infty} \sum_{i=k}^{n-1} \; q_i \circ \big( \pi_y \circ \Psi^n_{i,\,{\bf v}}(w) \big)  \\[-0.5em]
&= \ \big( \di_z \de_k (\tau_k) \big)^{-1} \cdot \di_y \de_k (\tau_k) + \lim_{n \ra \infty} \sum_{i=k}^{n-1} \; q_i \circ \big( \pi_y(\tau_{i}) \big) \\ 
&= \ \big( \di_z \de_k (\tau_k) \big)^{-1} \cdot \di_y \de_k (\tau_k) + \lim_{n \ra \infty} \; \sum_{i=k}^{n-1} \, q_i \circ \big( \pi_x (c_{F_i} ) \big) \\[0.3em]
&= \ \big( \di_z \de_k (\tau_k) \big)^{-1} \cdot \di_y \de_k (\tau_k) + \di_x \de_k (c_{F_k}) \\[1em]
&= \ 0  \\[-1.3em]
\end{align*} 
%
Recall that $ \diam \Psi^n_{k,\,{\bf w}}(B) \leq C \si^{n-k} $ for some $ c > 0 $. Thus $ \Psi^n_{k,\,\vv}(B) $ and $ \Psi^n_{k,\,\cc}(B) $ converge to $ \tau_F $ and $ c_F $ respectively as $ n \ra \infty $ at the same rate. 
Then 
\begin{equation} \label{asymptotic of second factor in di-y de-n}
\big(\,\di_z \de_k \circ \Psi^n_{k,\,\vv}(w)\,\big)^{-1} \cdot \di_y \de_k \circ \Psi^n_{k,\,\vv}(w) + \sum_{i=k}^{n-1} \; q_i \circ \big( \pi_y \circ \Psi^n_{i,\,{\bf v}}(w) \big) \leq C \si^{n-k}
\end{equation}
for some $ C>0 $. 
\end{proof}


\subsection{Universal number $ b_1 $ and $ \di_y \eps $}
In Section \ref{Universal number b-1}, the universal number $ b_2 $ represents asymptotic for $ \di_z \de $. Universality of Jacobian determinant implies that average Jacobian of $ F $, $ b_F $ is the universal number. Define the number $ b_1 $ as the ratio $ b_F/b_2 $ and we would show that $ b_1 $ is also the universal number which describes the asymptotic of $ \di_y \eps $.
\ssk \\
Suppose that H\'enon-like diffeomorphism $ F $ is in $ \NN \cap \II(\bar \eps) $. 
By the universal expression of Jacobian determinant with average Jacobian and Proposition \ref{asymptotic of di-z de-n in Delta}, we obtain that\ssk
\begin{align*}
\Jac F_k(w) & = b_F^{2^k}a(x)(1+O(\rho^k)) \\[0.2em]
&=  \di_y \eps_k(w) \cdot \di_z \de_k(w) - \di_z \eps_k(w) \cdot \di_y \de_k(w) \\[0.2em]
&= \big[\, \di_y \eps_k(w) - \di_z \eps_k(w) \cdot \di_y \de_k(w) \cdot \big(\,\di_z \de_k(w) \big)^{-1} \,\big] \cdot \di_z \de_k(w) \\[0.2em]
&= \big[\, \di_y \eps_k(w) - \di_z \eps_k(w) \cdot \di_y \de_k(w) \cdot \big(\,\di_z \de_k(w) \big)^{-1} \,\big] \cdot b_2^{2^k}(1 + O(\rho^k)) . \\[-1em]
\end{align*} 
The above equation implies the existence of another universal number $ b_F/b_2 $. 

\begin{lem} \label{Another universal number b_1}
Let $ F $ be the H\'enon-like map in $ \NN \cap \II(\bar \eps) $. Then there exists the number $ b_1 \equiv b_F/b_2 $ satisfying the following equation
\begin{equation*} \label{expression for b-1}
\di_y \eps_k(w) - \di_z \eps_k(w) \cdot \di_y \de_k(w) \cdot \big(\,\di_z \de_k(w) \big)^{-1} = b_1^{2^k}a(x)(1 + O(\rho^k)) 
\end{equation*}
for each $ k \in \N $.
\end{lem}
\msk
\begin{lem} \label{di-y eps and q asymptotic from n to k}
Let $ F $ be the H\'enon-like diffeomorphism in $ \NN \cap \II(\bar \eps) $. 
Then the following equation holds for $ k < n $
$$ \di_y \eps_k \circ (\Psi^n_{k,\,\vv}(w)) + \di_z \eps_k \circ (\Psi^n_{k,\,\vv}(w)) \cdot \sum_{i=k}^{n-1} \; q_i \circ \big( \pi_y \circ \Psi^n_{i,\,{\bf v}}(w) \big) \leq C_1\,b_1^{2^k} + C_2\,\bar \eps^{2^k} \si^{n-k} $$
where $ w \in B(R^nF) $ for some positive $ C_1 $ and $ C_2 $.
\end{lem} 
\begin{proof}
The equation \eqref{recursive formula of di-y de-n} and \eqref{expression for b-1} implies that 
\begin{align*}
&\quad \ \ b_1^{2^k}a \circ (\Psi^n_{k,\,\vv}(w))(1 + O(\rho^k)) \\[0.3em]
&= \ \di_y \eps_k \circ (\Psi^n_{k,\,\vv}(w)) - \di_z \eps_k \circ (\Psi^n_{k,\,\vv}(w)) \cdot \di_y \de_k \circ (\Psi^n_{k,\,\vv}(w)) \cdot \big(\,\di_z \de_k \circ (\Psi^n_{k,\,\vv}(w)) \big)^{-1} \\[0.7em]
&= \ \di_y \eps_k \circ (\Psi^n_{k,\,\vv}(w)) \\[-0.5em]
& \qquad - \di_z \eps_k \circ (\Psi^n_{k,\,\vv}(w)) \cdot \Big[\, -\sum_{i=k}^{n-1} \; q_i \circ \big( \pi_y \circ \Psi^n_{i,\,{\bf v}}(w) \big) + \big(\di_z \de_n(w)\big)^{-1} \cdot \di_y \de_n(w) \,\Big] \\[-0.4em]
&= \ \di_y \eps_k \circ (\Psi^n_{k,\,\vv}(w)) + \di_z \eps_k \circ (\Psi^n_{k,\,\vv}(w)) \cdot \sum_{i=k}^{n-1} \; q_i \circ \big( \pi_y \circ \Psi^n_{i,\,{\bf v}}(w) \big) \\[-0.2em]
& \qquad - \di_z \eps_k \circ (\Psi^n_{k,\,\vv}(w)) \cdot \big(\di_z \de_n(w)\big)^{-1} \cdot \di_y \de_n(w) \\[-1em]
\end{align*} 
Lemma \ref{asymptotic of di-y de-n in Delta} implies that
$$ \| \:\!\di_z \eps_k \circ (\Psi^n_{k,\,\vv}(w)) \| \cdot \| \:\! \big(\di_z \de_n(w)\big)^{-1} \cdot \di_y \de_n(w) \| \leq C_2\,\bar \eps^{2^k}\si^{n-k} $$
for some $ C_2 > 0 $ independent of $ k $. Hence,
\begin{equation*}
\di_y \eps_k \circ (\Psi^n_{k,\,\vv}(w)) + \di_z \eps_k \circ (\Psi^n_{k,\,\vv}(w)) \cdot \sum_{i=k}^{n-1} \; q_i \circ \big( \pi_y \circ \Psi^n_{i,\,{\bf v}}(w) \big) \leq C_1\,b_1^{2^k} + C_2\,\bar \eps^{2^k} \si^{n-k}
\end{equation*} 
\end{proof}

\msk

\section{Recursive formula of $ \Psi^n_k $}
\nin 
In this section, let us calculate recursive formulas of some components of $ D\Psi^n_k $. Recall that $ \Psi^n_k $ is the conjugation between $ F_k^{2^{n-k}} $ and $ F_n $. These formulas in this section would be used in the estimation of minimal distances of a particular adjacent boxes and the diameter of boxes in the next sections.
\ssk

\begin{lem} \label{recursive formula of d, u, and t}
Let $ F \in \II(\bar \eps) $. 
The derivative of non-linear conjugation $ \Psi^n_k $ at the tip, $ \tau_{F_k} $ between $ F_k^{2^{n-k}} $ and $ F_n $ is as follows 
\begin{equation*}
\begin{aligned}
D^n_k \equiv D\Psi^n_k(\tau_n) = 
\begin{pmatrix}
\alpha_{n,\,k} & \si_{n,\,k}\, t_{n,\,k} & \si_{n,\,k}\, u_{n,\,k} \\[0.2em]
& \si_{n,\,k} & \\[0.2em]
& \si_{n,\,k}\, d_{n,\,k} & \si_{n,\,k}
\end{pmatrix}
\end{aligned} \ssk
\end{equation*}
where $ \si_{n,\,k} $ and $ \alpha_{n,\,k} $ are linear scaling factors such that $ \si_{n,\,k} = (-\si)^{n-k} (1 + O(\rho^k)) $ and $ \alpha_{n,\,k} = \si^{2(n-k)} (1 + O(\rho^k)) $.
Then 
\begin{align*}
d_{n,\,k} &= \sum_{i=k}^{n-1} d_{i+1,\,i} \, , \quad u_{n,\,k} = \sum_{i=k}^{n-1} \si^{i-k}\,u_{i+1,\,i}(1 + O(\rho^k))\\
  t_{n,\,k} &= \sum_{i=k}^{n-1} \si^{i-k}\, \big[\, t_{i+1,\,i} + u_{i+1,\,i}\,d_{n,\,i+1} \big](1 + O(\rho^k)) \\
  t_{n,\,k}- u_{n,\,k}\, d_{n,\,k} &= \sum_{i=k}^{n-1} \si^{i-k}\, \big[\, t_{i+1,\,i} - u_{i+1,\,i}\,d_{i+1,\,k} \big](1 + O(\rho^k)) 
\end{align*}
where\; $ \si^{i-k} (1 + O(\rho^k)) = {\displaystyle\prod_{j=k}^{i-1}} \dfrac{\alpha_{j+1,\,j}}{\si_{j+1,\,j}} $. Moreover, $ d_{n,\,k} $, $ u_{n,\,k} $ and $ t_{n,\,k} $ are convergent as $ n \ra \infty $ super exponentially fast.
\end{lem}

\begin{proof}
$ D^n_k = D^m_k \cdot D^n_m $ for any $ m $ between $ k $ and $ n $ because $ \Psi^n_k(\tau_{F_n}) $ is $ \tau_{F_k} $, the tip of $ k^{th} $ level. By the direct calculation, we obtain that \ssk
\begin{equation*}
\begin{aligned}
& D^m_k \cdot D^n_m \\ 
& \quad = 
\begin{pmatrix}
\alpha_{n,\,k} & \boxed{ \alpha_{m,\,k}\,\si_{n,\,m}\, t_{n,\,m} + \si_{n,\,k}\,t_{m,\,k} + \si_{n,\,k}\,u_{m,\,k}\,d_{n,\,m} } & \boxed{ \alpha_{m,\,k}\si_{n,\,m}\, u_{n,\,m} + \si_{n,\,k}\,u_{m,\,k} } \ \ 
\\[0.2em]
& \si_{n,\,k} & \\[0.2em]
& \boxed{ \si_{n,\,k}\, d_{m,\,k} + \si_{n,\,k}\,d_{n,\,m} } & \si_{n,\,k} 
\end{pmatrix} .
\end{aligned} 
\end{equation*}
 Then
\begin{align*}
\si_{n,\,k}\,t_{n,\,k} &= \alpha_{m,\,k}\,\si_{n,\,m}\, t_{n,\,m} + \si_{n,\,k}\,t_{m,\,k} + \si_{n,\,k}\,u_{m,\,k}\,d_{n,\,m} \\[0.3em]
\si_{n,\,k}\,u_{n,\,k} &= \alpha_{m,\,k}\,\si_{n,\,m}\, u_{n,\,m} + \si_{n,\,k}\,u_{m,\,k} \\[0.3em]
\si_{n,\,k}\,d_{n,\,k} &= \si_{n,\,k}\, d_{m,\,k} + \si_{n,\,k}\,d_{n,\,m} \\[-1em]
\end{align*} 
for any $ m $ between $ k $ and $ n $. Recall that $ \si_{n,\,k} = \si_{n,\,m} \cdot \si_{m,\,k} $ and $ \alpha_{n,\,k} = \alpha_{n,\,m} \cdot \alpha_{m,\,k} $. Let $ m $ be $ k+1 $. Then
\begin{align*}
d_{n,\,k} &= d_{n,\,k+1} + d_{k+1,\,k} \\[0.3em]
&= d_{n,\,k+2} + d_{k+2,\,k+1} + d_{k+1,\,k} \\
& \hspace{1in} \vdots \\
&= d_{n,\,n-1} + \cdots + d_{k+2,\,k+1} + d_{k+1,\,k} \\
&=  \sum_{i=k}^{n-1} d_{i+1,\,i} . \\[-1em]
\end{align*} 
%
Each term is bounded by $ \eps^{2^{\!\:i}} $ for each $ i $, that is, $ |\,d_{i+1,\,i} | \asymp |\,q_i (\pi_y(\tau_{i+1}))| \leq \| D\de_i \| = O(\bar \eps^{2^{\!\:i}}) $. Then $ d_{n,\,k} $ converges to the number, say $ d_{*,\,k} $ super exponentially fast.
\ssk \\
Let us see the recursive formula of $ u_{n,\,k} $
\begin{align*}
u_{n,\,k} &= \frac{\alpha_{k+1,\,k}}{\si_{k+1,\,k}} u_{n,\,k+1} + u_{k+1,\,k} \\
&= \frac{\alpha_{k+1,\,k}}{\si_{k+1,\,k}} \left[\,\frac{\alpha_{k+2,\,k+1}}{\si_{k+2,\,k+1}} u_{n,\,k+2} + u_{k+2,\,k+1} \,\right] + u_{k+1,\,k} \\
& \hspace{1in} \vdots \\
&= \sum_{i=k+1}^{n-1}\prod_{j=k}^{i-1} \frac{\alpha_{j+1,\,j}}{\si_{j+1,\,j}}\ u_{i+1,\,i} + u_{k+1,\,k}\\
&=  \sum_{i=k}^{n-1} \si^{i-k} u_{i+1,\,i}\, (1 + O(\rho^k)) . \\[-1em]
\end{align*} 
Since $ u_{i+1,\,i} \asymp \di_z \eps_i(\tau_{F_{i+1}}) $, $ u_{n,\,k} $ converges to the number, say $ u_{*,\,k} $ also super exponentially fast. 
Let us see the recursive formula of $ t_{n,\,k} $ \msk
\begin{align*}
t_{n,\,k} &= \frac{\alpha_{k+1,\,k}}{\si_{k+1,\,k}}\ t_{n,\,k+1} + t_{k+1,\,k} + u_{k+1,\,k}\,d_{n,\,k+1} \\
&= \frac{\alpha_{k+1,\,k}}{\si_{k+1,\,k}} \left[\,\frac{\alpha_{k+2,\,k+1}}{\si_{k+2,\,k+1}}\ t_{n,\,k+2} + t_{k+2,\,k+1} + u_{k+2,\,k+1}\,d_{n,\,k+2} \,\right] + t_{k+1,\,k} + u_{k+1,\,k}\,d_{n,\,k+1} \\
& \hspace{2in} \vdots \\
&= \sum_{i=k+1}^{n-1}\prod_{j=k}^{i-1} \frac{\alpha_{j+1,\,j}}{\si_{j+1,\,j}}\ t_{i+1,\,i} + t_{k+1,\,k} +  \sum_{i=k+1}^{n-1}\prod_{j=k}^{i-1} \frac{\alpha_{j+1,\,j}}{\si_{j+1,\,j}}\ u_{i+1,\,i}\,d_{n,\,i+1} + u_{k+1,\,k}\,d_{n,\,k+1} \\
&= \ \sum_{i=k}^{n-1} \si^{i-k} \big[ \,t_{i+1,\,i} + u_{i+1,\,i}\,d_{n,\,i+1} \,\big] (1 + O(\rho^k)) . \\[-1em]
\end{align*} 
By the above equations for $ d_{n,\,k} $, $ u_{n,\,k} $ and $ t_{n,\,k} $, 
we obtain the recursive formula of $ t_{n,\,k}- u_{n,\,k}\, d_{n,\,k} $ as follows
\msk
\begin{align*}
& \qquad  t_{n,\,k}- u_{n,\,k}\, d_{n,\,k} \\
&= \ \sum_{i=k+1}^{n-1}\prod_{j=k}^{i-1} \frac{\alpha_{j+1,\,j}}{\si_{j+1,\,j}} \big[ \,t_{i+1,\,i} + u_{i+1,\,i}\,d_{n,\,i+1} \,\big] + t_{k+1,\,k} + u_{k+1,\,k}\,d_{n,\,k+1} \\
& \qquad - \left[\, \sum_{i=k+1}^{n-1}\prod_{j=k}^{i-1} \frac{\alpha_{j+1,\,j}}{\si_{j+1,\,j}}\ u_{i+1,\,i} + u_{k+1,\,k} \right]  d_{n,\,k} \\
&= \ \sum_{i=k+1}^{n-1}\prod_{j=k}^{i-1} \frac{\alpha_{j+1,\,j}}{\si_{j+1,\,j}} \big[\,t_{i+1,\,i} + u_{i+1,\,i}\,d_{n,\,i+1} - u_{i+1,\,i} d_{n,\,k} \,\big] + t_{k+1,\,k} + u_{k+1,\,k}\,d_{n,\,k+1} - u_{k+1,\,k}\:\! d_{n,\,k} \\
&= \ \sum_{i=k+1}^{n-1}\prod_{j=k}^{i-1} \frac{\alpha_{j+1,\,j}}{\si_{j+1,\,j}} \big[\,t_{i+1,\,i} - u_{i+1,\,i} d_{i+1,\,k} \,\big] + t_{k+1,\,k} - u_{k+1,\,k}\:\! d_{k+1,\,k} \\
&= \ \ \sum_{i=k}^{n-1} \si^{i-k} \big[ \,t_{i+1,\,i} - u_{i+1,\,i}\,d_{i+1,\,k} \,\big] (1 + O(\rho^k)) . \\[-1em]
\end{align*} 
%
\nin Recall the derivative of coordinate change map at the tip on each level  \msk
\begin{equation*}
\begin{aligned}
\si_k \cdot DH_k (\tau_{F_k}) = (D^{k+1}_k)^{-1} = 
\begin{pmatrix}
(\alpha_k)^{-1} & & \\
& (\si_k)^{-1} & \\
& & (\si_k)^{-1}
\end{pmatrix}  \cdot
\begin{pmatrix}
1 & - t_k + u_k \, d_k & - u_k \\
& 1 & \\
& -d_k & 1
\end{pmatrix} .
\end{aligned} \msk
\end{equation*}
Since $ H_k(w) = (f_k(x) -\eps_k(w),\ y,\ z - \de_k(y,f_k^{-1}(y),0)) $, \ssk we see that $ \di_y \eps_k(\tau_{F_k}) \asymp - t_k + u_k \, d_k $ for every $ k \in \N $. Moreover, the fact that $ t_{i+1,\,i} - u_{i+1,\,i}\,d_{i+1,\,i} \asymp \di_y \eps_i(\tau_{F_{i+1}}) $ and $ |\,u_{i+1,\,i}\,d_{n,\,i}| $ is super exponentially small for each $ i<n $ implies that $ t_{n,\,k} $ converges to a number, say $ t_{*,\,k} $ super exponentially fast.

\end{proof}

\nin Recall the expression of $ \Psi^n_k $ from $ B(R^nF) $ to $ B^{n-k}_{{\bf v}}(R^kF) $ \msk
\begin{align*}
\Psi^n_k(w) = 
\begin{pmatrix}
1 & t_{n,\,k} & u_{n,\,k} \\[0.2em]
& 1 & \\
& d_{n,\,k} & 1
\end{pmatrix}
\begin{pmatrix}
\alpha_{n,\,k} & & \\
& \si_{n,\,k} & \\
& & \si_{n,\,k}
\end{pmatrix}
\begin{pmatrix}
x + S^n_k(w) \\
y \\[0.2em]
z + R^n_k(y)
\end{pmatrix} \\[-1em]
\end{align*}
where $ {\bf v} = v^{n-k} \in W^{n-k} $. Recall that $ \Psi^n_k $ be the map from $ B(R^nF) $ to $ B(R^kF) $ as the conjugation between $ (R^kF)^{2^{n-k}} $ and $ R^nF $. 
\msk

\begin{lem} \label{exponential smallness of R-n-k}
Let $ F \in \II(\bar \eps) $.  Then both $ R^n_k(y) $ and $ (R^n_k)'(y) $ converges to zero exponentially fast as $ n \ra \infty $ where $ R^n_k(y) $ be non-linear part of $ \pi_z \circ \Psi^n_k $ depending only on the second variable $ y $.
\end{lem}
\begin{proof}
Let $ w =(x,y,z) $ be the point in $ B(R^nF) $ and let $ \Psi^n_{n-1}(w) $ be $ w' = (x',y',z') $. Recall $ \Psi^n_k = \Psi^n_{n-1} \circ \Psi^{n-1}_k $. Thus
\begin{equation*}
\begin{aligned}
z' &= \ \pi_z \circ \Psi^n_{n-1}(w) = \si_{n,\,n-1} \,\big[\, d_{n,\,n-1}\, y + z + R^n_{n-1}(y) \,\big] \\
y' &= \ \pi_y \circ \Psi^n_{n-1}(w) = \si_{n,\,n-1}\,y .
\end{aligned}
\end{equation*}
Then by the composition of $ \Psi^{n-1}_k $ and $ \Psi^n_{n-1} $, we obtain the recursive formula of $ \pi_z \circ \Psi^n_k $ as follows  
\begin{align*} 
 \pi_z \circ & \Psi^n_k(w) \\[0.2em]  \numberthis \label{recursive formula of pi-z Psi-1}
&= \ \si_{n,\,k}\,\big[\,d_{n,\,k}\,y + z + R^n_k(y) \,\big] 
\\[0.2em]
&= \ \pi_z \circ \Psi^{n-1}_k(w') =  \si_{n-1,\,k}\,\big[\,d_{n-1,\,k}\,y' + z' + R^{n-1}_k(y') \,\big] \\[0.2em]
&= \ \si_{n-1,\,k}\,\big[\,d_{n-1,\,k}\,\si_{n,\,n-1}\,y + \si_{n,\,n-1} \,\big[\, d_{n,\,n-1}\:\! y + z + R^n_{n-1}(y) \,\big] + R^{n-1}_k(\si_{n,\,n-1}\,y) \,\big] \\[0.2em] \numberthis \label{recursive formula of pi-z Psi-2}
&= \ \si_{n,\,k}\,(d_{n-1,\,k} + d_{n,\,n-1}) + \si_{n,\,k}\,z + \si_{n,\,k}\,R^n_{n-1}(y) + \si_{n-1,\,k}\,R^{n-1}_k(\si_{n,\,n-1}\,y) . \\[-1em]
\end{align*} 
By Proposition \ref{recursive formula of d, u, and t}, $ d_{n,\,k} = d_{n-1,\,k} + d_{n,\,n-1} $. Let us compare \eqref{recursive formula of pi-z Psi-1} with \eqref{recursive formula of pi-z Psi-2}. Recall the equation $ \si_{n,\,k} = \si_{n,\,n-1} \cdot \si_{n-1,\,k} $. Then
\begin{equation*}
R^n_k(y) = R^n_{n-1}(y) + \frac{1}{\si_{n,\,n-1}}\ R^{n-1}_k(\si_{n,\,n-1}\,y) .
\end{equation*}
\nin Each $ R^i_j(y) $ is the sum of the second and higher order terms of $ \pi_z \circ \Psi^i_j $ for $ i>j $. Thus 
\begin{equation*}
R^n_k(y) =  a_{n,\,k}\,y^2 + A_{n,\,k}(y)\cdot y^3
\end{equation*}

\nin Moreover, $ \|\;\! R^n_{n-1} \| = O(\bar \eps^{2^{n-1}}) $ because $ R^n_{n-1}(y) $ is the second and higher order terms of the map $ \de_{n-1}(\si_{n,\,n-1}\,y,\,f_{n-1}^{-1}(\si_{n,\,n-1}\,y),\,0) $. Then
\begin{equation*}
R^n_k(y) = \frac{1}{\si_{n,\,n-1}}\ R^{n-1}_k(\si_{n,\,n-1}\,y) + c_{n,\,k}\,y^2 + O(\bar \eps^{2^{n-1}}y^3) 
\end{equation*}
where $ c_{n,\,k} = O(\bar \eps^{2^{n-1}}) $. The recursive formula for $ a_{n,\,k} $ and $ A_{n,\,k} $ as follows
\begin{equation*}
R^n_k(y) = \frac{1}{\si_{n,\,n-1}}\ \Big( a_{n-1,\,k}\cdot (\si_{n,\,n-1}\,y)^2 + A_{n-1,\,k}(\si_{n,\,n-1}\,y) \cdot (\si_{n,\,n-1}\,y)^3 \Big) + O(\bar \eps^{2^{n-1}}y^3) .
\end{equation*}
Then \ssk $ a_{n,\,k} = \si_{n,\,n-1}\,a_{n-1,\,k} + c_{n,\,k} $ and $ \| \:\! A_{n,\,k} \| \leq \| \,\si_{n,\,n-1}\|^2 \| \:\! A_{n-1,\,k} \| + O(\bar \eps^{2^{n-1}}) $ and for each fixed $ k<n $,\, $ a_{n,\,k} \ra 0 $ and $ A_{n,\,k} \ra 0 $ exponentially fast as $ n \ra \infty $. \ssk Thus $ R^n_k(y) $ converges to zero as $ n \ra \infty $ exponentially fast. 
Let us estimate $ \| \:\! A_{n,\,k}' \| $ in order to measure how fast $ (R^n_k)'(y) $ is convergent. By similar method, we have the recursive formula of $ (R^n_k)'(y) $ as follows \msk
\begin{equation*}
\begin{aligned}
(R^n_k)'(y) &= \ 2\, a_{n,\,k}\,y + 3 \:\! A_{n,\,k}(y)\cdot y^2 + A_{n,\,k}'(y)\cdot y^3 \\[0.3em]
\textrm{Thus} \quad
  (R^n_k)'(y) &= \ (R^n_{n-1})'(y) + R^{n-1}_k(\si_{n,\,n-1}\,y) \\
&= \ R^{n-1}_k(\si_{n,\,n-1}\,y) + 2\, c_{n,\,k}\,y + O(\bar \eps^{2^{n-1}}y^2) .
\end{aligned}
\end{equation*}
Then
\begin{equation*}
\begin{aligned}
(R^n_k)'(y) &= \ 2\, a_{n-1,\,k}\,\si_{n,\,n-1}\,y + 3 \:\! A_{n-1,\,k}(\si_{n,\,n-1}\,y)\cdot (\si_{n,\,n-1}\,y)^2 + A_{n,\,k}'(\si_{n,\,n-1}\,y)\cdot (\si_{n,\,n-1}\,y)^3 \\
 & \quad + 2\, c_{n,\,k}\,y + O(\bar \eps^{2^{n-1}}y^2) .
\end{aligned}
\end{equation*}
Let us compare quadratic and higher order terms of $ (R^n_k)'(y) $
\begin{equation*}
\begin{aligned}
3 \, A_{n,\,k}(y)\cdot y^2 + A_{n,\,k}'(y)\cdot y^3 &= \ 3 \:\! A_{n-1,\,k}(\si_{n,\,n-1}\,y)\cdot (\si_{n,\,n-1}\,y)^2 + A_{n,\,k}'(\si_{n,\,n-1}\,y)\cdot (\si_{n,\,n-1}\,y)^3 \\[0.2em]
 & \qquad + O(\bar \eps^{2^{n-1}}y^2) .
\end{aligned}
\end{equation*}
Thus
\begin{equation*}
\begin{aligned}
A_{n,\,k}'(y)\,y = A_{n,\,k}'(\si_{n,\,n-1}\,y)\cdot \si_{n,\,n-1}^3 \,y - 3 \:\! A_{n,\,k}(y) + 3 \:\! A_{n-1,\,k}(\si_{n,\,n-1}\,y)\cdot \si_{n,\,n-1}^2 + O(\bar \eps^{2^{n-1}}) .
\end{aligned}
\end{equation*}
Then 
\begin{equation*}
\begin{aligned}
\| \:\! A_{n,\,k}' \| &\leq \ \| \:\! A_{n-1,\,k}' \| \cdot \|\:\! \si_{n,\,n-1} \|^3 + 3 \| \:\! A_{n,\,k}\| + 3 \| \,A_{n-1,\,k}\| \cdot \|\, \si_{n,\,n-1} \|^2 + O(\bar \eps^{2^{n-1}}) \\[0.2em]
&\leq \ \| \,A_{n-1,\,k}' \| \cdot \|\, \si_{n,\,n-1} \|^3 + C \|\, \si_{n,\,n-1} \|^2
\end{aligned} \msk
\end{equation*}
for some $ C>0 $. Then $ A_{n,\,k}' \ra 0 $ as $ n \ra \infty $ exponentially fast. Hence, so does $ (R^n_k)'(y) $ exponentially fast.
\end{proof}

\nin Let $ w^1 $ and $ w^2 $ be two points in $ B(R^nF) $ and $ w^j = (x^j,\, y^j,\, z^j) $ for $ j =1,2 $. Let $ \Psi^n_{i,\, {\bf v}}(w^j) = w_i^j $ for $ i \in \N $ and $ j =1,2 $.
\msk

\begin{prop} \label{formal expression of pi-z Psi difference}
Let $ F \in \II(\bar \eps) $. 
Then 
\begin{equation*}
\dot z^1 - \dot z^2 = \pi_z \circ \Psi^n_k(w^1) - \pi_z \circ \Psi^n_k(w^2) =  \si_{n,\,k} \cdot (z^1 - z^2) + \si_{n,\,k} \sum_{i=k}^{n-1} q_i (\si_{n,\,i} \,\bar y) \cdot (y^1 -y^2)
\end{equation*}
where 
$ \bar y $ is in the line segment between $ y_1 $ and $ y_2 $. Moreover,
\begin{equation*}
\sum_{i=k}^{n-1} q_i \circ (\si_{n,\,i} \,\bar y) \cdot (y^1 -y^2) = d_{n,\,k} \cdot (y^1 -y^2) + R^n_k(y^1) - R^n_k(y^2) .
\end{equation*}
\end{prop}

\begin{proof}
Firstly, let us express $ \pi_z \circ \Psi^n_k(w) $. Let $ p_{\,i}(y) $ be $ \de_i (y, f^{-1}_i(y),0) $ in order to simplify expression. 
Let $ \Psi^n_{i,\, {\bf v}}(w) = w_i $ for $ k \leq i \leq n-1 $ and let $ w_i = (x_i, y_i, z_i) $ \footnote{For notational compatibility, let $ \Psi^i_i(B) = B $, that is, $ \Psi^i_i = \id $ and let $ \si_{i,\, i} =1 $ for every $ i \in \N $.}. Let $ w = w_n $. Recall $ \pi_z \circ \psi^{i+1}_i (w_{i+1}) = \si_i\, z_{i+1} + p_{\,i} (\si_i\, y_{i+1}) $. Since $ \Psi^n_k = \psi^{k+1}_k \circ \Psi^n_{k+1} $, we estimate $ z_k $ using recursive formula \msk
\begin{align*}
z_k &= \ \pi_z \circ \Psi^n_k(w) = \pi_z \circ \psi^{k+1}_k(w_{k+1}) \\
&= \ \si_k \cdot z_{k+1} + p_{\,k} (\si_k \cdot y_{k+1}) \\[0.2em]
&= \ \si_k \big[\,\si_{k+1} \cdot z_{k+2} + p_{\,k+1} (\si_{k+1} \cdot y_{k+2})\,\big] + p_{\,k}(\si_k \cdot y_{k+1}) \\[0.2em]
&= \ \si_k \si_{k+1} \cdot z_{k+2} + \si_k \cdot p_{\,k+1} (\si_{k+1} \cdot y_{k+2}) + p_{\,k}(\si_k \cdot y_{k+1}) \\
&\hspace{2in} \vdots \\
&= \ \si_k \si_{k+1} \cdots \si_{n-1} \cdot z + \big[\, \si_k \si_{k+1} \cdots \si_{n-2} \cdot p_{\,n-1}(\si_{n-1} \cdot y) \\
&\qquad + \si_k \si_{k+1} \cdots \si_{n-3} \cdot p_{\,n-2}(\si_{n-2} \cdot y_{n-1}) + \cdots + p_{\,k}(\si_k \cdot y_{k+1}) \,\big] \\[0.6em]
&= \ \si_{n,\, k} \cdot z + \si_{n-1,\, k} \cdot p_{\,n-1}(\si_{n-1} \cdot y) + \si_{n-2,\, k} \cdot p_{\,n-2}(\si_{n-2} \cdot y_{n-1}) + \cdots + p_{\,k}(\si_k \cdot y_{k+1}) \\
&= \ \si_{n,\, k} \cdot z + \sum_{i=k}^{n-1} \si_{i,\, k} \cdot p_{\,i}\:\!(\si_i \cdot y_{i+1})
\end{align*}
where $ \si_{k+1,\, k} = \si_k $. 
Moreover, 
$ H_i \circ \La_i(w) = (\phi_i^{-1}(\si_i w),\ \si_i\, y,\ \bullet ) $ \ for each $ k \leq i \leq n-1 $.  
Thus
\begin{equation*}
\begin{aligned}
\si_{n,\, i} \cdot y = \si_i \cdot y_{i+1} = y_i 
\end{aligned} \msk
\end{equation*} 

\nin Secondly, let us estimate \ssk $ \dot z^1 - \dot z^2 = \pi_z \circ \Psi^n_k(w^1) - \pi_z \circ \Psi^n_k(w^2) $ where $ w^j \in B(R^nF) $ for $ j=1,2 $. Recall the definition of $ q_i(y) $, namely, $ \frac{d}{dy}\,p_{\,i}(y) = q_i(y) $. By the above equation 
and mean value theorem, we obtain that \msk
\begin{align*} 
\dot z^1 - \dot z^2 &= \ \pi_z \circ \Psi^n_k(w^1) - \pi_z \circ \Psi^n_k(w^2) \\
&= \ \si_{n,\, k} \cdot (z^1 - z^2) + \sum_{i=k}^{n-1} \si_{i,\, k} \cdot \big[\,p_{\,i}\:\!(\si_i \cdot y_{i+1}^1) - p_{\,i}\:\!(\si_i \cdot y_{i+1}^2)\,\big]  \\
&= \ \si_{n,\, k} \cdot (z^1 - z^2) + \sum_{i=k}^{n-1} \si_{i,\, k} \cdot \big[\,p_{\,i}\:\!(\si_{n,\, i} \cdot y^1) - p_{\,i}\:\!(\si_{n,\, i} \cdot y^2)\,\big]  \\
&= \ \si_{n,\, k} \cdot (z^1 - z^2) + \sum_{i=k}^{n-1} \si_{i,\, k} \cdot q_{\,i}\circ (\si_{n,\, i}\cdot \bar y )  \cdot \si_{n,\, i+1} \cdot ( y^1 - y^2 ) \\
&= \ \si_{n,\, k} \cdot (z^1 - z^2) + \si_{n,\, k}  \sum_{i=k}^{n-1} q_{\,i}\circ (\si_{n,\, i}\cdot \bar y )  \cdot ( y^1 - y^2 ) \numberthis \label{difference of z-1 and z-2}
\end{align*}
where $ \bar y $ is in the line segment between $ y^1 $ and $ y^2 $ which is contained in $ \pi_y \circ B(R^nF) $. Moreover, by the expression of $ \Psi^n_k $, 
$$ \pi_z \circ \Psi^n_k(w) = \si_{n,\, k}\:\! \big[\,d_{n,\, k}\, y + z + R^n_k(y) \,\big] . $$
Then
\begin{align*} 
\dot z^1 - \dot z^2 &= \ \pi_z \circ \Psi^n_k(w^1) - \pi_z \circ \Psi^n_k(w^2) \\[0.2em]
&= \ \si_{n,\, k}\, \big[\,d_{n,\, k}\, (y^1 -y^2) + (z_1 - z_2) + R^n_k(y^1) -  R^n_k(y^2) \,\big] \\[0.2em]
&= \ \si_{n,\, k} \cdot (z^1 - z^2) + \si_{n,\, k} \cdot \big[\,d_{n,\, k}\, (y^1 -y^2) + R^n_k(y^1) -  R^n_k(y^2) \,\big] . \numberthis \label{difference of z-1 and z-2 2}
\end{align*} 
Hence, taking the limit as $ n \ra \infty $, we obtain
$$  \sum_{i=k}^{n-1} q_{\,i}\circ (\si_{n,\, i}\cdot \bar y )  \cdot ( y^1 - y^2 ) = d_{n,\, k}\, (y^1 -y^2) + R^n_k(y^1) -  R^n_k(y^2) $$
of which convergence is exponentially fast.
\end{proof}
\msk

\begin{cor} \label{distortion parts of z-coordinate of Psi}
Let $ F \in \II(\bar \eps) $. Then
\begin{equation*}
\sum_{i=k}^{n-1} q_i \circ (\pi_y \circ \Psi^n_{i,\,{\bf v}}(w)) = d_{n,\,k} + (R^n_k)'(\pi_y (w))
\end{equation*}
for every $ w \in B(R^nF) $ and for each $ k < n $. Moreover,
$$ \lim_{n \ra \infty}\; \sum_{i=k}^{n-1} q_i \circ (\pi_y \circ \Psi^n_{i,\,{\bf v}}(w)) = d_{*,\,k} . $$
\end{cor}

\begin{proof}
Let us compare the equation \eqref{difference of z-1 and z-2} and \eqref{difference of z-1 and z-2 2} below
\begin{align*}
\sum_{i=k}^{n-1} \si_{i,\, k} \cdot \big[\,p_{\,i}\,(\si_{n,\, i} \cdot y^1) - p_{\,i}\,(\si_{n,\, i} \cdot y^2)\,\big] &\ = \ \si_{n,\, k} \cdot \big[\,d_{n,\, k}\, (y^1 -y^2) + R^n_k(y^1) -  R^n_k(y^2) \,\big] \\ 
\sum_{i=k}^{n-1} \si_{i,\ k} \cdot \frac{p_{\,i}\,(\si_{n,\,i} \cdot y^1) - p_{\,i}\,(\si_{n,\,i} \cdot y^2)}{y^1 -y^2} 
 &\ = \ \si_{n,\, k} \cdot \left[\,d_{n,\, k} + \frac{R^n_k(y^1) -  R^n_k(y^2)}{y^1 -y^2}  \,\right] . \numberthis  \label{slope of secant lines}
\end{align*}
Since both $ y^1 $ and $ y^2 $ are arbitrary, we may choose two points $ y $ and $ y + h $ instead of $ y^1 $ and $ y^2 $. The differentiability of both $ p_i $ and $ R^n_k $ enable us to take the limit of \eqref{slope of secant lines} as $ h \ra 0 $. Then 
\begin{equation*}
\si_{n,\, k} \cdot \sum_{i=k}^{n-1} q_{\,i}\circ (\si_{n,\,i} \cdot y ) =  \si_{n,\, k} \cdot \big[\,d_{n,\, k} + (R^n_k)'(y)\,\big]
\end{equation*}
for every $ y \in \pi_y(B(R^nF)) $. Moreover, $ d_{n,\,k} \ra d_{*,\,k} $ as $ n \ra \infty $ super exponentially fast by Lemma \ref{recursive formula of d, u, and t} and $ (R^n_k)' $ converges to zero exponentially fast by Lemma \ref{exponential smallness of R-n-k}. Hence,
\begin{equation*}
\begin{aligned}
\lim_{n \ra \infty} \;\sum_{i=k}^{n-1} q_i \circ (\pi_y \circ \Psi^n_{i,\,{\bf v}}(w)) &= \ \lim_{n \ra \infty} \big[\,d_{n,\,k} + (R^n_k)'(y)\,\big] \\
 &= \ d_{*,\,k} .
\end{aligned}
\end{equation*}

\end{proof}

\nin Let us collect the estimations of numbers and functions which are used in the following sections. \msk
\begin{enumerate}
\item $ |\:\! t_{k+1,\,k} | $, $ |\:\! u_{k+1,\,k}| $ and $ |\:\! d_{k+1,\,k}| $ are $ O \big(\bar \eps^{2^k} \big) $. \msk
\item $ |\:\! u_{n,\,k}| $ and $ |\:\! d_{n,\,k}| $ are $ O \big(\bar \eps^{2^k} \big) $. \msk
\item $ \si_{n,\,k} = (-\si)^{n-k}(1+O(\rho^k)) $ and $ \alpha_{n,\,k} = \si^{2(n-k)}(1+O(\rho^k)) $ for $ k<n $ and for some $ 0<\rho<1 $. \msk
\item $ \|\:\! R^{k+1}_k \| $ is $ O \big(\bar \eps^{2^k} \big) $. \msk
\item $ \|\:\! R^{n}_k \| $ and $ \|\:\! (R^{n}_k)' \| $ are $ O \big(\si^{n-k} \bar \eps^{2^k} \big) $ by Lemma \ref{exponential smallness of R-n-k}.
\end{enumerate}

\bsk

\section{Unbounded geometry on critical Cantor set}

\subsection{Boxing and bounded geometry}
Recall the {\em pieces} $B^n_{\bf w} \equiv B^n_{\bf w}(F) = \Psi^n_{\bf w}(B)$ on the $n^{th}$ level or $n^{th}$ generation 
. The word, $ {\bf w} = (w_1 \ldots  w_n) \in W^n := \lbrace v, c \rbrace^n  $ has length $ n $. Recall that the map 
$$ {\bf w} = (w_1 \ldots  w_n) \mapsto \sum_{k=0}^{n-1} w_{k+1}2^k $$
is one to one correspondence between words of length $ n $ and the additive group of numbers with base 2 mod $ 2^n $. Let the subset of critical Cantor set on each pieces be $ \OO_{\bf w} \equiv B^n_{\bf w} \cap \OO $. Then by the definition of $ \OO_{\bf w} $, 
we have the following facts.\ssk
\begin{enumerate}
\item $$ \OO_F = \bigcup_{ {\bf w} \in W^n}  \OO_{\bf w} $$
\item $ F(B^n_{\bf w}) \subset B^n_{{\bf w}+1} $  for every $ {\bf w} = (w_1 \ldots  w_n) \in W^n $.\msk
\item $ \diam(B^n_{\bf w}) \leq C \si^n $ for some $ C>0 $ depending only on $ B $ and $ \bar \eps $.\ssk
\end{enumerate}
Then we can define boxing of Cantor set. 
The notations in the following definitions are used in \cite{HLM}. These definitions are adapted to three dimensional maps. 
\msk
\begin{defn} \label{boxing}
Let $ F \in \II(\bar \eps) $. A collection of simply connected sets with interior $ {\bf B}^n = \{ B^n_{\bf w} \Subset \Dom(F) \, | \, {\bf w} \in W^n  \} $ is called {\em boxing}\footnote{Element $ B^n_{\bf w} $ in Definition \ref{boxing} is defined topologically. In other words, this definition does not require that $ B^n_{\bf w} \in {\bf B}^n $ is $ \Psi^n_{\bf w}(B) $.} of $ \OO_F $ if \msk
\begin{enumerate}
\item $ \OO_{\bf w} \Subset B^n_{\bf w} $ for each $ {\bf w} \in W^n $. \msk
\item $ B^n_{\bf w} $ and $ B^n_{\bf w'} $ has disjoint closure if $ {\bf w} \neq {\bf w'}$. \msk
\item $ F(B^n_{\bf w}) \subset B^n_{{\bf w}+1} $  for every $ {\bf w} \in W^n $. \msk
\item Each element of $ {\bf B}^n $ is nested for each $ n $, that is,
$$ B^{n+1}_{{\bf w} \nu} \subset B^n_{\bf w}, \ {\bf w} \in W^n,\ \ \nu \in \{v, c\} . $$
\end{enumerate}
\end{defn} 
\ssk
\nin On the above definition, the elements of boxing are just topological boxes. 
Denote $ \Dom(F_{2d}) $ by $ B_{2d} $ in order to distinguish the domain of three dimensional H\'enon-like map from that of two dimensional one. 
Let the minimal distance between two boxes $ B_1, B_2 $ be the infimum of distance between all elements of each boxes and express this distance to be $ \dist_{\min}(B_1, B_2) $.
\begin{defn}
The (given) boxing $ {\bf B}^n $ has the {\em bounded geometry} if
\begin{align*}
 \dist_{\min}(B^{n+1}_{{\bf w}v}, B^{n+1}_{{\bf w}c}) &\asymp \diam(B^{n+1}_{{\bf w}\nu}) \quad \text{for} \ \nu \in \{v, c\} \\[0.2em]
 \diam(B^n_{\bf w}) &\asymp \diam(B^{n+1}_{{\bf w}\nu}) \quad \text{for} \ \nu \in \{v, c\}
\end{align*}
for all $ {\bf w} \in W^n $ and for all $ n \geq 0 $. 
\end{defn}
\nin Moreover, if the boxing has bounded geometry, then we just call $ \OO_F $ has bounded geometry. If any given boxing does not have bounded geometry, then we call $ \OO_F $ has {\em unbounded geometry}.
\msk

\subsection{Horizontal overlap of two adjacent boxes} \label{Horizontal overlap}

The proof of unbounded geometry of the Cantor set requires to compare diameter of boxes and the minimal distance of two adjacent boxes in the boxing. In order to compare these quantities, we would use the maps, $ \Psi^n_k(w) $ and $ F_k(w) $ with the two points $ w_1 = (x_1, y_1, z_1) $ and $ w_2 = (x_2, y_2, z_2) $ in $ F_n(B) $. 
Let us each successive image of $ w_j $ under $ \Psi^n_k(w) $ and $ F_k(w) $ be $ \dot{w_j} $, $ \ddot{w}_j $ and $ \dddot{w}_j $ for $ j=1,2 $.
\begin{displaymath}
\xymatrix  
{   { w_j}  \ar@{|->}[r]^ {\Psi^n_k} & \dot{w}_j  \ar@{|->}[r]^{F_k}  &\ddot{w}_j \ar@{|->}[r]^{ \Psi^k_0}     & \dddot{w}_j
   }
\end{displaymath}

\nin For example, $ \dot{w}_j = \Psi^n_k(w_j) $ and $ \dot{w}_j = (\dot x_j, \dot y_j, \dot z_j ) $ for  $ j=1,2 $. Let $ S_1 $ and $ S_2 $ be the (path) connected set on $ \R^3 $. If $ \pi_x(\overline{S_1}) \cap \pi_x(\overline{S_2}) $ contains at least two points, then this intersection is called the {\em $ x- $axis overlap} or {\em horizontal overlap} of $ S_1 $ and $ S_2 $. We say $ S_1 $ {\em overlaps} $ S_2 $ on the $ x- $axis or horizontally. Recall $ \si $ is the linear scaling of $ F_* $, the fixed point of renormalization operator and $ \si_k = \si ( 1 + O(\rho^k)) $ for each $ k \in \N $.
\ssk \\
Recall the map $ \Psi^n_k $ from $ B(R^nF) $ to $ B^{n}_{{\bf v}}(R^kF) $ where $ {\bf v} = v^{n-k} \in W^{n-k} $.
\msk
\begin{equation*}
\begin{aligned}
\Psi^n_k(w) = 
\begin{pmatrix}
1 & t_{n,\,k} & u_{n,\,k} \\[0.2em]
& 1 & \\
& d_{n,\,k} & 1
\end{pmatrix}
\begin{pmatrix}
\alpha_{n,\,k} & & \\
& \si_{n,\,k} & \\
& & \si_{n,\,k}
\end{pmatrix}
\begin{pmatrix}
x + S^n_k(w) \\
y \\[0.2em]
z + R^n_k(y)     
\end{pmatrix} 
\end{aligned} \msk
\end{equation*}
where $ \alpha_{n,\,k} = \si^{2(n-k)}(1 + O(\rho^k)) $ and $ \si_{n,\,k} = (-\si)^{n-k}(1 + O(\rho^k)) $. Thus for any $ w \in B(R^nF) $ we have the following equation
\begin{equation*}
\pi_x \circ \Psi^n_k(w) = \alpha_{n,\,k} (x + S^n_k(w)) + \si_{n,\,k} \big( t_{n,\,k}\,y + u_{n,\,k} (z + R^n_k(y)) \big) .
\end{equation*}

\nin Let us find the sufficient condition of horizontal overlapping. Horizontal overlapping means that there exist two points $ w_1 \in B^1_v(R^nF) $ and $ w_2 \in B^1_c(R^nF) $ satisfying the equation
\begin{equation*}
\pi_x \circ \Psi^n_k(w_1) - \pi_x \circ \Psi^n_k(w_2) = 0 .
\end{equation*}
Equivalently,
\begin{equation} \label{equation of x-axis overlap 1}
\begin{aligned}
&\ \alpha_{n,\,k} \Big[ \big(x_1 + S^n_k(w_1) \big) - \big(x_2 + S^n_k(w_2)\big) \Big] \\
& \qquad 
+ \si_{n,\,k} \Big[\, t_{n,\,k}(y_1 - y_2) + u_{n,\,k} \,\big \{ z_1 -z_2 + R^n_k(y_1) - R^n_k(y_2) \big \} \Big] =0 .
\end{aligned} 
\end{equation} 
\nin Recall that $ x + S^n_k(w) = v_*(x) + O( \bar \eps^{2^k} + \rho^{n-k}) $ for some $ 0 < \rho < 1 $. 
Since the universal map $ v_*(x) $ is a diffeomorphism and $ | \:\! x_1 - x_2 | = O(1) $, we have the estimation by mean value theorem
$$ |\, x_1 + S^n_k(w_1) - \big(x_2 + S^n_k(w_2)\big) | = O(1) . $$
Recall $ \tau_i $ the tip of $ F_i $ for $ i \in \N $.
\msk
\begin{prop} \label{estimation of t-n-k by b-1}
Let $ F \in {\NN} \cap \II(\bar \eps) $. Let $ b_1 = b_F/b_2 $ where $ b_F $ is the average Jacobian of $ F $ and $ b_2 $ is the universal number defined in Proposition \ref{asymptotic of di-z de-n in Delta}. Then 
$$ b_1^{2^k} \asymp t_{n,\,k} $$
for every $ k + A < n $ where $ A $ is the big enough number depending only on $ b_1 $ and $ \bar \eps $.
\end{prop}
\begin{proof}

Applying Lemma \ref{recursive formula of di-x de-n} to $ \de_k $, we see
\begin{equation*}
\di_x \de_n(w) = \di_x \de_k \circ \Psi^n_{k,\,{\bf c}}(w) - \sum_{i=k}^{n-1} \; q_i \circ \big( \pi_x \circ \Psi^n_{i,\,{\bf c}}(w) \big)
\end{equation*}
for $ k < n $. Let us take $ w= c_{F_n} $ which is the critical point of $ F_n $. Then
\begin{equation*}
\begin{aligned}
\di_x \de_n(c_{F_n}) = \di_x \de_k (c_{F_k}) - \sum_{i=k}^{n-1} \; q_i \circ \big( \pi_x (c_{F_i}) \big) 
= \di_x \de_k (c_{F_k}) - \sum_{i=k}^{n-1} \; q_i \circ \big( \pi_y (\tau_i) \big)
\end{aligned}
\end{equation*}
where $ c_{F_i} $ is the critical point of $ F_i $ for $ k \leq i \leq n-1 $ because $ F_i(c_{F_i}) = \tau_i $. 
Moreover, by Proposition \ref{asymptotic of di-z de-n in Delta}, $ \di_z \de_k = b_2^{2^k}( 1 + O(\rho^n)) $ for some positive $ \rho < 1 $. \ssk The fact that $ F \in {\NN} \cap \II(\bar \eps) $ implies that $ \di_y \de_k(F_k(w)) = - \di_z \de_k(F_k(w)) \cdot \di_x \de_k(w) $. 
Then $ \Jac F_k(\tau_k) $ as follows
\msk
\begin{align*}
\Jac F_k (\tau_k) &= \di_y \eps_k(\tau_k) \cdot \di_z \de_k(\tau_k) - \di_z \eps_k(\tau_k) \cdot \di_y \de_k(\tau_k) \\[0.5em]
&= \big[\, \di_y \eps_k(\tau_k) + \di_z \eps_k(\tau_k) \cdot \di_x \de_k (c_{F_k})\,\big] \cdot \di_z \de_k(\tau_k) \\
&= \big[\, \di_y \eps_k(\tau_k) + \di_z \eps_k(\tau_k) \cdot \sum_{i=k}^{n-1} \; q_i \circ \big( \pi_y (\tau_i) \big)\,\big] \cdot b_2^{2^k}(1 + O(\rho^k)) \\[-0.3em] & \qquad 
+ \di_z \eps_k(\tau_k) \cdot \di_x \de_n(c_{F_n}) \cdot b_2^{2^k}(1 + O(\rho^k)) . \\[-1em]
\end{align*}  
By universality theorem, $ \Jac F_k (\tau_k) = b_F^{2^k} a(\pi_x(\tau_k))(1 + O(\rho^k)) $ and by the definition of $ b_1 $, we obtain that
\begin{equation} \label{eq-asymptotic of di-y eps-k with b-1}
\begin{aligned}
\di_y \eps_k(\tau_k) + \di_z \eps_k(\tau_k) \cdot \sum_{i=k}^{n-1} \; q_i \circ \big( \pi_y (\tau_i) \big) + \di_z \eps_k(\tau_k) \cdot \di_x \de_n(c_{F_n}) = b_1^{2^k} a(\pi_x(\tau_k))(1 + O(\rho^k)) .
\end{aligned}
\end{equation}

\nin Observe that \ssk $ \| \:\! \di_z \eps_k \cdot \di_x \de_n \| = O(\,\bar \eps^{2^k} \bar \eps^{2^n}) $ and $ k < n $. Let us find the sufficient condition satisfying $ \bar \eps^{2^k} \bar \eps^{2^n} \lesssim b_1^{2^k} $. If $ b_1 \geq \bar \eps^2 $, then $ \bar \eps^{2^k} \bar \eps^{2^n} \leq b_1^{2^k} $ for $ n > k $. Assume that $ b_1 < \bar \eps^2 \ll 1 $. Thus
\begin{equation*}
\begin{aligned}
\bar \eps^{2^n} \bar \eps^{2^k} \lesssim  b_1^{2^k} &\Longleftrightarrow \ ( 2^n + 2^k) \log \bar \eps \ \lesssim \ 2^k \log b_1 \\[0.3em]
&\Longleftrightarrow \ 2^n \geq 2^k \left( \frac{\log b_1}{ \log \bar \eps} - 1 \right) + C_0
\end{aligned}
\end{equation*}
for some positive $ C_0>0 $. 
Define $ A $ as follows \msk
\begin{equation}
\begin{aligned}
A = \left\{
\begin{array}{cll}
0 & \quad \text{where} & b_1 \geq \bar \eps^2 \\[0.6em]
C_1\, \log_2 \! \left( \dfrac{\log b_1}{ \log \bar \eps} - 1 \right) & \quad \text{where} & b_1 < \bar \eps^2 
\end{array} \right.
\end{aligned} \msk
\end{equation}
Thus if $ n \geq k + A $, then $ \bar \eps^{2^k} \bar \eps^{2^k} \lesssim b_1^{2^k} $. 
Recall that $ H_k^{-1} \circ \La_k^{-1} = \Psi^{k+1}_k $. \ssk Thus let us compare each components of derivatives , $ D(\Psi^{k+1}_k)^{-1}(\tau_k) = (D^{k+1}_k)^{-1} = D(\La_k \circ H_k)(\tau_k) $. Then comparison of the second and the third column of the matrices shows that
\msk
\begin{align*}
\frac{\alpha_{k+1,\,k}}{\si_{k+1,\,k}} \cdot \di_y \eps_k(\tau_k) &= t_{k+1,\,k} - u_{k+1,\,k}\;\! d_{k+1,\,k} \\[0.2em]
\frac{\alpha_{k+1,\,k}}{\si_{k+1,\,k}} \cdot \di_z \eps_k(\tau_k) &= u_{k+1,\,k}, \qquad q_i \circ \big( \pi_y (\tau_i) \big) = d_{i+1,\,i} . \\[-1em]
\end{align*} 
The equation, $ d_{n,\,k+1} = \sum_{i=k+1}^n d_{i+1,\,i} $ holds by Lemma \ref{recursive formula of d, u, and t}. Then \msk
\begin{align*}
t_{k+1,\,k} + u_{k+1,\,k}\cdot d_{n,\,k+1} &= \ t_{k+1,\,k} - u_{k+1,\,k}\cdot d_{k+1,\,k} + u_{k+1,\,k}\cdot d_{n,\,k} \\
&= \ t_{k+1,\,k} - u_{k+1,\,k}\cdot d_{k+1,\,k} + u_{k+1,\,k} \sum_{i=k}^{n-1} d_{i+1,\,i} \\[-0.4em]
&= \ \frac{\alpha_{k+1,\,k}}{\si_{k+1,\,k}} \,\Big[\, \di_y \eps_k(\tau_k) + \di_z \eps_k(\tau_k) \sum_{i=k}^{n-1} \; q_i \circ \big( \pi_y (\tau_i) \big) \,\Big] \\[0.2em]
&= \ b_1^{2^k} \cdot \big[\,\si \cdot a(\pi_x(\tau_k))(1 + O(\rho^k)) + C_1 \,\big] \\[-1em]
\end{align*} 
for some $ C_1>0 $ where $ n \geq k + A $. Hence, by Lemma \ref{recursive formula of d, u, and t} again, $ t_{n,\,k} $ is as follows
\begin{equation*}
\begin{aligned}
t_{n,\,k} & \asymp \sum_{i=k}^{n-1} \si^{i-k}\,\big[\,t_{i+1,\,i} + u_{i+1,\,i} \cdot d_{n,\,i+1} \,\big](1 + O(\rho^k)) \\[0.2em]
& \asymp t_{k+1,\,k} + u_{k+1,\,k} \cdot d_{n,\,k+1} \\[0.4em]
& \asymp b_1^{2^k} .
\end{aligned}
\end{equation*}

\end{proof}

\nin Let us choose two points $ w'_1 $ and $ w'_2 $ in $ F_n(B) $. In particular, we may assume that 
$ w_j \in \OO_{R^nF} $ where $ \pi_z(w_j) =z_j $ for $ j=1,2 $. 
Let $ w'_1 $ and $ w'_2 $ be the pre-image of $ w_1 $ and $ w_2 $ respectively. Then 
\begin{equation} \label{z distance of two points}
\begin{aligned}
| \, z_1 -z_2 | = | \,\de_n(w_1') - \de_n(w_2') | \leq  C\:\!\| D \de_n \| \cdot \| \, w'_1 - w'_2 \| = O(\,\bar \eps^{2^n})
\end{aligned} \msk
\end{equation}
for some $ C>0 $. Thus $ | \:\!z_1 -z_2 | = O(\bar \eps^{2^n}) $.
\msk
\begin{cor} \label{comparison sigma n-k with b-1}
Let $ F \in \NN \cap \II(\bar \eps) $. \ssk Suppose that $ \Psi^n_k(B^{n+1}_v) $ overlaps $ \Psi^n_k(B^{n+1}_c) $ on $ x- $axis. In particular, 
$ \pi_x \circ \Psi^n_k (w_1) = \pi_x \circ \Psi^n_k (w_2) $ where $ w_1 \in B^{n+1}_v \cap \OO_{R^nF} $ and $ w_2 \in B^{n+1}_c \cap \OO_{R^nF} $. Suppose also that $ k $ and $ n $ are big enough and $ n \geq k + A $ where $ A $ is the number defined in Proposition \ref{estimation of t-n-k by b-1}. Then 
$$ \si^{n-k} \asymp b_1^{2^k}  $$
for every big enough $ k $.
\end{cor}
\begin{proof}
Recall the following equation
\msk
\begin{equation*}
\begin{aligned}
\pi_x \circ \Psi^n_k(w) = \alpha_{n,\,k} \big[\,x + S^n_k(w)\,\big] + \si_{n,\,k} \big[\, t_{n,\,k}\,y + u_{n,\,k} (z + R^n_k(y)) \,\big] .
\end{aligned} \msk
\end{equation*}
Recall 
$ x + S^n_k(w) = v_*(x) + O(\bar \eps^{2^k} + \rho^{n-k}) $ where $ v_*(x) $ is a diffeomorphism and
$$ |\, v_*(x_1) - v_*(x_2) | = |\,v_*'(\bar x)\cdot (x_1 - x_2) | \geq C_0 > 0  $$
where $ \bar x $ is in the line segment between $ x_1 $ and $ x_2 $. Thus 
\msk
\begin{align*}
 \dot x_1 - \dot x_2 
&= \ \alpha_{n,\,k} \big[ \big(x_1 + S^n_k(w_1) \big) - \big(x_2 + S^n_k(w_2)\big) \big] \\
& \qquad 
+ \si_{n,\,k} \big[\, t_{n,\,k}(y_1 - y_2) + u_{n,\,k} \,\big \{ z_1 -z_2 + R^n_k(y_1) - R^n_k(y_2) \big \} \big] \\[0.3em]
&= \ \alpha_{n,\,k} \big[\,v_*'(\bar x)\cdot (x_1 - x_2) + O(\bar \eps^{2^k} + \rho^{n-k}) \big] \\
&\qquad + \si_{n,\,k} \big[\,t_{n,\,k}(y_1 - y_2) + u_{n,\,k} \,\big \{ z_1 -z_2 + (R^n_k)'(\bar y)\cdot (y_1 - y_2) \big \} \big] . \\[-1em]
\end{align*} 

\nin Then by Proposition \ref{estimation of t-n-k by b-1} and the estimations in the end of Section \,\ref{Horizontal overlap}, we obtain that
\msk 
\begin{align*}
\big| \;\! \dot x_1 - \dot x_2 \big| &= \Big|\: C_3\:\! \si^{2(n-k)} + \si^{n-k} \,\big[\, C_4 \;\! b_1^{2^k} + C_5\;\! \bar \eps^{2^k} \big( \,\bar \eps^{2^n} + \si^{n-k} \bar \eps^{2^k} \big) \,\big] \:\!\Big| \\[0.2em]
&= \Big|\:\si^{2(n-k)}\,\big[\,C_3 + C_4\,\bar \eps^{2^{k+1}} \,\big] + \si^{n-k} \,\big[\, C_4 \;\! b_1^{2^k} + C_5\;\! \bar \eps^{2^k} \bar \eps^{2^n} \,\big] \:\!\Big| \\[0.4em]   \numberthis  \label{lower estimate of x dot distance}
&\leq \;C_5\,\si^{2(n-k)} + C_6\,\si^{n-k} \,\big[\, b_1^{2^k} + \bar \eps^{2^k} \bar \eps^{2^n} \,\big] \\[-1em]
\end{align*}  
for some constants $ C_3 $, $ C_4 $, and $ C_5 $, which do not have to be positive. \ssk Let us take big enough $ n $ such that the condition $ b_1^{2^k} \gtrsim \bar \eps^{2^k} \bar \eps^{2^n} $ is satisfied 
by Proposition \ref{estimation of t-n-k by b-1}. However, $ \dot x_1 - \dot x_2 = 0 $ is the horizontal overlapping assumption. Hence,
\begin{equation*}
\begin{aligned}
\si^{2(n-k)} \asymp \si^{n-k} \, b_1^{2^k} .
\end{aligned}
\end{equation*}
\end{proof}

\msk

\subsection{Unbounded geometry on the critical Cantor set}

\comm{************
Recall the coordinate change map at the tip, $ D_k = D\Psi_k(\tau_{F_k}) $. Then \ssk
\begin{align*}
D_k = 
\begin{pmatrix}
1 & t_k & u_k \\
& 1 & \\
& d_k & 1
\end{pmatrix}
 \cdot
\begin{pmatrix}
\alpha_k & & \\
& \si_k & \\
& & \si_k
\end{pmatrix} = DH_k^{-1}(\tau_{F_k}) 
\end{align*}

\msk
****************}

Let us assume that the $ x- $axis overlapping of two boxes, $ \Psi^n_{k,\,{\bf v}}(B^{n+1}_v) $ and $ \Psi^n_{k,\,{\bf v}}(B^{n+1}_c) $. Under this assumption, we can measure the upper bound the minimal distances of two adjacent boxes $ \dist_{\min}(B^{n}_{{\bf w}v}, B^{n}_{{\bf w}c}\,) $, where $ B^{n}_{{\bf w}v} $ and $ B^{n}_{{\bf w}c} $ are the image of $ B^{n+1}_v $ and $ B^{n+1}_c $ under $ \Psi^k_{0,\,{\bf v}} \circ F_k \circ \Psi^n_{k,\,{\bf v}} $ respectively. Compare this minimal distance with the lower bound of the diameter of the above boxes. Then Cantor attractor has unbounded geometry. Moreover, this result is only related to the universal constant $ b_1 $ (Theorem \ref{unbounded geometry with b-1 a.e.}).

\msk
\begin{lem} \label{upper bound the distance of boxes}
Let $ F $ be the H\'enon-like diffeomorphism in $ {\NN} \cap \II(\bar \eps) $. Suppose that $ B^{n}_{{\bf v}v}(R^kF) $ overlaps $ B^{n}_{{\bf v}c}(R^kF) $ 
on the $ x- $axis where the word $ {\bf v} = v^{n-k} \in W^{n-k} $. Then 
$$ \dist_{\min}(B^{n}_{{\bf w}v}, B^{n}_{{\bf w}c}\,) \leq C \,\Big[\, \si^{2k} \si^{n-k} b_1^{2^k} + \si^{2k} \si^{2(n-k)} \bar \eps^{2^k} + \si^{k} \si^{2(n-k)} b_2^{2^k} \,\Big]  $$
where $ {\bf w} = v^kc\,v^{n-k-1} \in W^n $ for some $ C>0 $ and sufficiently big $ k $.
\end{lem}
\begin{proof}
Recall the map $ \Psi^n_k $ from $ B(R^nF) $ to $ B^{n-k}_{{\bf v}}(R^kF) $ 
\msk
\begin{equation*}
\begin{aligned}
\Psi^n_k(w) = 
\begin{pmatrix}
1 & t_{n,\,k} & u_{n,\,k} \\[0.2em]
& 1 & \\
& d_{n,\,k} & 1
\end{pmatrix}
\begin{pmatrix}
\alpha_{n,\,k} & & \\
& \si_{n,\,k} & \\
& & \si_{n,\,k}
\end{pmatrix}
\begin{pmatrix}
x + S^n_k(w) \\
y \\[0.2em]
z + R^n_k(y)
\end{pmatrix} .
\end{aligned} \msk
\end{equation*}
where $ {\bf v} = v^{n-k} \in W^{n-k} $. Let us choose two different points as follows
$$ w_1 = (x_1,\,y_1,\,z_1) \in B^1_v(R^nF) \cap \OO_{R^nF}, \quad w_2 = (x_2,\,y_2,\,z_2) \in B^1_c(R^nF) \cap \OO_{R^nF} . $$
Then by the above expression of $ \Psi^n_k $ and overlapping on the $ x- $axis, we may assume the following equations
\begin{align*}
\dot x_1 - \dot x_2 &= \ 0 \\
\dot y_1 - \dot y_2  &= \ \si_{n,\,k} (y_1 -y_2) \\
\dot z_1  - \dot z_2  &= \ \si_{n,\,k} \big[\, d_{n,\,k}\,(y_1 -y_2) + z_1 -z_2 + R^n_k(y_1) - R^n_k(y_2)\, \big] . \\[-1em]
\end{align*} 
\nin Moreover, definitions of $ F_k $ and $ \Psi^k_{0,\,{\bf v}} $ implies that 
$$ \dddot y_1 - \dddot y_2 = \si_{k,\,0}\cdot (\ddot y_1 - \ddot y_2) = \si_{k,\,0}\cdot(\dot x_1 - \dot x_2) =0 . $$
By mean value theorem and the fact that $ (\ddot x_j,\, \ddot y_j,\, \ddot z_j ) = R^kF (\dot x_j,\, \dot y_j,\, \dot z_j ) $ for $ j =1,2 $, we obtain the following equations \msk
\begin{align*}
\ddot{x}_1 - \ddot{x}_2 
&= \ f_k(\dot x_1) - \eps_k(\dot w_1) - \big[\,f_k(\dot x_2) - \eps_k(\dot w_2)\,\big] \\[0.4em]
&= \ - \eps_k(\dot w_1) + \eps_k(\dot w_2)\\[0.4em]
&= \ - \di_y \eps_k(\eta) \cdot (\dot y_1 - \dot y_2 ) - \di_z \eps_k(\eta) \cdot (\dot z_1 - \dot z_2 ) \\[0.4em]
&= \ - \di_y \eps_k(\eta)  \cdot  \si_{n,\,k} (y_1 -y_2) 
 \\ & \qquad \qquad
- \di_z \eps_k(\eta) \cdot \si_{n,\,k} \big[\, d_{n,\,k}(y_1 -y_2) + z_1 -z_2 + R^n_k(y_1) - R^n_k(y_2)\, \big] \\
%
(*)\ &= \ -  \di_y \eps_k(\eta)  \cdot  \si_{n,\,k} (y_1 -y_2)  - \di_z \eps_k(\eta) \cdot \si_{n,\,k} \ \sum_{i=k}^{n-1} q_i \circ (\si_{n,\,i}\, \bar y) \cdot (y_1 -y_2) \\[-0.6em]
&\qquad \qquad - \di_z \eps_k(\eta) \cdot \si_{n,\,k} (z_1 -z_2) \\
&= \ - \Big[\, \di_y \eps_k(\eta) + \di_z \eps_k(\eta) \cdot \sum_{i=k}^{n-1} q_i \circ (\si_{n,\,i}\, \bar y) \,\Big] \cdot \si_{n,\,k} \, (y_1 -y_2) - \di_z \eps_k(\eta) \cdot \si_{n,\,k} \, (z_1 -z_2)        \numberthis        \label{distance of ddot x-1 and ddot x-2}            
\end{align*} 
where $ \eta $ is some point in the line segment between $ \dot{w}_1 $ and $ \dot{w}_2 $ in $ \Psi^n_k(B) $ and $ \bar y $ is in the line segment between $ y_1 $ and $ y_2 $. The second last equation $ (*) $ is involved with Proposition \ref{formal expression of pi-z Psi difference}. Recall that $ | \;\! y_1 - y_2 | \asymp 1 $ and $ | \;\! z_1 - z_2 | = O\big(\bar \eps^{2^n}\big) $ because every point in the critical Cantor set, $ \OO_{F_n} $ has its inverse image under $ F_n $. Thus by Lemma \ref{di-y eps and q asymptotic from n to k}, we obtain that
\begin{equation} \label{ddot x distance estimation}
\begin{aligned}
| \, \ddot x_1 - \ddot x_2 | \leq C_1\, \si^{n-k}\big[\, b_1^{2^k} + \bar \eps^{2^k} \si^{n-k} + \bar \eps^{2^k}\bar \eps^{2^n} \,\big] .
\end{aligned}\msk
\end{equation}
Similarly, we have 
\msk
\begin{align*}
\ddot{z}_1 - \ddot{z}_2 
&= \ \de_k(\dot{w}_1) - \de_k(\dot{w}_{\,2}) \\[0.3em]
&= \ \di_y \de_k(\zeta) \cdot (\dot y_1 - \dot y_2 ) + \di_z \de_k(\zeta) \cdot (\dot z_1 - \dot z_2 ) \\[0.4em] \numberthis   \label{distance of ddot z-1 and ddot z-2}
&= \ \di_y \de_k(\zeta) \cdot \si_{n,\,k} (y_1 -y_2) \\[0.2em] & \qquad 
+ \di_z \de_k(\zeta) \cdot \si_{n,\,k} \big[\, d_{n,\,k}(y_1 -y_2) + z_1 -z_2 + R^n_k(y_1) - R^n_k(y_2)\, \big] \\
&= \ \di_y \de_k(\zeta) \cdot \si_{n,\,k} \, (y_1 -y_2) + \di_z \de_k(\zeta) \cdot \si_{n,\,k} \; \sum_{i=k}^{n-1} q_i \circ (\si_{n,\,i}\, \bar y) \cdot (y_1 -y_2) \\[-0.5em]
&\qquad + \di_z \de_k(\zeta) \cdot \si_{n,\,k} \, (z_1 -z_2) \\
&= \ \Big[\, \di_y \de_k(\zeta) + \di_z \de_k(\zeta) \; \sum_{i=k}^{n-1} q_i \circ (\si_{n,\,i}\, \bar y) \,\Big] \cdot \si_{n,\,k} \, (y_1 -y_2) + \di_z \de_k(\zeta) \cdot \si_{n,\,k} \, (z_1 -z_2) 
\end{align*}
where $ \zeta $ is some point in the line segment between $ \dot{w}_1 $ and $ \dot{w}_2 $ in $ \Psi^n_k(B) $. By Lemma \ref{asymptotic of di-y de-n in Delta}, the upper bounds of $ | \:\!\ddot z_1 - \ddot z_2 | $ is 
\begin{equation} \label{ddot z distance estimation}
\begin{aligned}
|\:\! \ddot z_1 - \ddot z_2 | \leq C_2\, \si^{n-k} \big[\, \si^{n-k} b_2^{2^k} + b_2^{2^k} \bar \eps^{2^n} \,\big] . 
\end{aligned}  
\end{equation}
Recall 
\begin{equation*}
\begin{aligned}
\pi_x \circ \Psi^n_k(w) = \alpha_{n,\,k} \big[\,x + S^n_k(w) \,\big] + \si_{n,\,k} \big[\, t_{n,\,k}\,y + u_{n,\,k} (z + R^n_k(y)) \,\big] .
\end{aligned} \msk
\end{equation*}
Then the fact that $ \ddot y_1 - \ddot y_2 = 0 $ implies that
\msk
\begin{align*}
\dddot x_1 - \dddot x_2 &= \ \pi_x \circ \Psi^k_0(\ddot w_1) - \pi_x \circ \Psi^k_0(\ddot w_2) \\[0.3em]
&= \ \alpha_{k,\,0} \big[\,(\ddot x_1 + S^k_0( \ddot w_1)) - (\ddot x_2 + S^k_0( \ddot w_2)) \,\big] \\
&\qquad + \si_{k,\,0}\, \big[\, t_{k,\,0}\,(\ddot y_1 - \ddot y_2 ) + u_{k,\,0} \big( \ddot z_1 -\ddot z_2 + R^k_0( \ddot y_1) - R^k_0( \ddot y_2) \big)\, \big] \\[0.3em]
&= \ \alpha_{k,\,0} \big[\, v_*'(\bar x) + O(\bar \eps + \rho^k) \,\big](\ddot x_1 - \ddot x_2) + \si_{k,\,0} \cdot u_{k,\,0}\,( \ddot z_1 -\ddot z_2) \\[-1em]
\end{align*}  
where $ \bar x $ is some point in the line segment between $ \ddot x_1 $ and $ \ddot x_2 $. Moreover,
\msk
\begin{align*}
\dddot z_1 - \dddot z_2 &= \ \pi_z \circ \Psi^k_0(\ddot w_1) - \pi_z \circ \Psi^k_0(\ddot w_2) \\[0.2em]
&= \ \si_{k,\,0}\,(\ddot z_1 - \ddot z_2 ) +  \si_{k,\,0} \big[\, d_{k,\,0}(\ddot y_1 - \ddot y_2) + R^n_k(\ddot y_1) - R^n_k(\ddot y_2)\, \big] \phantom{**\;\;} 
\\[0.3em]
&= \ \si_{k,\,0}\,(\ddot z_1 - \ddot z_2 ) . \\[-1em]
\end{align*}  
Let us apply the estimations in \eqref{ddot x distance estimation} and \eqref{ddot z distance estimation} to $ \dddot x_1 - \dddot x_2 $ and $ \dddot z_1 - \dddot z_2 $. Then the minimal distance is bounded above as follows
\msk
\begin{align*}
\dist_{\min}(B^{n}_{{\bf w}v}, B^{n}_{{\bf w}c}\,) &\leq \ | \, \dddot x_1 - \dddot x_2 | + | \,\dddot z_1 - \dddot z_2 | \\[0.2em]
&\leq \ \big[\,\si^{2k} \cdot | \,\ddot x_1 - \ddot x_2 | \cdot v_*(\bar x) + \si^k \cdot (1 + u_{k,\,0})\,| \,\ddot z_1 -\ddot z_2 |\,\big] (1 + O(\rho^k)) \\[0.2em]
&\leq \ C_3 \, \si^{2k} \si^{n-k}\,\big[\, b_1^{2^k} + \bar \eps^{2^k}\si^{n-k} + \bar \eps^{2^k} \bar \eps^{2^n} \,\big] + C_4 \, \si^k \si^{n-k} \,\big[ \si^{n-k} b_2^{2^k} + b_2^{2^k} \bar \eps^{2^n} \,\big] \\[0.2em]    \numberthis   \label{intermediate minimal distance between regions}
&\leq \ C_3\,\Big[\, \si^{2k} \si^{n-k} b_1^{2^k} + \si^{2k} \si^{2(n-k)} \bar \eps^{2^k} \,\Big]+ C_4 \, \si^k \si^{2(n-k)} b_2^{2^k} \\[-1em]
\end{align*}  
for some positive numbers $ C_3 $ and $ C_4 $. Hence, the estimation \eqref{intermediate minimal distance between regions} is refined as follows
\ssk
\begin{equation} \label{minimal distance between domains}
\begin{aligned}
\dist_{\min}(B^{n}_{{\bf w}v}, B^{n}_{{\bf w}c}\,) &\leq \ | \, \dddot x_1 - \dddot x_2 | + | \,\dddot z_1 - \dddot z_2 | \\
&\leq \ C \big[\, \si^{2k} \si^{n-k} b_1^{2^k} + \si^{2k} \si^{2(n-k)} \bar \eps^{2^k} + \si^{k} \si^{2(n-k)} b_2^{2^k} \,\big] 
\end{aligned}
\end{equation}
for some $ C>0 $.
\end{proof}
\msk

\begin{lem} \label{lower bounds of the diameter of boxes}
Let $ F \in {\NN} \cap \II(\bar \eps) $. Then 
$$  \diam (B^{n}_{{\bf w}v}) \geq | \, C_1 \, \si^k \si^{2(n-k)} - C_2\, \si^k \si^{n-k}b_1^{2^k} | $$
\msk
where $ {\bf w} = v^kc\,v^{n-k-1} \in W^n $ and $ n $ is big enough satisfying $ n \geq k + A $ for $ A $ defined in Proposition \ref{estimation of t-n-k by b-1}. 
\end{lem}

\begin{proof}
Let us choose two points 
$$ w_j = (x_j,\,y_j,\,z_j) \in B^1_v(R^nF) \cap \OO_{R^nF} $$
for $ j =1,2 $ 
satisfying $ | \:\! x_1 - x_2| \asymp 1 $ and $ | \:\! y_1 - y_2| = O(1) $. \ssk Thus we may assume that $ | \:\! z_1 - z_2| = O(\bar \eps^{2^n}) $ by the equation \eqref{z distance of two points}. \ssk Recall that the box, $ B^{n}_{{\bf w}v} $ is $ \Psi^n_{0,\,{\bf w}}(B^1_v(R^nF)) $ and that the diameter of the box $ B^{n}_{{\bf w}v} $ is greater than the distance between any two points in $ B^{n}_{{\bf w}v} $. Let $ \dddot w_j = \Psi^k_{0,\,{\bf v}} \circ F_k \circ \Psi^n_{k,\,{\bf v}}(w_j) $, $ \Psi^k_{0,\,{\bf v}}(w_j) = \dot w_j $, and $ F_k(\dot w_j) = \ddot w_j $ for $ j =1,2 $. 
Then
\begin{equation*}
\begin{aligned}
\diam (B^{1}_{v}) = \sup\; \{\,| \;\! w_1 -  w_2 | \;\big| \ w_1,\;  w_2 \in B^{1}_{v} \,\} \asymp 1 . 
\end{aligned} 
\end{equation*} 
We may assume that $ |\;\! x_1 - x_2 | \asymp 1 $ and $ |\;\! y_1 - y_2 | \asymp 1 $ by the appropriate choice of $ w_1 $ and $ w_2 $. The definition of $ F_k $ and $ \Psi^k_0 $ implies that
\msk
\begin{align*}
\diam (B^{n}_{{\bf w}v}) & \geq | \;\!\dddot w_1 - \dddot w_2 |\geq | \;\!\dddot y_1 - \dddot y_2 | \\[0.2em]
&= | \,\si_{k,\,0}\, (\ddot y_1 - \ddot y_2 )| \\[0.2em]
&= | \,\si_{k,\,0}\, (\dot x_1 - \dot x_2 ) | \\[0.2em]
&= \big| \, \si_{k,\,0} \,\big[\,\pi_x \circ \Psi^n_k(w_1) - \pi_x \circ \Psi^n_k(w_2) \,\big] \:\!\big| \\[-1em]
\end{align*} 
for any two points $ \dddot w_1,\, \dddot w_2 \in B^{n}_{{\bf w}v} $. Recall the equation
\msk
\begin{equation*}
\begin{aligned}
\pi_x \circ \Psi^n_k(w) = \alpha_{n,\,k} \big[\,x + S^n_k(w)\,\big] + \si_{n,\,k} \big[\, t_{n,\,k}\,y + u_{n,\,k} (z + R^n_k(y)) \,\big] .
\end{aligned} \msk
\end{equation*}
and recall $ x + S^n_k(w) = v_*(x) + O(\bar \eps^{2^k} + \rho^{n-k}) $. Thus 
$$ |\, v_*(x_1) - v_*(x_2) | = |\,v_*'(\bar x)\cdot (x_1 - x_2) | \geq C_0 > 0  $$
where $ v_*(x) $ is a diffeomorphism and $ \bar x $ is in the line segment between $ x_1 $ and $ x_2 $. Thus 

\begin{align*}
 \dot x_1 - \dot x_2 
&= \ \alpha_{n,\,k} \big[ \big(x_1 + S^n_k(w_1) \big) - \big(x_2 + S^n_k(w_2)\big) \big] \\
& \qquad 
+ \si_{n,\,k} \big[\, t_{n,\,k}(y_1 - y_2) + u_{n,\,k} \,\big \{ z_1 -z_2 + R^n_k(y_1) - R^n_k(y_2) \big \} \big] \\[0.3em]
&= \ \alpha_{n,\,k} \big[\,v_*'(\bar x)\cdot (x_1 - x_2) + O(\bar \eps^{2^k} + \rho^{n-k}) \big] \\
&\qquad + \si_{n,\,k} \big[\,t_{n,\,k}(y_1 - y_2) + u_{n,\,k} \,\big \{ z_1 -z_2 + (R^n_k)'(\bar y)\cdot (y_1 - y_2) \big \} \big]  \\[-1em]
\end{align*}  
\nin Then by Proposition \ref{estimation of t-n-k by b-1} and the estimations in the end of Section \,\ref{Horizontal overlap}, we obtain that
\msk 
\begin{align*}
\big| \;\! \dot x_1 - \dot x_2 \big| &= \Big|\: C_3\:\! \si^{2(n-k)} + \si^{n-k} \,\big[\, C_4 \;\! b_1^{2^k} + C_5\;\! \bar \eps^{2^k} \big( \,\bar \eps^{2^n} + \si^{n-k} \bar \eps^{2^k} \big) \,\big] \:\!\Big| \\[0.2em]   \numberthis    \label{lower estimate of x dot distance 2}
&= \Big|\:\si^{2(n-k)}\,\big[\,C_3 + C_4\,\bar \eps^{2^{k+1}} \,\big] + \si^{n-k} \,\big[\, C_4 \;\! b_1^{2^k} + C_5\;\! \bar \eps^{2^k} \bar \eps^{2^n} \,\big] \:\!\Big|
\end{align*}  
for some constants $ C_3 $, $ C_4 $, and $ C_5 $ which do not have to be positive. \ssk Let us take big enough $ n $ such that the condition $ b_1^{2^k} \gtrsim \bar \eps^{2^k} \bar \eps^{2^n} $ is satisfied. 
$$ n \geq k +A  $$
where $ A $ is the number depending only on $ \bar \eps $ and $ b_1 $ in Proposition \ref{estimation of t-n-k by b-1}. Hence, \ssk
\begin{equation*}
\begin{aligned}
\diam (B^{n}_{{\bf w}}) \geq \big|\;\!\dddot y_1 - \dddot y_2 \big| \geq \big| \,\si_{k,\,0}\, (\dot x_1 - \dot x_2 ) \big| \geq
\big| \, C_1 \, \si^k \si^{2(n-k)} - C_2\, \si^k \si^{n-k}b_1^{2^k} \big|
\end{aligned} \ssk
\end{equation*}
where $ {\bf w} = v^kc\,v^{n-k-1} \in W^n $ for some positive $ C_1 $ and $ C_2 $.
\end{proof}

\msk
\begin{rem}
In the above lemma, we may choose two points $ w_1 $ and $ w_2 $ which maximize $ |\;\!\dot x_1 - \dot x_2| $. Thus we may assume that
$$ \diam (B^{n}_{{\bf w}}) \geq \max \,\{\; C_1 \, \si^k \si^{2(n-k)},\ \ C_2\, \si^k \si^{n-k}b_1^{2^k} \,\} $$
with appropriate positive constants $ C_1 $ and $ C_2 $.
\end{rem} \msk
\nin Horizontal (or $ x- $axis) overlap is only related to the $ x- $coordinates of points $ \dot w_j \equiv \pi_x \circ \Psi^n_k (w_j) $ for $ j=1,2 $ where $ w_1 \in B^1_v(R^nF) \cap \OO_{R^nF} $ and $ w_2  \in B^1_c(R^nF) \cap \OO_{R^nF} $. 
Recall that $ b_2 $ is the number defined in Proposition \ref{asymptotic of di-z de-n in Delta} and $ b_1 $ is defined by the equation, $ b_1b_2 = b_F $ where $ b_F $ is the average Jacobian of $ F $. Recall that $ b_1 $ is also another universal constant by Lemma \ref{Another universal number b_1}.

\begin{prop} \label{generic unbounded geometry with b-1}
Let $ F_{b_1} $ be an element of parametrized space in $ {\NN} \cap \II(\bar \eps) $. If $ \si^{n-k} \asymp b_1^{2^k} $ for infinitely many $ k $ and $ n $, then there exists $ b_1 $ for $ F_{b_1} $ such that $ B^{n}_{{\bf v}v}(R^kF) $ overlaps $ B^{n}_{{\bf v}c}(R^kF) $ on the $ x- $axis where the word $ {\bf v} = v^{n-k} \in W^{n-k} $. Furthermore, $ F $ has no bounded geometry of \,$ \OO_{F} $. 
\end{prop}

\begin{proof}
Let us choose the two points
\begin{equation*}
\begin{aligned}
w_1 = (x_1,\,y_1,\,z_1) \in B^1_v(R^nF) \cap \OO_{R^nF}, \quad w_2 = (x_2,\,y_2,\,z_2) \in B^1_c(R^nF) \cap \OO_{R^nF}
\end{aligned}
\end{equation*}
such that $ | \:\! x_1 - x_2 | \asymp 1 $, $ | \:\! y_1 - y_2 | \asymp 1 $. Recall $ | \:\! z_1 - z_2 | = O(\bar \eps^{2^n}) $. Let $ \dot w_j = (\dot x_j,\, \dot y_j,\, \dot z_j) $ be $ \Psi^n_k(w_j) $ for $ j =1,2 $. Thus
\begin{equation} \label{x coordinate of the image under Psi n-k again}
\begin{aligned}
 \dot x_1 - \dot x_2 
&= \ \alpha_{n,\,k} \Big[ \big(x_1 + S^n_k(w_1) \big) - \big(x_2 + S^n_k(w_2)\big) \Big] \\[0.2em]
& 
+ \si_{n,\,k} \Big[\, t_{n,\,k}(y_1 - y_2) + u_{n,\,k} \big \{ z_1 -z_2 + R^n_k(y_1) - R^n_k(y_2) \big \} \Big] .
\end{aligned} \ssk
\end{equation}
Recall that $ \alpha_{n,\,k} = \si^{2(n-k)}(1 + O(\rho^k)) $,\, $ \si_{n,\,k} = (-\si)^{n-k}(1 + O(\rho^k)) $ and $ x + S^n_k(w) = v_*(x) + O(\bar \eps^{2^k} + \rho^{n-k}) $. Since $ v_* $ is a diffeomorphism and $ | \;\! x_1 - x_2 | \asymp 1 $, $ | \, v_*(x_1) - v_*(x_2) | \asymp 1 $ by mean value theorem. 
Moreover, Proposition \ref{estimation of t-n-k by b-1} implies that
$$ b_1^{2^k} \asymp t_{n,\,k} . $$
In addition to the above estimation, the fact that $ \| (R^n_k)' \| = O\big(\si^{n-k} \bar \eps^{2^k} \big) $ and the estimation in \eqref{lower estimate of x dot distance} implies that 
\msk
\begin{equation*}
\begin{aligned}
\big| \, u_{n,\,k}\, \big[ z_1 -z_2 + R^n_k(y_1) - R^n_k(y_2) \,\big]\,\big| &\leq \big| \, u_{n,\,k}\cdot (z_1 -z_2) \big| + \big| \,(R^n_k)'(\bar y)\cdot (y_1 - y_2) \big| \\[0.2em]
&= O\big(\,\bar \eps^{2^k} \bar \eps^{2^n} \big) + O\big( \si^{n-k}\bar \eps^{2^k}\big) .
\end{aligned} \msk
\end{equation*}
\nin If $ n \geq k + A $, then we express the equation \eqref{x coordinate of the image under Psi n-k again} as follows
\begin{equation*}
\begin{aligned}
\dot x_1 - \dot x_2 = \si^{2(n-k)} \big[\, v_*(x_1) - v_*(x_2) \,\big] \cdot \big[\, 1 + r_{n,\,k}\,b_1^{2^k} (-\si)^{-(n-k)} \big](1 + O(\rho^k))
\end{aligned}
\end{equation*}
where $ \frac{1}{r} \leq r_{n,\,k} \leq r $ for some number $ r $ depends uniformly on $ b_1 $. Let us take $ n $ such that 
$$ \si^{n-k} \asymp b_1^{2^k} . $$
\nin Then we may assume \ssk that $ B^{n-k}_{{\bf v}v}(R^kF) $ overlaps $ B^{n-k}_{{\bf v}c}(R^kF) $ on the $ x- $axis where $ {\bf v} = v^{n-k-1} \in W^{n-k-1} $ for infinitely many big enough $ n-k $. Let us compare the distance of two adjacent boxes and the diameter of box. By Lemma \ref{lower bounds of the diameter of boxes} and Lemma \ref{upper bound the distance of boxes}, we obtain that 
\begin{align*}
 \diam (B^{n}_{{\bf w}v}) &\geq \big| \, C_1 \, \si^k \si^{2(n-k)} - C_2\, \si^k \si^{n-k}b_1^{2^k} \big|  \\[0.3em]
 \dist_{\min}(B^{n}_{{\bf w}v}, B^{n}_{{\bf w}c}) &\leq C_0 \big[\, \si^{2k} \si^{n-k} b_1^{2^k} + \si^{2k} \si^{2(n-k)} \bar \eps^{2^k} + \si^{k} \si^{2(n-k)} b_2^{2^k} \,\big] \\[-1em]
\end{align*}  
where $ {\bf w} = v^kc\,v^{n-k-1} \in W^n $ for some numbers $ C_0 > 0 $ and $ C_1 $ and $ C_2 $. Hence, 
\begin{equation*}
\dist_{\min}(B^{n}_{{\bf w}v}, B^{n}_{{\bf w}c}) \leq C \,\si^k  \diam (B^{n}_{{\bf w}v}) \msk
\end{equation*}
for every sufficiently large $ k \in \N $ and for some $ C>0 $. Then the critical Cantor set has unbounded geometry.
\end{proof}
\msk

\nin Overlapping is almost everywhere property in the sense of Lebesgue in \cite{HLM}. See the following Theorem. 
\begin{thm}[\cite{HLM}]
Given any $ 0<A_0<A_1 $, $ 0<\si <1 $ and any $ p \geq 2 $, the set of parameters $ b \in [0,1] $ for which there are infinitely many $ 0<k<n $ satisfying 
$$ A_0 < \frac{b^{p^k}}{\si^{n-k}} <A_1 $$
is a dense $ G_{\de} $ set with full Lebesgue measure.
\end{thm}
\msk
\nin Then unbounded geometry is almost everywhere property in the parameter set of $ b_1 $ for every fixed $ b_2 $.
\begin{thm} \label{unbounded geometry with b-1 a.e.}
Let $ F_{b_1} $ be an element of parametrized space in $ {\NN} \cap \II(\bar \eps) $ with $ b_1 = b_F/b_2 $. Then there exists a small interval $ [0, b_{\bullet}] $ for which there exists a $ G_{\de} $ subset $ S \subset [0, b_{\bullet}] $ with full Lebesgue measure such that the critical Cantor set, $ \OO_{F_{b_1}} $ has unbounded geometry for all $ b_1 \in S $. 
\end{thm}

\bsk

\section{Non rigidity on critical Cantor set}
Let $ F $ and $ \widetilde F $ be H\'enon-like maps in $ \NN \cap \II(\bar \eps) $. Let the universal number $ b_1 $ and $ \widetilde b_1 $ are for the map $ F $ and $ \widetilde F $ respectively. Non rigidity on critical Cantor set with respect to the universal constant $ b_1 $ means that the homeomorphism between critical Cantor sets, $ \OO_F $ and $ \OO_{\widetilde F} $ is at most $ \alpha- $H\"older continuous with a constant $ \alpha < 1 $ (Theorem \ref{Non rigidity with b-1} below). This kind of non rigidity phenomenon is a generalization of two dimensional one in \cite{CLM}. However, non rigidity of three dimension maps only depends essentially on the universal number, $ b_1 $ from two dimensional H\'enon-like map in three dimension.

\subsection{Bounds of the distance between two points}
Let us consider the box 
$$ B^n_{\bf w} = \Psi^k_0 \circ F_k \circ \Psi^n_k (B) $$ 
where $ B = B(R^nF) $. Since $ \diam B(R^nF) \asymp \diam B^1_v(R^nF) $, 
by Lemma \ref{lower bounds of the diameter of boxes} we have the lower bound of $ \diam B(R^nF) $ as follows
\ssk
\begin{equation} \label{lower bound of the distance of two points}
\begin{aligned}
\diam (B^{n}_{{\bf w}}) \geq \big| \, C_1 \, \si^k \si^{2(n-k)} - C_2\, \si^k \si^{n-k}b_1^{2^k} \big|
\end{aligned} \ssk
\end{equation}
where $ {\bf w} = v^kc\,v^{n-k-1} \in W^n $ for some positive $ C_1 $ and $ C_2 $. 

\begin{lem} \label{upper bound of the distance of two points}
Let $ F \in {\NN} \cap \II(\bar \eps) $. 
Then 
$$  \diam (B^{n}_{{\bf w}}) \leq C \, \big[\,\si^k \si^{2(n-k)} + \si^k \si^{n-k} b_1^{2^k} \,\big] $$
where $ {\bf w} = v^kc\,v^{n-k-1} \in W^n $ for some $ C>0 $.
\end{lem}

\begin{proof}
Recall the map $ \Psi^n_k $ from $ B(R^nF) $ to $ B^{n}_{{\bf v}}(R^kF) $.\msk
\begin{align*}
\Psi^n_k(w) = 
\begin{pmatrix}
1 & t_{n,\,k} & u_{n,\,k} \\[0.2em]
& 1 & \\
& d_{n,\,k} & 1
\end{pmatrix}
\begin{pmatrix}
\alpha_{n,\,k} & & \\
& \si_{n,\,k} & \\
& & \si_{n,\,k}
\end{pmatrix}
\begin{pmatrix}
x + S^n_k(w) \\
y \\[0.2em]
z + R^n_k(y)     \msk
\end{pmatrix} .
\end{align*}
\\
where $ {\bf v} = v^{n-k} \in W^{n-k} $. Let us \ssk choose two points 
$$ w_1 = (x_1,\,y_1,\,z_1) \in B^1_v(R^nF) \cap \OO_{R^nF}, \quad w_2 = (x_2,\,y_2,\,z_2) \in B^1_c(R^nF) \cap \OO_{R^nF}. $$
Recall $ \dot w_j = \Psi^n_k(w_j) $, $ \ddot w_j = F_k(\dot w_j) $ and $ \dddot w_j = \Psi^k_0(\ddot w_j) $ for $ j =1,2 $. 
Observe that $ | \:\! x_1 - x_2 | $ and $ | \:\! y_1 - y_2| $ is $ O(1) $. We may assume that $ | \:\! z_1 -z_2| = O(\bar \eps^{2^n}) $ because $ \OO_{R^nF} $ is a completely invariant set under $ R^nF $. 
By Corollary \ref{comparison sigma n-k with b-1} and the equation \eqref{lower estimate of x dot distance}, we have 
\msk
\begin{align*}
\dot x_1 - \dot x_2 
&= \ \alpha_{n,\,k} \big[ \big(x_1 + S^n_k(w_1) \big) - \big(x_2 + S^n_k(w_2)\big) \big] \\[0.2em]
&\qquad + \si_{n,\,k} \big[\, t_{n,\,k}(y_1 - y_2) + u_{n,\,k} \,\big \{ z_1 -z_2 + R^n_k(y_1) - R^n_k(y_2) \big \} \big] \\[0.4em]
&= \ \alpha_{n,\,k} \big[\,v_*'(\bar x) + O(\bar \eps^{2^k} + \rho^{n-k}) \big] (x_1 - x_2) \\[0.2em]
&\qquad + \si_{n,\,k} \big[\; t_{n,\,k}(y_1 - y_2) + u_{n,\,k} \,\big \{ z_1 -z_2 + R^n_k(y_1) - R^n_k(y_2) \big \} \big] \\[0.5em]
&\leq \ C\, \big[\, \si^{2(n-k)} + \si^{n-k} ( b_1^{2^k} + \bar \eps^{2^k} \bar \eps^{2^n} ) \,\big] \\[-1em]  \numberthis \label{distance of x coordinate under Psi-n-k}
\end{align*} 
for some $ C>0 $. Moreover, \ssk
\begin{align*}
\dot y_1 - \dot y_2 &= \ \si_{n,\,k} (y_1 - y_2) \\[0.3em]
\dot z_1 - \dot z_2 &= \ \si_{n,\,k} \big[\, d_{n,\,k}(y_1 -y_2) + z_1 -z_2 + R^n_k(y_1) - R^n_k(y_2)\, \big] . \\[-1.2em]
\end{align*} 
\nin By the equation \eqref{distance of ddot x-1 and ddot x-2}, we estimate the distance between each coordinates of $ F_k(\dot w_1) $ and $ F_k(\dot w_2) $ as follows \msk
\begin{align*}
\ddot{x}_1 - \ddot{x}_2 &= \ f_k(\dot x_1) - \eps_k(\dot w_1) - [f_k(\dot x_2) - \eps_k(\dot w_2)] \\[0.3em]
&= \ f'_k(\bar x)\cdot(\dot x_1 - \dot x_2) - \eps_k(\dot w_1 ) + \eps_k(\dot w_2)\\[0.4em]
&= \ [\,f'_k(\bar x) - \di_x \eps_k(\eta) \,]\cdot(\dot x_1 - \dot x_2)- \di_y \eps_k(\eta) \cdot (\dot y_1 - \dot y_2 ) - \di_z \eps_k(\eta) \cdot (\dot z_1 - \dot z_2) \\[0.5em]
&= \ [\,f'_k(\bar x) - \di_x \eps_k(\eta) \,]\cdot(\dot x_1 - \dot x_2) - \di_y \eps_k(\eta)  \cdot  \si_{n,\,k} (y_1 -y_2) \\
& \qquad - \di_z \eps_k(\eta) \cdot \si_{n,\,k} \big[\, d_{n,\,k}(y_1 -y_2) + z_1 -z_2 + R^n_k(y_1) - R^n_k(y_2)\, \big] \\[0.6em]
\ddot{y}_1 - \ddot{y}_2 &= \ \dot x_1 - \dot x_2  \\[0.6em]
\ddot{z}_1 - \ddot{z}_2 &= \ \de_k(\dot{w}_1) - \de_k(\dot{w}_2) \\[0.4em]
&= \ \di_x \de_k(\zeta) \cdot(\dot x_1 - \dot x_2) + \di_y \de_k(\zeta) \cdot (\dot y_1 - \dot y_2 ) + \di_z \de_k(\zeta) \cdot (\dot z_1 - \dot z_2 ) \\[0.4em]
&= \ \di_x \de_k(\zeta) \cdot(\dot x_1 - \dot x_2) + \di_y \de_k(\zeta) \cdot \si_{n,\,k} (y_1 -y_2) \\
& \qquad + \di_z \de_k(\zeta) \cdot \si_{n,\,k} \big[\, d_{n,\,k}(y_1 -y_2) + z_1 -z_2 + R^n_k(y_1) - R^n_k(y_2)\, \big] \\[-1em]
\end{align*} 
where $ \eta $ and $ \zeta $ are some points in the line segment between $ \dot{w}_1 $ and $ \dot{w}_2 $ in $ \Psi^n_k(B) $. The equations \eqref{distance of ddot x-1 and ddot x-2} and \eqref{ddot x distance estimation} in Lemma \ref{upper bound the distance of boxes} implies that
\begin{equation} \label{ddot x distance upper bound}
\begin{aligned}
| \, \ddot x_1 - \ddot x_2 | \leq |\,f'_k(\bar x) - \di_x \eps_k(\eta) \,| \cdot | \, \dot x_1 - \dot x_2 | + C_2 \, \si^{n-k}\big[\, b_1^{2^k} + \bar \eps^{2^k}\si^{n-k} + \bar \eps^{2^k}\bar \eps^{2^n} \,\big] 
\end{aligned} \msk
\end{equation}
and the equations \eqref{distance of ddot z-1 and ddot z-2} and \eqref{ddot z distance estimation} in the same Lemma implies that
\msk
\begin{equation} \label{ddot z distance upper bound}
\begin{aligned}
| \, \ddot z_1 - \ddot z_2 | \leq |\,\di_x \de_k(\zeta) | \cdot | \, \dot x_1 - \dot x_2 | + C_3\,\si^{n-k} \big[\, \si^{n-k} b_2^{2^k} + b_2^{2^k} \bar \eps^{2^n} \,\big] .
\end{aligned} \msk
\end{equation}
Then the difference of each coordinates of $ \Psi^k_0(\ddot w_1) $ and $ \Psi^k_0(\ddot w_2) $ as follows
\msk
\begin{align*}
\dddot x_1 - \dddot x_2 &= \ \pi_x \circ \Psi^k_0(\ddot w_1) - \pi_x \circ \Psi^k_0(\ddot w_2) \\[0.4em]
&= \ \alpha_{k,\,0} \big[\,(\ddot x_1 + S^k_0( \ddot w_1)) - (\ddot x_2 + S^k_0( \ddot w_2)) \,\big] \\[0.2em]
&\qquad + \si_{k,\,0}\, \big[\, t_{k,\,0}\,(\ddot y_1 - \ddot y_2 ) + u_{k,\,0} \big( \ddot z_1 -\ddot z_2 + R^k_0( \ddot y_1) - R^k_0( \ddot y_2) \big)\, \big] \\[0.4em]   \numberthis   \label{dddot x distance upper bound} 
&= \ \alpha_{k,\,0} \big[\, v_*'(\bar x) + O(\bar \eps + \rho^k) \,\big](\ddot x_1 - \ddot x_2) + \si_{k,\,0} \cdot u_{k,\,0}\,( \ddot z_1 -\ddot z_2) \\[0.2em]  
&\qquad + \si_{k,\,0}\, \big[\, t_{k,\,0}\,(\dot x_1 - \dot x_2 ) + u_{k,\,0} \big( R^k_0(\dot x_1) - R^k_0( \dot x_2) \big)\, \big] \\[0.5em]   \numberthis   \label{dddot y distance upper bound} 
\dddot y_1 - \dddot y_2 &= \ \si_{k,\,0}\,(\ddot y_1 - \ddot y_2 ) = \si_{k,\,0}\,(\dot x_1 - \dot x_2 )   \\[0.5em]
\dddot z_1 - \dddot z_2 &= \ \pi_z \circ \Psi^k_0(\ddot w_1) - \pi_z \circ \Psi^k_0(\ddot w_2) \\[0.3em]
&= \ \si_{k,\,0}\,(\ddot z_1 - \ddot z_2 ) +  \si_{k,\,0} \big[\, d_{k,\,0}(\ddot y_1 - \ddot y_2) + R^n_k(\ddot y_1) - R^n_k(\ddot y_2)\, \big] 
\\[0.3em]    \numberthis   \label{dddot z distance upper bound}
&= \ \si_{k,\,0}\,(\ddot z_1 - \ddot z_2 ) + \si_{k,\,0} \big[\, d_{k,\,0}(\dot x_1 - \dot x_2) + R^n_k(\dot x_1) - R^n_k(\dot x_2)\, \big] . \\[-1em]
\end{align*} 
Let us calculate a upper bound of the distance, $ |\, \dddot w_1 - \dddot w_2 | $. Applying the estimation \eqref{dddot x distance upper bound}, \eqref{dddot y distance upper bound}  and \eqref{dddot z distance upper bound}, we obtain that \ssk
\begin{align*}
& \quad \ \ |\, \dddot w_1 - \dddot w_2 | \leq \ |\,\dddot x_1 - \dddot x_2 | + |\,\dddot y_1 - \dddot y_2 | + |\,\dddot z_1 - \dddot z_2 | \\[0.4em]
& \leq \ \big|\, \alpha_{k,\,0} \big[\, v_*'(\bar x) + O(\bar \eps + \rho^k) \,\big](\ddot x_1 - \ddot x_2) + \si_{k,\,0} \cdot u_{k,\,0}\,( \ddot z_1 -\ddot z_2) \\[0.2em]
&\qquad + \si_{k,\,0}\, \big[\, t_{k,\,0}\,(\dot x_1 - \dot x_2 ) + u_{k,\,0} \big( R^k_0(\dot x_1) - R^k_0( \dot x_2) \big)\, \big] \:\!\big| 
\\[0.2em]
&\qquad + \big|\,\si_{k,\,0}\,(\dot x_1 - \dot x_2 ) \,\big| + \big|\,\si_{k,\,0}\,(\ddot z_1 - \ddot z_2 ) + \si_{k,\,0} \big[\, d_{k,\,0}(\dot x_1 - \dot x_2) + R^n_k(\dot x_1) - R^n_k(\dot x_2)\, \big] \:\!\big| \\[0.4em]
\quad \ \; & \leq \ \big|\,\alpha_{k,\,0} \big[\, v_*'(\bar x) + O(\bar \eps + \rho^k) \,\big]\cdot |\,f'_k(\bar x) - \di_x \eps_k(\eta) \,| \cdot | \, \dot x_1 - \dot x_2 | \\[0.3em]
&\qquad + \big|\,\alpha_{k,\,0} \big[\, v_*'(\bar x) + O(\bar \eps + \rho^k) \,\big]\cdot C_2 \, \si^{n-k}\big[\, b_1^{2^k} + \bar \eps^{2^k}\si^{n-k} + \bar \eps^{2^k}\bar \eps^{2^n} \,\big]\:\!\big| \\[0.3em] 
&\qquad + \big|\,\si_{k,\,0}\, [\,1 + |\;\! t_{k,\,0}| + |\;\! d_{k,\,0}| \,]\,(\dot x_1 - \dot x_2 ) \,\big| 
+ \big|\,\si_{k,\,0}\, [\,1 + |\;\! u_{k,\,0}| \,] \:\! \big| \cdot \big| \,(R^n_k)'(\widetilde x) \cdot (\dot x_1 - \dot x_2) \big| \\[0.3em]
&\qquad + \big|\,\si_{k,\,0}\, [\,1 + |\;\! u_{k,\,0}| \,]\:\!\big| \cdot \big[\,\big|\, \di_x \de_k(\zeta)  \cdot 
(\dot x_1 - \dot x_2) \big| + C_3\,\si^{n-k} \big( \si^{n-k} b_2^{2^k} + b_2^{2^k} \bar \eps^{2^n} \big)\:\!\big]  
\end{align*} 
After factoring out $ | \;\! \dot x_1 - \dot x_2 | $\,, this inequality continues as follows \ssk
\begin{align*}
& \leq \ C_4 \,\si^k |\,\dot x_1 - \dot x_2 | + C_5 \, \si^{2k}\si^{n-k}\big[\, b_1^{2^k} + \bar \eps^{2^k}\si^{n-k} + \bar \eps^{2^k}\bar \eps^{2^n} \,\big] \\
& \qquad + C_6\,\si^k\si^{n-k} \big[\, \si^{n-k} b_2^{2^k} + b_2^{2^k} \bar \eps^{2^n} \:\!\big] \\[0.5em]
(*)\ \ & \leq \ C_7\,\si^k \big[\, \si^{2(n-k)} + \si^{n-k} ( b_1^{2^k} + \bar \eps^{2^k} \bar \eps^{2^n} ) \,\big]  \\[0.2em]
& \qquad + C_5 \, \si^{2k}\si^{n-k}\big[\, b_1^{2^k} + \bar \eps^{2^k}\si^{n-k} + \bar \eps^{2^k}\bar \eps^{2^n} \,\big] + C_6\,\si^k\si^{n-k} \big[\, \si^{n-k} b_2^{2^k} + b_2^{2^k} \bar \eps^{2^n} \:\!\big] \hspace{0.85in}
\end{align*} 
for some positive constants, $ C_j $, $ 2 \leq j \leq 7 $ which are independent of $ k $ and $ n $. The second last line, $ (*) $ holds by the estimation \eqref{distance of x coordinate under Psi-n-k} and $ \| (R^n_k)'\| $ at the end of Section \,\ref{Horizontal overlap}. 
Observe that $ \si^n \gg \bar \eps^{2^n} $ for all big enough $ n $. Then the above estimation continues
\msk
\begin{align*}
 & \leq \ C_7\,\si^k \big[\, \si^{2(n-k)} + \si^{n-k} ( b_1^{2^k} + \bar \eps^{2^k} \bar \eps^{2^n} ) \,\big]  \\
 & \qquad 
+ C_5 \, \si^{2k}\si^{n-k}\big[\, b_1^{2^k} + \bar \eps^{2^k}\si^{n-k} + \bar \eps^{2^k}\bar \eps^{2^n} \,\big] + C_8\,\si^k\si^{2(n-k)} b_2^{2^k} \\[0.3em]
 & \leq \ \big( C_7 + C_5\,\si^k\bar \eps^{2^k} + C_6\,b_2^{2^k} \big)\, \si^k \si^{2(n-k)} + \big(C_7 + C_5\,\si^k)\, \si^k\si^{n-k}b_1^{2^k} \hspace{1.3in} \\
& \qquad + \big(C_7 + C_5\,\si^k \big)\, \si^k \si^{n-k} \bar \eps^{2^k}\bar \eps^{2^n} \\[-1em]
\end{align*} 
for some positive constant $ C_8 $. Moreover, if $ n \geq k +A $ where $ A $ is the number defined in Proposition \ref{estimation of t-n-k by b-1}, then 
$$ b_1^{2^k} \gtrsim \bar \eps^{2^k}\bar \eps^{2^n} . $$
Hence,
\begin{equation*}
\diam (B^{n}_{{\bf w}}) \leq C \, \big[\,\si^k \si^{2(n-k)} + \si^k \si^{n-k} b_1^{2^k} \,\big]
\end{equation*}
where $ {\bf w} = v^kc\,v^{n-k-1} \in W^n $ for some $ C>0 $.

\end{proof}

\begin{rem}
Lemma \ref{lower bounds of the diameter of boxes} and Lemma \ref{upper bound of the distance of two points} implies the lower and the upper bounds of $ \diam B_{\bf w} $ where $ B_{\bf w} = \Psi^k_0 \circ F_k \circ \Psi^n_k (B(R^nF)) $ as follows
\begin{equation*}
C_0\,\si^k | \;\! \dot x_1 - \dot x_2 | \leq \diam B_{\bf w} \leq C_1 \,\si^k |\;\! \dot x_1 - \dot x_2 |
\end{equation*}
for every big enough $ k \in \N $, that is, \,$ \diam B_{\bf w} \asymp \si^k |\;\! \dot x_1 - \dot x_2 | $.
\end{rem}

\msk

\subsection{Non rigidity on critical Cantor set with respect to $ b_1 $}
Recall that $ b_1 $ is $ b_F/b_2 $ where $ b_F $ is the average Jacobian of $ F $ and $ b_2 $ is the number defined in Proposition \ref{asymptotic of di-z de-n in Delta}. The number $ \widetilde{b}_1 $ is defined by the similar way for the map $ \widetilde{F} $.
\begin{thm} \label{Non rigidity with b-1}
Let H\'enon-like maps $ F $ and $ \widetilde{F} $ be in $ \NN \cap \II(\bar \eps) $. 
Let $ \phi \colon \OO_{\widetilde{F}} \ra \OO_F $ be a homeomorphism \ssk which conjugate $ F_{\,\OO_F} $ and $ \widetilde{F}_{\,\OO_{\widetilde{F}}} $ and $ \phi(\tau_{\widetilde{F}}) = \tau_{F} $. If \,$ b_1 > \widetilde{b}_1 $, \ssk then the H\"older exponent of $ \phi $ is not greater than $ \displaystyle{\frac{1}{2} \left(1 + \frac{\log b_1}{ \log \widetilde{b}_1} \right)} $.
\end{thm}

\begin{proof}
Let two points $ w_1 $ and $ w_2 $ be in $ B^1_v(R^nF) $ and $ B^1_c(R^nF) $ respectively. Similarly, assume that $ \widetilde {w}_1 $ and $ \widetilde {w}_2 $ are in $ B(R^n\widetilde{F}) $. Let us define
$ \dddot w_j = \Psi^k_0 \circ F_k \circ \Psi^n_k (w_j) $ for $ j=1,2 $. The points $ \dddot {\widetilde {w}}_1 $ and $ \dddot {\widetilde {w}}_2 $ are defined by the similar way. For sufficiently large $ k \in \N $, let us choose $ n $ depending on $ k $ which satisfies the following inequality
\begin{equation*}
\si^{n-k+1} \leq \widetilde{b}_1^{2^k} < \si^{n-k} .
\end{equation*}
Observe that $ b_1^{2^k} \gg \widetilde{b}_1^{2^k} $.
By Lemma \ref{lower bounds of the diameter of boxes} and Lemma \ref{upper bound of the distance of two points}, we have the following inequalities
\msk
\begin{align*}
\dist(\dddot {\widetilde {w}}_1, \dddot{\widetilde{w}}_2) &\leq \ C_0 \, \Big[\,\si^k \si^{2(n-k)} + \si^k \si^{n-k} \,\widetilde{b}_1^{2^k} \,\Big] \ \ \leq \ C_1 \,\si^k \,\widetilde{b}_1^{2^k}\, \widetilde{b}_1^{2^k}  \\[0.2em]
\dist(\dddot w_1, \dddot w_2) &\geq \ \big| \, C_2 \, \si^k \si^{2(n-k)} - C_3\, \si^k \si^{n-k}{b}_1^{2^k} \big| \geq \ C_4 \, \si^k \,\widetilde{b}_1^{2^k}\, b_1^{2^k} \\[-1em]
\end{align*}  
for some positive $ C_j $ where $ j =0,1,2,3 $ and $ 4 $. 
\nin H\"older continuous function, $ \phi $ with the H\"older exponent $ \alpha $ has to satisfy 
\begin{equation*}
\dist(\dddot w_1, \dddot w_2) \leq C \, \big( \dist(\dddot{\widetilde{w}}_1, \dddot{\widetilde{w}}_2) \big)^{\alpha}
\end{equation*}
for some $ C>0 $. Then we see that
\begin{equation*}
\si^k \,\widetilde{b}_1^{2^k}\, {b}_1^{2^k} \leq C \, \Big( \si^k \,\widetilde{b}_1^{2^k}\, \widetilde b_1^{2^k} \Big)^{\alpha}
\end{equation*}
\nin Take the logarithm both sides and divide them by $ 2^k $. After passing the limit, divide both sides by the negative number, $ 2\log \widetilde{b}_1 $. Then the desired upper bound of H\"older exponent is obtained
\begin{align*}
k \log \si + 2^k \log \widetilde{b}_1 + 2^k \log {b}_1 &\leq \log C + \alpha \:\! \big( k \log \si + 2^k \log \widetilde{b}_1 + 2^k \log \widetilde b_1 \big) \\
\frac{k}{2^k} \log \si + \log \widetilde{b}_1 + \log {b}_1 &\leq \frac{1}{2^k} \log C + \alpha \left( \frac{k}{2^k} \log \si + \log \widetilde{b}_1 + \log \widetilde b_1 \right) \\
\log \widetilde{b}_1 + \log {b}_1 &\leq \alpha \cdot 2\log \widetilde b_1 \\
\alpha &\leq \frac{1}{2} \left(1 + \frac{\log b_1}{ \log \widetilde{b}_1} \right) . \\[-1em]
\end{align*}
\end{proof} \msk

\nin In renormalization theory of two dimensional H\'enon-like map, the answer of rigidity problem with average Jacobian is unknown. In other words, the best regularity which the conjugation $ \phi $ should satisfy is not known yet where $ b_F = b_{\widetilde{F}} $ for two dimensional H\'enon-like maps $ F $ and $ \widetilde{F} $. However, the average Jacobian of three dimensional H\'enon-like map in $ \NN \cap \II(\bar \eps) $ less affects rigidity than $ b_1 $. Moreover, in higher dimension, we do not expect rigidity with average Jacobian between Cantor attractors.
\msk
\begin{example}
Let us consider a map $ F $ in $ \II(\bar \eps) $ as follows
\begin{equation*}
F(w) = (f(x) - \eps(x,y),\, x,\, \de(z)) .
\end{equation*}
We call the three dimensional H\'enon-like map satisfying $ \de(w) \equiv \de(z) $ a {\em trivial extension} of two dimensional H\'enon-like map. Let the set of these maps be $ \TT $. 
If $ F \in \TT \cap \II(\bar \eps) $, then the $ n^{th} $ renormalized map of $ F $, $ F_n \equiv R^nF $ is as follows
\msk
\begin{equation*}
\begin{aligned}
F_n(w) = \big(\,f_n(x) - a(x)\,b_1^{2^n}y\,(1 +O(\rho^n)),\ x,\ b_2^{2^n}z\,(1 +O(\rho^n)) \big) 
\end{aligned} \msk
\end{equation*} 
where $ b_1 $ is the average Jacobian of two dimensional map, $ \pi_{xy} \circ F $ and $ b_2 = b_F/b_1 $ for some $ 0<\rho <1 $. Let $ \widetilde{F} $ be another map in $ \TT \cap \II(\bar \eps) $ with the corresponding numbers $ \widetilde{b}_1 $, $ \widetilde{b} $ and $ \widetilde{b}_2 $. By Theorem \ref{Non rigidity with b-1}, if $ b_1 > \widetilde{b}_1 $, the upper bound of H\"older exponent is 
$$ \frac{1}{2} \left(1 + \frac{\log b_1}{ \log \widetilde{b}_1} \right) $$
Let $ \de $ and $ \widetilde{\de} $ be the third coordinate map of $ F $ and $ \widetilde{F} $ respectively. 
Recall that $ b_1 b_2 = b_F $ for every map $ F \in \TT \cap \II(\bar \eps) $ and $ b_2 $ is the contracting rate along the third coordinate. The condition, $ b_2 \neq \widetilde{b}_2 $ may require non rigidity of homeomorphic conjugacy between critical Cantor sets of $ F $ and $ \widetilde{F} $ even if $ b_F = b_{\widetilde{F}} $. Assume that 
$$ b_1b_2 = b = \widetilde{b} = \widetilde{b}_1 \widetilde{b}_2 $$ 
Thus the condition $ b_2 \neq \widetilde{b}_2 $ implies either $ b_1 > \widetilde{b}_1 $ or $ b_1 < \widetilde{b}_1 $. Then Theorem \ref{Non rigidity with b-1} implies the non rigidity between Cantor attractors of $ F $ and $ \widetilde{F} $.
\end{example} \msk


\bsk


\titleformat{\section}[display]{\normalfont\Large\bfseries}{Appendix~\Alph{section}}{12pt}{\Large}

\begin{appendices}


\section{Recursive formula of $ \Jac R^nF $}

Recall the definition of $ H $ and $ H^{-1} $
\msk
\begin{equation*}
\begin{aligned}
H(x,y,z) &= (f(x) - \eps(w),\ y,\ z - \de(y,f^{-1}(y),0)) \\
H^{-1}(x,y,z) &= (\phi^{-1}(w),\ y,\ z + \de(y,f^{-1}(y),0)) .
\end{aligned} 
\end{equation*}
Thus $ \phi^{-1}(x,y,z) $ is the straightening map satisfying $ \phi^{-1} \circ H(w) = x $.
$$ f \circ \phi^{-1}(w) - \eps \circ H^{-1}(w) = x . $$
Then
$$ \phi^{-1}(w) = f^{-1}(x + \eps \circ H(w)) $$
\nin Recall $ \eps \circ H^{-1}(w) = \eps ( \phi^{-1}(w),\ y,\ z + \de(y, f^{-1}(y), 0)) $. Then by the chain rule, each partial derivatives of $ \phi^{-1} $ is as follows \msk
\begin{align*}
\di_x \phi^{-1}(w) &= \ (f^{-1})'(x + \eps \circ H^{-1}(w)) \cdot \big[\, 1 + \di_x \eps \circ H^{-1}(w) \cdot \di_x \phi^{-1}(w) \,\big] \\[0.4em]
\di_y \phi^{-1}(w) &= \ (f^{-1})'(x + \eps \circ H^{-1}(w)) \\[-0.3em]
& \quad \cdot \big[\, \di_x \eps \circ H^{-1}(w) \cdot \di_y \phi^{-1}(w) + \di_y \eps \circ H^{-1}(w) + \di_z \eps \circ H^{-1}(w) \cdot \frac{d}{dy}\, \de(y,f^{-1}(y),0) \,\big] \\[0.2em]
\di_z \phi^{-1}(w) &= \ (f^{-1})'(x + \eps \circ H^{-1}(w)) \cdot \big[\, \di_x \eps \circ H^{-1}(w) \cdot \di_z \phi^{-1}(w) + \di_z \eps \circ H^{-1}(w) \,\big] . \\[-1em]
\end{align*} 
Then
\begin{align*}
\di_x \phi^{-1}(w) &= \ \frac{(f^{-1})'(x + \eps \circ H^{-1}(w))}{1 - (f^{-1})'(x + \eps \circ H^{-1}(w)) \cdot \di_x \eps \circ H^{-1}(w)} \\ 
\di_y \phi^{-1}(w) &= \ \di_x \phi^{-1}(w) \cdot \big[\,\di_y \eps \circ H^{-1}(w) + \di_z \eps \circ H^{-1}(w) \cdot \frac{d}{dy}\, \de(y,f^{-1}(y),0)\,\big] \\ 
\di_z \phi^{-1}(w) &= \ \di_x \phi^{-1}(w) \cdot \di_x \eps \circ H^{-1}(w) \ .\\[-1em]
\end{align*} 


\nin Recall {\em pre-renormalization} of $ F $, $ PRF $ is defined as follows
\begin{equation*}
\begin{aligned}
PRF = H \circ F^2 \circ H^{-1}
\end{aligned}
\end{equation*}
where $ H(w) = (f(x) - \eps(w),\ y ,\ z - \de (y ,f^{-1}(y),0))  $. Recall the renormalized map $ RF $ is defined as $ \La \circ PRF \circ \La^{-1} $ where $ \La(w) = (sx,sy,sz) $ for the appropriate number $ s < -1 $ from the renormalized one dimensional map, $ f(x) $. Denote $ \si_0 = 1/s $. 
Let the first coordinate map of $ H^{-1}(w) $ be $ \phi^{-1}(w) $. Then
\begin{equation*}
H^{-1}(w) = (\phi^{-1}(w),\ y,\ z + \de (y ,f^{-1}(y),0)) .
\end{equation*}
By the direct calculation $ PRF $ is as follows.
\begin{equation*}
PRF(w) = ( f(f(x) - \eps \circ F \circ H^{-1}(w)) - \eps \circ F^2 \circ H^{-1}(w),\; x,\; \de \circ F \circ   H^{-1}(w) - \de (x, f^{-1}(x),0))
\end{equation*}
Let the perturbed part of the first coordinate map of $ PRF $ be Pre\;$ \eps_1(w) $. Let the third coordinate map of $ PRF $ be Pre\;$ \de_1(w) $. Moreover, Pre\;$ \eps_1(w) $ and Pre\;$ \de_1(w) $ is defined as the corresponding parts of $ PR^k\!F $ for each $ k \in \N $. Then relations between Pre\;$ \eps_k(w) $ and $ \eps_k(w) $ and between Pre\;$ \de_k(w) $ and $ \de_k(w) $ respectively are as follows
\begin{equation*}
\begin{aligned}
\textrm{Pre}\; \eps_k(w) = \si_{k-1} \cdot \eps_k \circ \left(\frac{w}{\si_{k-1}} \right) \quad \text{and} \quad 
\textrm{Pre}\; \de_k(w) = \si_{k-1} \cdot \de_k \circ \left(\frac{w}{\si_{k-1}} \right) .
\end{aligned} \msk 
\end{equation*}

\begin{lem} \label{appendix - Dde recursive form}
Let $ F $ be a renormalizable three dimensional H\'enon-like map. Let $ \de_1 $ be the third coordinate of $ RF $, namely, $ \pi_z \circ RF $. Then \msk
\begin{align*}
\di_x \de_1(w) &= \ \big[\,\di_y \de \circ \psi^1_c(w) + \di_z \de \circ \psi^1_c(w) \cdot \di_x \de \circ \psi^1_v(w) \,\big] \cdot \di_x \phi^{-1}(\si_0 w) \\
& \qquad + \di_x \de \circ \psi^1_c(w) - \frac{d}{dx}\;\de(\si_0 x,\, f^{-1}(\si_0 x),\, 0) \\[0.6em]
\di_y \de_1(w) &= \ \big[\,\di_y \de \circ \psi^1_c(w) + \di_z \de \circ \psi^1_c(w) \cdot \di_x \de \circ \psi^1_v(w) \,\big] \cdot \di_y \phi^{-1}(\si_0 w) \\
& \qquad + \di_z \de \circ \psi^1_c(w) \cdot \Big[\, \di_y \de \circ \psi^1_v(w) + \di_z \de \circ \psi^1_v(w) \cdot \frac{d}{dy}\;\de(\si_0 y,\, f^{-1}(\si_0 y),\, 0)\,\Big] \\[0.6em]
\di_z \de_1(w) &= \ \big[\,\di_y \de \circ \psi^1_c(w) + \di_z \de \circ \psi^1_c(w) \cdot \di_x \de \circ \psi^1_v(w) \,\big] \cdot \di_z \phi^{-1}(\si_0 w) \\[0.2em]
& \qquad + \di_z \de \circ \psi^1_c(w) \cdot \di_z \de \circ \psi^1_v(w) 
\end{align*}
\end{lem}

\begin{proof}
\nin Let us calculate the recursive formula of each partial derivatives of Pre\;$ \de_1(w) $. Let us estimate $ \di_x \big(\text{Pre}\,\de_1(w)\big) $. Then 
\begin{align*}
&\quad \ \ {\di_x}\,(\de \circ F \circ H^{-1}(w) - \de(x, f^{-1}(x),0)) 
= \frac{\di }{\di x}\; \de(x, \phi^{-1}(x), \de \circ H^{-1}(w)) - \frac{d}{dx}\, \de (x, f^{-1}(x),0) \\
&= \ \di_x \de \circ (F \circ H^{-1}(w)) + \di_y \de \circ (F \circ H^{-1}(w)) \cdot \di_x \phi^{-1}(w) \\
&\qquad + \di_z \de \circ (F \circ H^{-1}(w)) \cdot \di_x (\de \circ H^{-1}(w)) - \frac{d}{dx}\; \de (x, f^{-1}(x),0) \\[0.5em]
&= \ \di_x \de \circ (F \circ H^{-1}(w)) + \di_y \de \circ (F \circ H^{-1}(w)) \cdot \di_x \phi^{-1}(w) \\
&\qquad + \di_z \de \circ (F \circ H^{-1}(w)) \cdot \di_x \de \circ H^{-1}(w) \cdot \di_x \phi^{-1}(w) - \frac{d}{dx}\; \de (x, f^{-1}(x),0) \\[0.5em]
&= \ \boxed{  \big[ \di_y \de \circ (F \circ H^{-1}(w)) + \di_z \de \circ (F \circ H^{-1}(w)) \cdot \di_x \de \circ H^{-1}(w) \big] } \cdot \di_x \phi^{-1}(w) \\ 
&\qquad + \di_x \de \circ (F \circ H^{-1}(w)) - \frac{d}{dx}\; \de (x, f^{-1}(x),0) . \\[-1em]
\end{align*} 

\nin Let us estimate $ \di_y \big(\text{Pre}\,\de_1(w)) $. Then 
\begin{align*}
&\quad \ \ \di_y (\de \circ F \circ H^{-1}(w) - \de(x, f^{-1}(x),0)) = \frac{\di }{\di y}\; \de(x, \phi^{-1}(x), \de \circ H^{-1}(w)) \\
&= \ \di_y \de \circ (F \circ H^{-1}(w)) \cdot \di_y \phi^{-1}(w) + \di_z \de \circ (F \circ H^{-1}(w)) \cdot  \frac{\di }{\di y}\; (\de \circ H^{-1}(w)) \\[0.5em]
&= \ \di_y \de \circ (F \circ H^{-1}(w)) \cdot \di_y \phi^{-1}(w) + \di_z \de \circ (F \circ H^{-1}(w)) \\
&\qquad \cdot \Big[ \di_x \de \circ H^{-1}(w) \cdot \di_y \phi^{-1}(w) + \di_y \de \circ H^{-1}(w) + \di_z \de \circ H^{-1}(w) \cdot \frac{d}{dy}\; \de (y, f^{-1}(y),0) \Big] \\[0.5em]     \numberthis   \label{di-y delta-1 estimation}
&= \ \boxed{  \big[ \di_y \de \circ (F \circ H^{-1}(w)) + \di_z \de \circ (F \circ H^{-1}(w)) \cdot \di_x \de \circ H^{-1}(w) \big] } \cdot \di_y \phi^{-1}(w) \\
&\qquad + \di_z \de \circ (F \circ H^{-1}(w)) \cdot \Big[\di_y \de \circ H^{-1}(w) + \di_z \de \circ H^{-1}(w) \cdot \frac{d}{dy}\; \de (y, f^{-1}(y),0) \Big] . \\[-1em]
\end{align*}  

\nin Similarly, we can estimate $ \di_z \big(\text{Pre}\,\de_1(w)\big) $. Then 
\begin{align*}
&\quad \ \ \di_z (\de \circ F \circ H^{-1}(w) - \de(x, f^{-1}(x),0)) =  \frac{\di }{\di z}\; \de(x, \phi^{-1}(x), \de \circ H^{-1}(w)) \\ 
&= \ \di_y \de \circ (F \circ H^{-1}(w)) \cdot \di_z \phi^{-1}(w) + \di_z \de \circ (F \circ H^{-1}(w)) \cdot  \frac{\di }{\di z}\; (\de \circ H^{-1}(w)) \\[0.3em]
&= \ \di_y \de \circ (F \circ H^{-1}(w)) \cdot \di_z \phi^{-1}(w) \\ &\qquad
 + \di_z \de \circ (F \circ H^{-1}(w))
\cdot \big[\, \di_x \de \circ H^{-1}(w) \cdot \di_z \phi^{-1}(w) + \di_z \de \circ H^{-1}(w) \big] \\[0.5em]   \numberthis   \label{di-z delta-1 estimation}
&= \ \boxed{  \big[\, \di_y \de \circ (F \circ H^{-1}(w)) + \di_z \de \circ (F \circ H^{-1}(w)) \cdot \di_x \de \circ H^{-1}(w) \big] } \cdot \di_z \phi^{-1}(w) \\ &\qquad  
+ \di_z \de \circ (F \circ H^{-1}(w)) \cdot \di_z \de \circ H^{-1}(w) .  \\[-1em]
\end{align*}  
Since $ \psi^1_v(w) = H^{-1}(\si_0 w) $ and $ \psi^1_c(w) = F \circ H^{-1}(\si_0 w) $, the fact that $ \di_t \de_1(w) = \di_t \big(\text{Pre}\,\de_1(\si_0 w)\big) $ for $ t = x,y,z $ implies the equations in Lemma. The proof is complete.
\end{proof}
\msk
\nin Let us estimate $ \di_y \big(\text{Pre}\,\eps_1(w)\big) $. Thus we need to estimate 
\msk
\begin{align*}
\di_y(\eps \circ F \circ H^{-1}(w)) \  \text{and} \ \ \di_y(\eps \circ F^2 \circ H^{-1}(w)) . \\[-1em]
\end{align*}
firstly. Let us estimate $ \di_y(\eps \circ F \circ H^{-1}(w)) $. Then
\begin{align*}
&\quad \ \ \di_y(\eps \circ F \circ H^{-1}(w)) = \frac{\di }{\di y}\; \eps (x, \phi^{-1}(x), \de \circ H^{-1}(w)) \\
&= \ \di_y \eps \circ (F \circ H^{-1}(w)) \cdot \di_y \phi^{-1}(w) + \di_z \eps \circ (F \circ H^{-1}(w)) \cdot \frac{\di }{\di y}\; (\de \circ H^{-1}(w)) \\[0.5em]
&= \ \di_y \eps \circ (F \circ H^{-1}(w)) \cdot \di_y \phi^{-1}(w) + \di_z \eps \circ (F \circ H^{-1}(w)) \\
&\qquad \cdot \Big[ \di_x \de \circ H^{-1}(w) \cdot \di_y \phi^{-1}(w) + \di_y \de \circ H^{-1}(w) + \di_z \de \circ H^{-1}(w) \cdot \frac{d}{dy}\; \de (y, f^{-1}(y),0) \Big] \\[0.5em]
&= \ \big[ \di_y \eps \circ (F \circ H^{-1}(w)) + \di_z \eps \circ (F \circ H^{-1}(w)) \cdot \di_x \de \circ H^{-1}(w) \big] \cdot \di_y \phi^{-1}(w) \\
&\qquad + \di_z \eps \circ (F \circ H^{-1}(w)) \cdot \Big[\di_y \de \circ H^{-1}(w) + \di_z \de \circ H^{-1}(w) \cdot \frac{d}{dy}\; \de (y, f^{-1}(y),0) \Big] .
\end{align*} 

\nin Moreover, we have 
\begin{align*}
&\qquad \ \di_y(\eps \circ F^2 \circ H^{-1}(w)) = \frac{\di }{\di y}\; \eps (f(x) - \eps \circ F \circ H^{-1}(w),\,x,\, \de \circ F \circ H^{-1}(w)) \\[0.3em]
&= \ -\, \di_x \eps \circ (F^2 \circ H^{-1}(w)) \cdot \frac{\di }{\di y}\;(\eps \circ F \circ H^{-1}(w)) 
+ \di_z \eps \circ (F^2 \circ H^{-1}(w)) \cdot \frac{\di }{\di y}\; (\de \circ F \circ H^{-1}(w)) . \\[-1em]
\end{align*}

\nin The map $ f'(f_{\eps}(x)) $ denote the function $ f'(f(x) - \eps \circ F \circ H^{-1}(w) - \di_x \eps \circ (F^2 \circ H^{-1}(w))) $. Then $ \di_y \big(\text{Pre}\,\eps_1(w)\big) $ can be expressed in terms of partial derivatives of $ \eps(w) $ and $ \de(w) $ as follows 
%
\begin{align*}
& \quad \ \  \di_y \text{Pre}\,\eps_1(w) = - \di_y \big[\,f(f(x) - \eps \circ F \circ H^{-1}(w)) - \eps \circ F^2 \circ H^{-1}(w) \big] \\[0.5em]
&= \ f'(f(x) - \eps \circ F \circ H^{-1}(w)) \cdot \di_y (\eps \circ F \circ H^{-1}(w)) + \di_y(\eps \circ F^2 \circ H^{-1}(w)) \\[0.5em]
&= \ \big[\, f'(f(x) - \eps \circ F \circ H^{-1}(w)) - \di_x \eps \circ (F^2 \circ H^{-1}(w)) \big]\cdot \di_y (\eps \circ F \circ H^{-1}(w)) \\[0.2em] 
&\qquad + \di_z\eps \circ (F^2 \circ H^{-1}(w)) \cdot \di_y (\de \circ F \circ H^{-1}(w)) \\  \numberthis  \label{di-y eps-1 estimation}
&= \ \Big[\, f'(f_{\eps}(x)) \cdot \{ \di_y \eps \circ (F \circ H^{-1}(w)) + \di_z \eps \circ (F \circ H^{-1}(w)) \cdot \di_x \de \circ H^{-1}(w) \} \\[-0.1em]
& \qquad + \di_z \eps \circ (F^2 \circ H^{-1}(w)) \cdot  \{ \; \boxed { \di_y \de \circ (F \circ H^{-1}(w)) + \di_z \de \circ (F \circ H^{-1}(w)) \cdot \di_x \de \circ H^{-1}(w) } \ \}    \Big] \\
& \qquad \quad  \cdot \di_y \phi^{-1}(w) \\
& \quad + \Big[\, f'(f_{\eps}(x)) \cdot \di_z \eps \circ (F \circ H^{-1}(w)) + \di_z \eps \circ (F^2 \circ H^{-1}(w)) \cdot \di_z \de \circ (F \circ H^{-1}(w)) \Big] \\
&\qquad \quad \cdot \Big[ \di_y \de \circ H^{-1}(w) + \di_z \de \circ H^{-1}(w) \cdot \frac{d}{dy}\; \de (y, f^{-1}(y),0) \Big] .
\end{align*} 

\nin Let us estimate $ \di_z(\eps \circ F \circ H^{-1}(w)) $. Then 
\begin{align*}
&\quad \ \   \di_z(\eps \circ F \circ H^{-1}(w)) =  \frac{\di }{\di z}\;\eps (x, \phi^{-1}(x), \de \circ H^{-1}(w)) \\[-0.2em]
&= \ \di_y \eps \circ (F \circ H^{-1}(w)) \cdot \di_z \phi^{-1}(w) + \di_z \eps \circ (F \circ H^{-1}(w)) \cdot  \frac{\di }{\di z}\;(\de \circ H^{-1}(w)) \\[0.1em]
&= \ \di_y \eps \circ (F \circ H^{-1}(w)) \cdot \di_z \phi^{-1}(w) \\
&\qquad  + \di_z \eps \circ (F \circ H^{-1}(w)) \cdot \big[\, \di_x \de \circ H^{-1}(w) \cdot \di_z \phi^{-1}(w) + \di_z \de \circ H^{-1}(w)   \big] \\[0.2em]
&= \ \big[\, \di_y \eps \circ (F \circ H^{-1}(w)) + \di_z \eps \circ (F \circ H^{-1}(w)) \cdot \di_x \de \circ H^{-1}(w) \big] \cdot \di_z \phi^{-1}(w) \\
&\qquad + \di_z \eps \circ (F \circ H^{-1}(w)) \cdot  \di_z \de \circ H^{-1}(w) . \\[-1em]
\end{align*}  

\nin Moreover, we can express $ \di_y(\eps \circ F^2 \circ H^{-1}(w)) $ in terms of $ \di_y(\eps \circ F \circ H^{-1}(w)) $ and $ \di_y (\de \circ F \circ H^{-1}(w)) $ as follows
\begin{align*}
&\qquad  \di_z(\eps \circ F^2 \circ H^{-1}(w)) =  \frac{\di }{\di z}\; \eps (f(x) - \eps \circ F \circ H^{-1}(w),\,x,\, \de \circ F \circ H^{-1}(w)) \\[0.3em]
&= - \di_x \eps \circ (F^2 \circ H^{-1}(w)) \cdot  \frac{\di }{\di z}\;(\eps \circ F \circ H^{-1}(w)) 
+ \di_z \eps \circ (F^2 \circ H^{-1}(w)) \cdot  \frac{\di }{\di z}\;(\de \circ F \circ H^{-1}(w)) .
\end{align*}  

\nin Then $ \di_z \big(\text{Pre}\,\eps_1(w)\big) $ can be estimated in terms of partial derivatives of $ \eps(w) $ and $ \de(w) $
\begin{align*}
& \quad \ \ \di_z \text{Pre}\,\eps_1(w) = - \di_z \big[\,f(f(x) - \eps \circ F \circ H^{-1}(w)) - \eps \circ F^2 \circ H^{-1}(w) \big] \\[0.5em]
&= \ f'(f(x) - \eps \circ F \circ H^{-1}(w)) \cdot \di_z (\eps \circ F \circ H^{-1}(w)) + \di_z(\eps \circ F^2 \circ H^{-1}(w)) \\[0.5em]
&= \ \big[\, f'(f(x) - \eps \circ F \circ H^{-1}(w)) - \di_x \eps \circ (F^2 \circ H^{-1}(w)) \big]\cdot \di_z (\eps \circ F \circ H^{-1}(w)) \\ 
&\qquad + \di_z\eps \circ (F^2 \circ H^{-1}(w)) \cdot \di_z (\de \circ F \circ H^{-1}(w)) \\  \numberthis  \label{di-z eps-1 estimation}
&= \ \Big[\, f'(f_{\eps}(x)) \cdot \{ \di_y \eps \circ (F \circ H^{-1}(w)) + \di_z \eps \circ (F \circ H^{-1}(w)) \cdot \di_x \de \circ H^{-1}(w) \} \\
& \qquad + \di_z \eps \circ (F^2 \circ H^{-1}(w)) \cdot \{ \; \boxed{ \di_y \de \circ (F \circ H^{-1}(w)) + \di_z \de \circ (F \circ H^{-1}(w)) \cdot \di_x \de \circ H^{-1}(w) } \  \} \Big] \\
& \qquad  \cdot \di_z \phi^{-1}(w) \\
& \quad + \Big[\, f'(f_{\eps}(x)) \cdot \di_z \eps \circ (F \circ H^{-1}(w)) + \di_z \eps \circ (F^2 \circ H^{-1}(w)) \cdot \di_z \de \circ (F \circ H^{-1}(w)) \Big] \\
&\qquad \quad \cdot  \di_z \de \circ H^{-1}(w) .
\end{align*}
\ssk
\begin{lem}
Let $ F $ be an infinitely renormalizable three dimensional H\'enon-like map. Then 
\begin{align*}
\Jac R^nF(w) 
&= \ (f_{n-1}^{-1})'(\si_{n-1} x) \cdot f'_{n-1}(f_{n-1}(\si_{n-1} x)) \\[0.2em]
&\qquad \cdot \Jac R^{n-1}F \circ  ( H_{n-1}^{-1} (\si_{n-1} w)) \cdot \Jac R^{n-1}F \circ (F_{n-1} \circ H_{n-1}^{-1} (\si_{n-1} w)) .
\end{align*}
\end{lem}
\begin{proof}
Let us calculate $ \Jac RF(w) $ in terms of partial derivatives of $ \eps $ and $ \de $. Recall the equations \eqref{di-y delta-1 estimation}, \eqref{di-z delta-1 estimation}, \eqref{di-y eps-1 estimation} and \eqref{di-z eps-1 estimation}. Let us express $ \Jac RF $ in terms of these expressions \ssk
\begin{align*}
& \Jac RF(w) = \di_y \eps_1(w) \cdot \di_z \de_1(w) - \di_z \eps_1(w) \cdot \di_y \de_1(w) \\ 
&= \ \Big[\, \big\{\, f'(f_{\eps}(\si_0 x)) \cdot \{\, \di_y \eps \circ (F \circ H^{-1}(\si_0 w)) + \di_z \eps \circ (F \circ H^{-1}(\si_0 w)) \cdot \di_x \de \circ H^{-1}(\si_0 w) \} \\
& \quad + \di_z \eps \circ (F^2 \circ H^{-1}(\si_0 w)) \cdot \\
& \qquad   \{ \; \di_y \de \circ (F \circ H^{-1}(\si_0 w)) + \di_z \de \circ (F \circ H^{-1}(\si_0 w)) \cdot \di_x \de \circ H^{-1}(\si_0 w)  \, \}  \big\} 
\cdot \di_y \phi^{-1}(\si_0 w) \\[0.4em] 
& \quad 
+ \big\{\, f'(f_{\eps}(\si_0 x)) \cdot \di_z \eps \circ (F \circ H^{-1}(\si_0 w)) + \di_z \eps \circ (F^2 \circ H^{-1}(\si_0 w)) \cdot \di_z \de \circ (F \circ H^{-1}(\si_0 w)) \big\} \\
&\qquad \quad \cdot \big\{ \di_y \de \circ H^{-1}(\si_0 w) + \di_z \de \circ H^{-1}(\si_0 w) \cdot \frac{d}{dy}\; \de (\si_0 y, f^{-1}(\si_0 y),0) \big\} \,\Big] \\
& \quad \cdot  \  \Big[\, \big\{\, \di_y \de \circ (F \circ H^{-1}(\si_0 w)) + \di_z \de \circ (F \circ H^{-1}(\si_0 w)) \cdot \di_x \de \circ H^{-1}(\si_0 w) \big\}  \cdot \di_z \phi^{-1}(\si_0 w) \\ &\qquad 
+ \di_z \de \circ (F \circ H^{-1}(\si_0 w)) \cdot \di_z \de \circ H^{-1}(\si_0 w)\, \Big] \\
& \ - \Big[\,\big \{\, f'(f_{\eps}(\si_0 x)) \cdot \{\, \di_y \eps \circ (F \circ H^{-1}(\si_0 w)) + \di_z \eps \circ (F \circ H^{-1}(\si_0 w)) \cdot \di_x \de \circ H^{-1}(\si_0 w) \} \\
& \qquad + \di_z \eps \circ (F^2 \circ H^{-1}(\si_0 w)) \\
& \qquad \cdot \{ \;  \di_y \de \circ (F \circ H^{-1}(\si_0 w)) + \di_z \de \circ (F \circ H^{-1}(\si_0 w)) \cdot \di_x \de \circ H^{-1}(\si_0 w) \,  \} \big\} \cdot \di_z \phi^{-1}(\si_0 w) \\[0.4em] 
& \quad + \big\{\, f'(f_{\eps}(\si_0 x)) \cdot \di_z \eps \circ (F \circ H^{-1}(\si_0 w)) + \di_z \eps \circ (F^2 \circ H^{-1}(\si_0 w)) \cdot \di_z \de \circ (F \circ H^{-1}(\si_0 w)) \big\} \\
&\qquad \quad \cdot  \di_z \de \circ H^{-1}(\si_0 w)  \Big] \\
& \quad \cdot \Big[ \, \big\{ \di_y \de \circ (F \circ H^{-1}(\si_0 w)) + \di_z \de \circ (F \circ H^{-1}(\si_0 w)) \cdot \di_x \de \circ H^{-1}(\si_0 w) \big\}  \cdot \di_y \phi^{-1}(\si_0 w) \\
&\quad \ \ + \di_z \de \circ (F \circ H^{-1}(\si_0 w)) \cdot \big\{\,\di_y \de \circ H^{-1}(\si_0 w) + \di_z \de \circ H^{-1}(\si_0 w) \cdot \frac{d}{dy}\; \de (\si_0 y, f^{-1}(\si_0 y),0) \big\} \Big]  \\[-1em]
\end{align*}
On the above equation, let us denote some factors to be $ A $, $ B $, $ C $ and $ D $ as follows \ssk
\begin{align*}
A &= \ f'(f_{\eps}(\si_0 x)) \cdot \{\, \di_y \eps \circ (F \circ H^{-1}(\si_0 w)) + \di_z \eps \circ (F \circ H^{-1}(\si_0 w)) \cdot \di_x \de \circ H^{-1}(\si_0 w) \} \\
& \qquad + \di_z \eps \circ (F^2 \circ H^{-1}(\si_0 w)) \\
& \qquad \ \ \cdot \{ \; \di_y \de \circ (F \circ H^{-1}(\si_0 w)) + \di_z \de \circ (F \circ H^{-1}(\si_0 w)) \cdot \di_x \de \circ H^{-1}(\si_0 w)  \, \} \\[0.4em]
B &= \ f'(f_{\eps}(\si x)) \cdot \di_z \eps \circ (F \circ H^{-1}(\si_0 w)) + \di_z \eps \circ (F^2 \circ H^{-1}(\si_0 w)) \cdot \di_z \de \circ (F \circ H^{-1}(\si_0 w)) \\[0.4em]
C &= \ \di_y \de \circ (F \circ H^{-1}(\si_0 w)) + \di_z \de \circ (F \circ H^{-1}(\si_0 w)) \cdot \di_x \de  \circ H^{-1}(\si_0 w) \\
D &= \ \di_y \de \circ H^{-1}(\si_0 w) + \di_z \de \circ H^{-1}(\si_0 w) \cdot \frac{d}{dy}\; \de (\si_0 y, f^{-1}(\si_0 y),0) . \\[-1em]
\end{align*}
Let us calculate $  A \cdot \di_z \de \circ (F \circ H^{-1}(\si w)) - BC $ for later use \ssk
\begin{align*}
& \ A \cdot \di_z \de \circ (F \circ H^{-1}(\si_0 w)) - BC \\
&= \ \Big[\, f'(f_{\eps}(\si_0 x)) \cdot \{\, \di_y \eps \circ (F \circ H^{-1}(\si_0 w)) + \di_z \eps \circ (F \circ H^{-1}(\si_0 w)) \cdot \di_x \de \circ H^{-1}(\si_0 w) \} \\
& \quad + \di_z \eps \circ (F^2 \circ H^{-1}(\si_0 w)) \\
& \quad \ \cdot  \{ \; \di_y \de \circ (F \circ H^{-1}(\si_0 w)) + \di_z \de \circ (F \circ H^{-1}(\si_0 w)) \cdot \di_x \de \circ H^{-1}(\si_0 w)  \, \} \, \Big] \\
& \quad \ \cdot \di_z \de \circ (F \circ H^{-1}(\si_0 w)) \\
& \quad - \Big[\,f'(f_{\eps}(\si_0 x)) \cdot \di_z \eps \circ (F \circ H^{-1}(\si_0 w)) + \di_z \eps \circ (F^2 \circ H^{-1}(\si_0 w)) \cdot \di_z \de \circ (F \circ H^{-1}(\si_0 w))\, \Big] \\
&\quad \ \cdot \Big[ \, \di_y \de \circ (F \circ H^{-1}(\si_0 w)) + \di_z \de \circ (F \circ H^{-1}(\si_0 w)) \cdot \di_x \de  \circ H^{-1}(\si_0 w) \, \Big] \\[0.5em]   \numberthis  \label{A di-z - BC, calculation}
&= \  f'(f_{\eps}(\si_0 x)) \cdot \big[\, \di_y \eps \circ (F \circ H^{-1}(\si_0 w)) \cdot \di_z \de \circ (F \circ H^{-1}(\si_0 w)) \\[0.2em]
 &\qquad \qquad \qquad \qquad - \di_z \eps \circ (F \circ H^{-1}(\si_0 w)) \cdot \di_y \de \circ (F \circ H^{-1}(\si_0 w))\,\big] . \\[-1em]
\end{align*}  

\nin Then the above equation of $ \Jac RF $ is expressed as follows \ssk
\begin{align*}
&\di_y \eps_1(w) \cdot \di_z \de_1(w) - \di_z \eps_1(w) \cdot \di_y \de_1(w) \\
&= \ \Big[\, A \cdot \di_y \phi^{-1}(\si_0 w) + B D \,\Big] 
\cdot \Big[\, C \cdot \di_z \phi^{-1}(\si_0 w) + \di_z \de \circ (F \circ H^{-1}(\si_0 w)) \cdot \di_z \de \circ H^{-1}(\si_0 w)\, \Big] \\
& \quad - \Big[\, A \cdot \di_z \phi^{-1}(\si_0 w) + B \cdot \di_z \de \circ H^{-1}(\si_0 w)  \Big] 
 \cdot \Big[\, C \cdot \di_y \phi^{-1}(\si_0 w) + \di_z \de \circ (F \circ H^{-1}(\si_0 w)) \cdot D \Big] \\
&= \ A \cdot \di_y \phi^{-1}(\si_0 w) \cdot \di_z \de \circ (F \circ H^{-1}(\si_0 w)) \cdot \di_z \de \circ H^{-1}(\si_0 w) + BCD \cdot \di_z \phi^{-1}(\si_0 w) \\
&\quad - \big[\, AD \cdot \di_z \phi^{-1}(\si_0 w) \cdot \di_z \de \circ (F \circ H^{-1}(\si_0 w)) + BC \cdot \di_z \de \circ H^{-1}(\si_0 w) \cdot \di_y \phi^{-1}(\si_0 w) \, \big] \\
%
%
&= \ A \cdot (f^{-1}_{\eps})'(\si_0 x)\cdot \big\{\,\di_y \eps \circ H^{-1}(\si_0 w) + \di_z \eps \circ H^{-1}(\si_0 w)\cdot \frac{d}{dy}\;\de(\si_0 y, f^{-1}(\si_0 y),0)\, \big\} \phantom{******\;}\\
&\qquad \cdot \di_z \de \circ (F \circ H^{-1}(\si_0 w)) \cdot \di_z \de \circ H^{-1}(\si_0 w) \\
&\quad \ + BC \cdot \big\{ \,\di_y \de \circ H^{-1}(\si_0 w) + \di_z \de \circ H^{-1}(\si_0 w) \cdot \frac{d}{dy}\; \de (\si_0 y, f^{-1}(\si_0 y),0) \big\} \\
&\qquad \cdot (f^{-1}_{\eps})'(\si_0 x) \cdot \di_z \eps \circ H^{-1}(\si_0 w) \\
&\quad - \Big[\, A \cdot \big\{ \di_y \de \circ H^{-1}(\si_0 w) + \di_z \de \circ H^{-1}(\si_0 w) \cdot \frac{d}{dy}\; \de (\si_0 y, f^{-1}(\si_0 y),0) \big\} \cdot (f^{-1}_{\eps})'(\si_0 x) \\
&\qquad \cdot \di_z \eps \circ H^{-1}(\si_0 w) \cdot \di_z \de \circ (F \circ H^{-1}(\si_0 w)) 
 + BC \cdot \di_z \de \circ H^{-1}(\si_0 w) \cdot (f^{-1}_{\eps})'(\si_0 x) \\ 
&\qquad \cdot \big\{\,\di_y \eps \circ H^{-1}(\si_0 w) + \di_z \eps \circ H^{-1}(\si_0 w)\cdot \frac{d}{dy}\;\de(\si_0 y, f^{-1}(\si_0 y),0)\, \big\}\, \Big] \\
&= \ A \cdot (f^{-1}_{\eps})'(\si_0 x) \cdot \di_z \de \circ (F \circ H^{-1}(\si_0 w)) \\
&\qquad \cdot \big[ \, \di_y \eps \circ H^{-1}(\si_0 w) \cdot \di_z \de \circ H^{-1}(\si_0 w) - \di_z \eps \circ H^{-1}(\si_0 w) \cdot \di_y \de \circ H^{-1}(\si_0 w)\, \big] \\
&\quad - BC \cdot (f^{-1}_{\eps})'(\si_0 x) \\
&\qquad \cdot \big[ \, \di_y \eps \circ H^{-1}(\si_0 w) \cdot \di_z \de \circ H^{-1}(\si_0 w) - \di_z \eps \circ H^{-1}(\si_0 w) \cdot \di_y \de \circ H^{-1}(\si_0 w) \, \big] \\[0.3em]
&= \ \big[\, A \cdot \di_z \de \circ (F \circ H^{-1}(\si_0 w)) - BC \,\big] \cdot (f^{-1}_{\eps})'(\si_0 x) \\
&\qquad \cdot \big[ \, \di_y \eps \circ H^{-1}(\si_0 w) \cdot \di_z \de \circ H^{-1}(\si_0 w) - \di_z \eps \circ H^{-1}(\si_0 w) \cdot \di_y \de \circ H^{-1}(\si_0 w) \, \big] . \\[-1em]
\end{align*}  

\nin By the equation \eqref{A di-z - BC, calculation}, the above equation is continued as follows \ssk
\begin{align*}
&= \ (f^{-1}_{\eps})'(\si_0 x) \cdot \big[ \, \di_y \eps \circ H^{-1}(\si_0 w) \cdot \di_z \de \circ H^{-1}(\si_0 w) - \di_z \eps \circ H^{-1}(\si_0 w) \cdot \di_y \de \circ H^{-1}(\si_0 w) \, \big] \phantom{**} \\[0.2em]
&\qquad \cdot f'(f_{\eps}(\si_0 x)) \cdot \big[\, \di_y \eps \circ (F \circ H^{-1}(\si_0 w)) \cdot \di_z \de \circ (F \circ H^{-1}(\si_0 w)) \\
 &\qquad \qquad \qquad \qquad - \di_z \eps \circ (F \circ H^{-1}(\si_0 w)) \cdot \di_y \de \circ (F \circ H^{-1}(\si_0 w))\,\big] \\
&= \ f'(f_{\eps}(\si_0 x)) \cdot (f^{-1}_{\eps})'(\si_0 x) \cdot \Jac F \circ (H^{-1}(\si_0 w)) \cdot \Jac F \circ (F \circ H^{-1}(\si_0 w)) . \\[-1em]
\end{align*}  
Similarly, $ \Jac R^nF(w) $ is expressed in terms of the partial derivatives of $ \eps_{n-1} $ and $ \de_{n-1} $ by induction \ssk
\begin{align*}
& \ \Jac R^nF(w) = \ (f^{-1}_{n-1,\;\eps})'(\si_{n-1} x) \cdot f_{n-1}'(f_{n-1,\;\eps}(\si_{n-1} x)) \\
&\qquad \qquad \qquad \qquad \cdot \Jac F_{n-1} \circ (H_{n-1}^{-1}(\si_{n-1} w)) \cdot \Jac F_{n-1} \circ (F_{n-1} \circ H_{n-1}^{-1}(\si_{n-1} w)) .
\end{align*}
\end{proof}

\bsk

\section{$ C^1 $ conjugation of H\'enon-like map $ F \in \NN $}

\begin{lem}
Let $ F $ and $ \widetilde F $ be H\'enon-like maps. Suppose that $ C^1 $ diffeomorphism $ \Phi \colon B \ra B $ is a conjugation between $ F $ and $ \widetilde F $ where $ \Phi(w) = (x,\ y,\ \varphi(y,z)) $. Then $ F \in \NN $ if and only if $ \widetilde F \in \NN $.
\end{lem}
\begin{proof}
The coordinates of $ F $ and $ \widetilde F $ are as follows
\msk
\begin{equation}
\begin{aligned}
F(x,y,z) &= (f(x) - \eps(w),\ x,\ \de(w)) \\[0.2em]
\widetilde F &= (\widetilde f(x) - \widetilde \eps(w),\ x,\ \widetilde \de(w)) .
\end{aligned} \msk
\end{equation}
\nin We may assume that $ \Phi \circ F = \widetilde F \circ \Phi $. Thus by chain rule 
$$ D\Phi \circ F(w) \cdot DF(w) = D\widetilde F \circ \Phi(w) \cdot D\Phi (w) . $$
Let us consider the map $ \varphi(w) = \varphi(x,y,z) $ rather than $ \varphi(y,z) $. Then
\msk
\begin{align*}
& \begin{pmatrix}
1& 0& 0 \\[0.2em]
0& 1& 0 \\[0.2em]
\di_x \varphi & \di_y \varphi & \di_z \varphi 
\end{pmatrix} \cdot
\begin{pmatrix}
f'(x) - \di_x \eps & -\di_y \eps & -\di_z \eps \\[0.2em]
1& 0& 0 \\[0.2em]
\di_x \de & \di_y \de & \di_z \de
\end{pmatrix} \\[0.8em]
& \qquad \qquad \qquad \qquad \qquad \qquad =
\begin{pmatrix}
\widetilde f'(x) - \di_x \widetilde \eps & -\di_y \widetilde \eps & -\di_z \widetilde \eps \\[0.2em]
1& 0& 0 \\[0.2em]
\di_x \widetilde \de & \di_y \widetilde \de & \di_z \widetilde \de
\end{pmatrix} \cdot
\begin{pmatrix}
1& 0& 0 \\[0.2em]
0& 1& 0 \\[0.2em]
\di_x \varphi & \di_y \varphi & \di_z \varphi 
\end{pmatrix} . \\[-1em]
\end{align*} 
Then we obtain \msk
\begin{align*}
 \di_x \varphi \circ F(w) \cdot \{ \,f'(x) - \di_x \eps(w)\, \} + \di_y \varphi \circ F(& w) + \di_z \varphi \circ F(w) \cdot \di_x \de(w) \\
& = \ \di_x \widetilde \de \circ \Phi(w) + \di_z \widetilde \de \circ \Phi(w) \cdot \di_x \varphi(w) \\[0.2em]
- \di_x \varphi \circ F(w) \cdot  \di_y \eps(w) + \di_z \varphi \circ F(w) \cdot \di_y \de(w) &= \ \di_y \widetilde \de \circ \Phi(w) + \di_z \widetilde \de \circ \Phi(w) \cdot \di_y \varphi(w) \\[0.2em] 
- \di_x \varphi \circ F(w) \cdot  \di_z \eps(w) + \di_z \varphi \circ F(w) \cdot \di_z \de(w) &= \ \di_z \widetilde \de \circ \Phi(w) \cdot \di_z \varphi(w) . \\[-1em]
\end{align*}  
Recall that $ F \in \NN $ means that $ \di_y \de \circ F(w) + \di_z \de \circ F(w) \cdot \di_x \de(w) \equiv 0 $. Then
\msk
\begin{align*}
& \quad \ \big[\, \di_y \widetilde \de \circ (\Phi \circ F)(w) + \di_z \widetilde \de \circ (\Phi \circ F)(w) \cdot \di_x \widetilde \de \circ \Phi(w) \, \big] \cdot \di_z \varphi \circ F(w) \cdot \di_z \varphi(w) \\[0.4em]
&= \ \di_z \varphi \circ F^2(w) \cdot \di_z \varphi \circ F(w) \cdot \di_z \varphi(w) \cdot \big[\, \di_y \de \circ F(w) + \di_z \de \circ F(w) \cdot \di_x \de(w)\,\big] \\[0.2em]
& \qquad - \di_x \varphi \circ F^2(w) \cdot \di_y \eps \circ F(w) \cdot \di_z \varphi \circ F(w) \cdot \di_z \varphi(w) \\[0.2em]
& \qquad - \di_x \varphi \circ F^2(w) \cdot \di_z \eps \circ F(w) \cdot \di_x \varphi \circ F(w) \cdot \di_z \varphi(w) \cdot \{ \,f'(x) - \di_x \eps(w) \,\} \\[0.2em]
& \qquad - \di_x \varphi \circ F^2(w) \cdot \di_z \eps \circ F(w) \cdot \di_z \varphi \circ F(w) \cdot \di_z \varphi(w) \cdot \di_x \de(w) \\[0.2em]
& \qquad + \di_x \varphi \circ F^2(w) \cdot \di_z \eps \circ F(w) \cdot \big[ -\di_x \varphi \circ F(w) \cdot \di_z \eps(w) + \di_z \varphi \circ F(w) \cdot \di_z \de(w) \,\big] \cdot \di_x \varphi(w) \\[0.2em]
& \qquad + \di_z \varphi \circ F^2(w) \cdot \di_z \de \circ F(w) \cdot \di_x \varphi \circ F(w) \cdot \di_z \varphi(w) \cdot \{ \,f'(x) - \di_x \eps(w) \,\}\\[0.2em]
& \qquad - \di_z \varphi \circ F^2(w) \cdot \di_z \de \circ F(w) \cdot \big[ -\di_x \varphi \circ F(w) \cdot \di_z \eps(w) + \di_z \varphi \circ F(w) \cdot \di_z \de(w) \,\big] \cdot \di_x \varphi(w) .  \\[-1em]
\end{align*}  
The detailed calculation to obtain above equation is left to the reader. The fact that $ \di_x \varphi(w) \equiv 0 $ implies that
\msk
\begin{align*}
& \quad \ \big[\, \di_y \widetilde \de \circ (\Phi \circ F)(w) + \di_z \widetilde \de \circ (\Phi \circ F)(w) \cdot \di_x \widetilde \de \circ \Phi(w) \, \big] \cdot \di_z \varphi \circ F(w) \cdot \di_z \varphi(w) \\[0.3em]
&= \ \di_z \varphi \circ F^2(w) \cdot \di_z \varphi \circ F(w) \cdot \di_z \varphi(w) \cdot \big[\, \di_y \de \circ F(w) + \di_z \de \circ F(w) \cdot \di_x \de(w)\,\big] .
\end{align*}  
Since the map $ \Phi(w) = (x,y, \varphi(y,z)) $ is a diffeomorphism, $ \di_z \varphi(y,z) $ is never zero for all $ w \in B $. Thus we have the following equation
\msk
\begin{equation*}
\begin{aligned}
& \di_y \widetilde \de \circ (\Phi \circ F)(w) + \di_z \widetilde \de \circ (\Phi \circ F)(w) \cdot \di_x \widetilde \de \circ \Phi(w)  \\[0.3em] 
& \qquad \qquad =  \ \di_z \varphi \circ F^2(w) \cdot \big[\, \di_y \de \circ F(w) + \di_z \de \circ F(w) \cdot \di_x \de(w)\,\big] .
\end{aligned}
\end{equation*}
Hence, $ F \in \NN $ if and only if $ \widetilde F \in \NN $. 

\end{proof}

\end{appendices}

\bsk


\bibliographystyle{alpha}

\end{document}